\documentclass[a4paper,11pt]{amsart}
\usepackage{amssymb,mathrsfs}
\usepackage{amsthm}
\usepackage{amsmath}
\usepackage{latexsym}
\usepackage{fancyhdr}
\usepackage{ascmac}
\usepackage{color}
\usepackage[all]{xy}
\usepackage{amscd}

\theoremstyle{plain}
\newtheorem{thm}{Theorem}[section]
\newtheorem{prop}[thm]{Proposition}
\newtheorem{cor}[thm]{Corollary}
\newtheorem{lem}[thm]{Lemma}
\newtheorem{dfn}[thm]{Definition}

\newtheorem{rmk}[thm]{Remark}
\makeatletter
 
  \@addtoreset{equation}{section}
\makeatother

\newcommand{\bQ}{\overline{\mathbb{Q}}}

\newcommand{\bF}{\overline{\mathbb{F}}}

\newcommand{\C}{\mathbb{C}}
\newcommand{\R}{\mathbb{R}}
\newcommand{\Q}{\mathbb{Q}}

\newcommand{\Z}{\mathbb{Z}}

\newcommand{\N}{\mathbb{N}}
\newcommand{\F}{\mathbb{F}}

\newcommand{\lra}{\longrightarrow}
\newcommand{\lla}{\longleftarrow}

\newcommand{\E}{\mathcal{E}}

\newcommand{\A}{\mathbb{A}}
\newcommand{\D}{\mathbb{D}}
\newcommand{\vp}{\varphi}
\newcommand{\mA}{\mathcal{A}}
\renewcommand{\O}{\mathcal{O}}

\newcommand{\br}{\overline{\rho}}

\newcommand{\diag}{{\rm diag}}
\newcommand{\ds}{\displaystyle}

\newcommand{\X}{\mathcal{X}}
\newcommand{\G}{\Gamma}

\newcommand{\bS}{\overline{S}}

\newcommand{\uk}{\underline{k}}
\newcommand{\w}{\omega}

\newcommand{\e}{\varepsilon}

\title[The weight reduction of mod $p$ Siegel modular forms for $GSp_4$ and 
theta operators]
{The weight reduction of mod $p$ Siegel modular forms for $GSp_4$ and 
theta operators}
\author{Takuya Yamauchi}
\thanks{The author
is partially supported by JSPS KAKENHI Grant Number (B) No.19H01778}

\address{Takuya Yamauchi \\ 
Mathematical Inst. Tohoku Univ.\\
 6-3,Aoba, Aramaki, Aoba-Ku, Sendai 980-8578, JAPAN}
\email{takuya.yamauchi.c3@tohoku.ac.jp}

\begin{document}

\maketitle

\begin{abstract}
In this paper we investigate the (classical) weights of mod $p$ Siegel modular forms of degree 2 toward 
studying Serre's conjecture for $GSp_4/\Q$. 
We first construct various theta operators on the space of such forms 
\`a la Katz and define the theta cycles for the specific theta operators. Secondly, we study the partial Hasse invariants on 
each Ekedahl-Oort stratum and their local behaviors.    
This enables us to obtain a kind of weight reduction theorem for mod $p$ Siegel modular 
forms of degree 2 without increasing the level.  
\end{abstract}

\tableofcontents

\section{Introduction}
Let $f$ be an elliptic cuspidal eigenform  of level $N$ and weight $k$ with character $\e$. 
It is well-known that for each prime $p$, one can associate to $f$ a mod $p$ Galois representation $\br_{f,p}:
G_\Q:={\rm Gal}(\bQ/\Q)\lra {\rm GL}_2(\bF_p)$. 
On the other hand, for any odd, irreducible Galois representation $\br:G_\Q\lra {\rm GL}_2(\bF_p)$, Serre \cite{serre} defined three invariants  $(k(\br),N(\br),\e(\br))$, 
called the weight, level, and a character, in this order. 
The first two invariants are positive integers and 
the last is a finite character of $G_\Q$ unramified outside $N(\br)$. 
In that definition the most difficult one is the weight $k(\br)$ and others are easy to define from $\br$. 
We say such a representation $\br$ is modular if there exists an elliptic 
cuspidal eigenform $f$ such that $\br\sim \br_{f,p}$. 
If so, there exist other cuspidal eigenforms with the same property. 
It is important to specify a minimal choice  of 
the weight (also the level and the character) among the candidates. Serre conjectured 
that if $\br$ is modular, then one can find 
an $f$ with weight $k(\br)$  
such that $\br\sim \br_{f,p}$. This conjecture is called the Serre's 
weight conjecture (for $GL_2/\Q$) and it was proved by Edixhoven in 
\cite{Edix}. After that, the modularity of  $\br$ as above was completely proved by Khare-Wintenberger and Serre's weight conjecture played an essential role there (see \cite{kw1},\cite{kw2}). 
In \cite{Edix}, Edixhoven exploited the theory of mod $p$ modular forms by Katz (cf. \cite{katz}) and 
theta cycles studied by Jochnowitz \cite{joch}. Plugging these into several deep properties of Galois 
representations, he proved Serre's weight conjecture. 
The weight conjecture has been studied in the context of 
a mod $p$ local Langlands conjecture 
and in Gee's philosophy so that how a given mod $p$ (local) Galois 
representation can lift to a crystalline lift (cf.  \cite{herzig}, \cite{bdj},\cite{gee},\cite{GHS}). 

Let $GSp_4=GSp_J$ be the symplectic similitude group in $GL_4$ associated to $J=\begin{pmatrix} 0_2& s\\-s &0_2\end{pmatrix},\ 
s=\begin{pmatrix} 0& 1\\1 &0\end{pmatrix}$  with the similitude character $\nu:GSp_4\lra GL_1$. We regard  $GSp_4$ as a group scheme over $\Z$. 
In this paper, we are concerned with the weight conjecture for $GSp_4$. 
Let $S_{N,p}$ be the Siegel modular threefold over $\bF_p$ with principal level  
structure $N\ge 3$. 
Let $F$ be a mod $p$ (geometric) Siegel cuspidal eigenform of degree 2 of level $N$ and the weight 
$(k_1,k_2), k_1\ge k_2\ge 1$,  
which is regarded as a global section of an automorphic sheaf on $S_{N,p}$ (see the next section 
for the precise definition). 
Thanks to the works \cite{taylor-thesis},\cite{taylor},\cite{taylor-low},\cite{laumon},\cite{wei1},\cite{wei}, 
multiplying by the Hasse invariant if necessary, 
one can associate to $F$ a mod $p$ Galois representation 
$\br_{F,p}:G_\Q\lra {\rm GL}_4(\bF_p)$. It is easy to see that in fact  
$\br_{F,p}$ takes the values in ${\rm GSp}_4(\bF_p)$. 
It also satisfies that $\nu(\br_{F,p}(c))=-1$, for any complex conjugation $c$, and 
we call such $\br_{F,p}$ symplectically odd.  
Conversely, one expects that any symplectically odd, irreducible Galois representation 
$\br:G_\Q\lra {\rm GSp}_4(\bF_p)$ is modular, namely, that there exists 
an $F$ as above such that $\br\sim \br_{F,p}$. To be more precise, it means that 
$\iota\circ \br \sim \iota\circ \br_{F,p}$ in ${\rm GL}_4(\bF_p)$, where 
$\iota:GSp_4\hookrightarrow GL_4$ is the natural inclusion.  We often say 
$\br$ is modular of weight $(k_1,k_2)$ to insist the weight of $F$.    
The main purpose of this paper is to study weight reduction theorem for 
classical weights of mod $p$ 
Siegel modular forms of degree 2. 
We will prove the following theorem:       
\begin{thm}\label{main-thm1} 
Suppose $p\ge 5$. 
Let $\br:G_\Q\lra {\rm GSp}_4(\bF_p)$ be a 
continuous representation which is not necessarily irreducible. Assume that 
$\br$ comes from a mod $p$ Siegel cuspidal eigenform of degree 2. 
There exist an integer $0\le \alpha \le p-2$ and a mod $p$ Siegel cuspidal eigenform $F$ of weight $(k_1,k_2),\ k_1\ge k_2\ge 1$ 
such that 
\begin{enumerate}
\item $p>k_1-k_2+3$ and $k_2\le p^4+p^2+2p+1$,
\item $F$ is not identically zero on the superspecial locus in $S_{N,p}$, and 
\item $\iota\circ\br\sim \iota\circ(\overline{\chi}^{\alpha}_p\otimes \br_{F,p})$,  
\end{enumerate}
where $\overline{\chi}_p$ stands for the mod $p$ cyclotomic character of $G_\Q$. 
\end{thm}
By using the operator $\theta_1=\theta^{(k_1,k_2)}_1$ defined in Section \ref{theta-operators}, 
one can reduce $k_1-k_2$ to $0$ or 1 according to its parity though 
$k_2$ may be changed. Therefore we also have the following  theorem: 
\begin{thm}\label{main-thm2} 
Suppose $p\ge 5$. 
Let $\br:G_\Q\lra {\rm GSp}_4(\bF_p)$ be a 
continuous representation which is not necessarily irreducible. Assume that 
$\br\sim \br_{G,p}$ for some mod $p$ Siegel cuspidal eigenform $G$ of degree 2 of 
weight $(k'_1,k'_2)$ with $k'_1\ge k'_2\ge 1$. Put $\varepsilon=\frac{1}{2}(1-(-1)^{k'_1-k'_2})\in \{0,1\}$. Then, for each 
$i$ satisfying $0\le i< \ds\frac{p-3}{2}$, 
there exist an integer $0\le \alpha \le p-2$ and a mod $p$ Siegel cuspidal eigenform $F$ of weight $(k_1,k_2),\ k_1\ge k_2\ge 1$ 
such that 
\begin{enumerate}
\item $k_1-k_2=\varepsilon+2i$ and $k_2\le p^4+p^2+2p+1$,
\item $F$ is not identically zero on the superspecial locus in $S_{N,p}$,  
\item $\iota\circ\br\sim \iota\circ(\overline{\chi}^{\alpha}_p\otimes \br_{F,p})$. 
\end{enumerate}
\end{thm}

The (classical) weight of a mod $p$ Siegel modular form is defined by a pair of integers 
$(k_1,k_2),k_1\ge k_2\ge 1$ and 
it corresponds to the algebraic representation 
${\rm Sym}^{k_1-k_2}{\rm St}_2\otimes {\rm det}^{k_2} {\rm St}_2$ of $GL_2$. 
The above theorem enables us to narrow down the possible weights to 
a specific range of weights via the consideration of mod $p$ Galois representations 
up to a twist by a power of the mod $p$ 
cyclotomic character.    

In the course of proofs of the above theorems, 
 several kinds of theta operators play an important role. 
Harris \cite{harris} geometrically studied differential operators acting on 
Siegel modular forms over the complex numbers. 
Our operators are mod $p$ analogues of some of these classical differential operators.   

By using geometric techniques we define theta 
operators (in Section \ref{theta-operators}) which preserve 
cuspidality: 
\begin{itemize}
\item for $i=1,2,3$, $k_1\ge k_2\ge 1$, and $p>k_1-k_2+2$, 
$$\theta_i:GM_{(k_1,k_2)}(\G(N),\bF_p)\lra 
GM_{(k_1+p-2+i,k_2+p+2-i)}(\G(N),\bF_p)$$
and 
\item for $k\ge 1$ and $p\ge 5$, 
$$\Theta:GM_{(k,k)}(\G(N),\bF_p)\lra 
GM_{(k+p+1,k+p+1)}(\G(N),\bF_p)$$
\end{itemize}
where $GM_{(k_1,k_2)}(\G(N),\bF_p)$ stands for the space of geometric 
modular forms over $\bF_p$ of weight $(k_1,k_2)$ with respect to $\G(N)$. 
For a Hecke eigenform $F$ in $GM_{(k_1,k_2)}(\G(N),\bF_p)$, 
we denote by $\lambda_F(\ell^\alpha)$ the Hecke eigenvalue of the 
Hecke operator $T(\ell^\alpha)$ for a prime $\ell$ and a non-negative integer $\alpha$ 
(see next section). 
Summarizing some of our results concerning theta operators, we have 
\begin{thm}$($Proposition \ref{theta} and  Proposition  \ref{theta-vector}$)$ 
If  $F$ is a cuspidal eigenform in $GM_{(k_1,k_2)}(\G(N),\bF_p)$, then 
\begin{enumerate}
\item $\theta_i(F)$ and $\Theta(F)$ are also cuspidal eigenforms; 
\item Suppose neither of them is identically zero, then for each $i=1,2,3$, 
$$\lambda_{\theta_i(F)}(\ell^\alpha)=\ell^\alpha\lambda_F(\ell^\alpha),\ 
\lambda_{\Theta(F)}(\ell^\alpha)=\ell^{2\alpha}\lambda_F(\ell^\alpha)$$
respectively. 
In other words, 
$$\br_{\theta_i(F),p}\sim \overline{\chi}_p\otimes \br_{F,p},\ 
\br_{\Theta(F),p}\sim \overline{\chi}^2_p\otimes \br_{F,p}.$$
\end{enumerate}
\end{thm}
 
By using these operators and working on the superspecial locus with Ghitza's results (\cite{ghitza1},\cite{ghitza2}), we will reduce the difference $k_1-k_2$ to satisfy  $p>k_1-k_2+3$. 
In particular, the theta operator $\theta_1$ reduces $k_1-k_2$ to $k_1-k_2-2$ but, 
instead,  
$k_2$ increases. 
This process can be continued if necessary until $k_1-k_2$ becomes 0 or 1, depending 
on the parity of the difference of the original weights.  
Next we will reduce $k_2$ by making use of the partial 
Hasse invariant on each Ekedahl-Oort stratum in $S_{N,p}$ and 
its extension to the Zariski closure of each stratum. 
The partial Hasse invariants we use here were defined by Oort in \cite{oort2} and 
several people have studied how such invariants extend (see \cite{boxer},\cite{Kos},\cite{GS}). In this paper we will apply the method of 
 \cite{GS} which shows us not only a way to extend but also the local  behaviors of the partial Hasse invariants 
along the boundary.  
We then restrict our mod $p$ Siegel modular forms to each stratum. The bottom stage is the 
superspecial locus  and there one can reduce $k_2$ to be less than $p+1$. Then we will lift them up to a form on $S_{N,p}$ keeping 
the difference $k_1-k_2$ but 
the vanishing of the obstruction for the liftability constraints us so that we have to increase $k_2$ depending on the weights 
of partial Hasse invariants.  

Our future purpose is definitely to study $\br_{F,p}|_{G_{\Q_p}}$ where $G_{\Q_p}:={\rm Gal}(\bQ_p/\Q_p)$. 
It should be studied in terms of $p$-adic Hodge theory as in \cite{blz}, \cite{berger1} \cite{bg}, \cite{yasuda-yamashita} for the 2-dimensional case.  
In the reducible case  
we would be able to apply similar arguments. However, if $\br_{F,p}|_{G_{\Q_p}}$ is irreducible we need to extend the previous results in loc.cit. to 
the higher dimensional case.  
Further as the readers would have realized, the bound for $k_2$ in the main theorem is too big and 
this will complicate studying the integral $p$-adic Hodge structure of the representations in question.  
 We would address this program somewhere else. 

On the other hand, the theta cycles studied in this paper have some interesting nature which do not appear in \cite{joch}, \cite{Edix} in the case of $GL_2/\Q$.     
It seems also interesting to study a relation between the weights in the theta cycles and 
the conjectural Serre weights defined in \cite{til&her}, \cite{tilouine1}. 
It may be possible to carry out it by using the results in \cite{gg} but this will be addressed in the future.  

The paper is organized as follows. 
In Section 2 we recall the basics of Siegel modular forms of degree 2. 
The readers who are familiar with such objects may skip this section. 
The explicit formulas for Hecke operators in Section 2  
will be used in the next section to compute the effect of theta operators on 
Hecke eigenvalues. Section 3 is devoted to studying  
mod $p$ Siegel modular forms and to the definition of various theta operators \`a la Katz. 
In Section 4, we study weight reduction by using Ghitza's works. 
In Section 5 and 6 we give the definition of theta cycles and study their basic properties.  
Finally, in the appendix we give an explicit form of Pieri's formula for a non-canonical 
decomposition of the tensor product of two symmetric representations of $GL_2/\bF_p$. This decomposition will be used 
to construct various theta operators.  

While preparing this paper the author realized that Max Flander and Ghitza-McAndrew have studied similar operators for 
mod $p$ Siegel modular forms of general degree  (see \cite{flander}, \cite{GMc}). 
As can be expected, some of the results in this paper have been generalized in  \cite{GMc}. 
We should also mention the theory of theta operators are developed by 
de Shalit and Goren \cite{GS},\cite{GS1} for unitary Shimura varieties associated to 
$U(n,1)$ and by Eischen with her collaborators \cite{EFGMM},\cite{EM} for split 
unitary groups and symplectic groups. 

Needless to say, our theory of theta operators (for $GSp_4/\Q$) is based on mod $p$ geometry and 
the consideration of (partial) Hasse invariants. These ingredients play an important role in recent great achievement in potentially 
automorphy for non-regular symplectic motives of rank 4 (in the sense of Gross 
\cite{G-motives}) due to   
Calegari-Geraghty \cite{CG1} and Boxer-Calegari-Gee-Pilloni \cite{BCGP}.

\textbf{Acknowledgement.} The author would like to thank K-W. Lan, A. Ghitza, F. Herzig, S. Morra, S. Harashita, and 
M-H Nicole for helpful discussions. In particular, Herzig kindly 
informed me a reference \cite{winter} and explained his joint work with Tilouine \cite{til&her}. 
Dr Ortiz Ramirez pointed out an error of the argument in Theorem \ref{main1} of 
an earlier version. The author would like to give sincere thanks to him. 
A part of this work was done during the author's visiting to University of Toronto as a  
JSPS Postdoctoral Fellowship for Research Abroad No.378. The author would like to special thank 
Professor Henry Kim, James Arthur, Kumar Murty and staffs there for kindness and the university for hospitality. 
Finally, the author would like to give special thanks to the referee, whose suggestions have greatly improved the presentation and readability of this paper. 

He is now partially supported by JSPS KAKENHI Grant Number (B) No.19H01778.

\section{Siegel modular forms of degree 2}
In this section we shall discuss Siegel modular forms in various settings. 
As basic references, we refer to \cite{asgari&schmidt} for the classical setting, \cite{borel&jacquet} for the adelic setting, and 
\cite{harris}, \cite{taylor}, \cite{s&u} for the geometric setting. 
In this section we will work on $GSp_4=GSp_{J'}$ defined by $J'=\begin{pmatrix} 0_2& I_2\\ -I_2 &0_2\end{pmatrix}$. 
It is easy to see that $GSp_J$ is conjugate to $GSp_{J'}$  in $GL_4$ and we can convert everything from  
one to the other. We remark that $GSp_4$ is a smooth group scheme over $\Z$ and 
sometimes denote it by $G$.  

Let $\nu$ be the similitude character of $GSp_4=GSp_{J'}$ and $Sp_4$ the kernel of $\nu$. 

\subsection{Classical Siegel modular forms}\label{class}
Let us consider the Siegel 
upper half-plane $\mathcal{H}_2=\{Z\in M_2(\C)|\ {}^tZ=Z,\  {\rm Im}(Z)>0\}$.  
For a pair of integers $\underline{k}=(k_1,k_2)$ such that $k_1\ge k_2\ge 1$, we define the 
algebraic representation $\lambda_{\underline{k}}$ of $GL_2$ by 
$$V_{\underline{k}}={\rm Sym}^{k_1-k_2}{\rm St}_2\otimes {\rm det}^{k_2} {\rm St}_2,
$$ 
where ${\rm St}_2$ is the standard representation of dimension 2 with the basis $\{e_1,e_2\}$. More explicitly, if $R$ is any ring, 
then $V_{\underline{k}}(R)=\ds\bigoplus_{i=0}^{k_1-k_2}Re^{k_1-k_2-i}_1\cdot e^i_2$ 
(note that $e^0_1=e^0_2:=1$ as convention) and for 
$g=\left(\begin{array}{cc}
a & b \\
c & d
\end{array}
\right)
\in GL_2(R)$,  $\lambda_{\underline{k}}(g)$ acts on $V_{\underline{k}}(R)$ by 
$$g\cdot e^{k_1-k_2-i}_1\cdot e^i_2:=\det(g)^{k_2}(ae_1+ce_2)^{k_1-k_2-i}\cdot (be_1+de_2)^i.$$
We identify $V_{\underline{k}}(R)$ (resp.  
$\lambda_{\underline{k}}(g)$) with $R^{\oplus (k_1-k_2+1)}$ (resp. the represent matrix of $\lambda_{\underline{k}}(g)$ 
with respect to the above basis).
We define the action and the automorphy factor $J$ by
\begin{equation}\label{sp4-action}
\gamma Z=(AZ+B)(CZ+D)^{-1}, \quad J(\gamma,Z)=CZ+D\in GL_2(\C),
\end{equation}
for
$\gamma=
\left(\begin{array}{cc}
A& B\\
C& D
\end{array}
\right)
\in GSp_4(\R)^+$ and $Z\in \mathcal{H}_2$.

For an integer $N\ge 1$, we define a principal congruence subgroup $\Gamma(N)$ to be 
the group consisting of the elements $g\in Sp_4(\Z)$ such that $g\equiv 1 \ {\rm mod}\ N$. 

For a holomorphic $V_{\underline{k}}(\C)$-valued function $f$ on $\mathcal{H}_2$, the action of 
$\gamma \in GSp_4(\R)^+$ is defined by 
\begin{equation}\label{transformation}
f(Z)|[\gamma]_{\underline{k}}:=\lambda_{\underline{k}}(\nu(\gamma)J(\gamma,z)^{-1})f(\gamma Z).
\end{equation} 
For an arithmetic subgroup $\Gamma$ of $Sp_4(\Q)$ and a finite character 
$\chi:\Gamma\lra \C^\times$,  
we denote by $M_{\underline{k}}(\Gamma,\chi)$ the space of  
(vector values) holomorphic functions $f:\mathcal{H}_2\lra V_{\underline{k}}(\C)$ 
satisfying 
$f|[\gamma]_{\underline{k}}=\chi(\gamma)f$ for any $\gamma\in\Gamma$. 

For a Siegel modular form $f\in M_{\underline{k}}(\Gamma,\chi)$, the Siegel $\Phi$-operator is defined by 
$$\Phi(f)(z):=\lim_{t\to \infty}
f(
\left(\begin{array}{cc}
z & 0 \\
0 & \sqrt{-1}t
\end{array}
\right))
\mbox{ for $z\in \mathcal{H}_1$ } $$
where $\mathcal{H}_1=\{z\in\C\ |\ {\rm Im}(z)>0\}$ and we say $f$ is a cusp form if $\Phi(f|[\gamma]_{\underline{k}})=0$ for any $\gamma\in Sp_4(\Q)$. 
We denote by  $S_{\underline{k}}(\Gamma,\chi)$  the space of such cusp forms inside $M_{\underline{k}}(\Gamma,\chi)$. 

We shall define the Hecke operators on $M_{\underline{k}}(\Gamma(N)):=
M_{\underline{k}}(\Gamma(N),\textbf{1})$ where $\textbf{1}$ stands for the trivial 
character. We refer to \cite{evdokimov} for Hecke operators in this setting.   
For any positive integer $n$ coprime to $N$, let 
$$\Delta(N):=\Bigg\{g\in M_4(\Z)\cap GSp_4(\Q)^+\ \Bigg|\ 
g\equiv\left(\begin{array}{cc}
I_2 & 0   \\
0 & \nu(g)I_2 
\end{array}
\right)\ {\rm mod}\ N ,\ \nu(g)^{\pm1}\in\Z[\frac{1}{N}]^\times \Bigg\}$$
where $GSp_4(\Q)^+$ is the subgroup of $GSp_4(\Q)$ consisting of 
all elements whose similitudes are positive. 
For $m\in \Delta(N)$, we introduce the actions of the Hecke operators on $M_{\underline{k}}(\Gamma(N))$ 
 by 
\begin{equation}\label{hecke-ope}
T_mf(Z):=\nu(m)^{\frac{k_1+k_2}{2}-3}\ds\sum_{\alpha\in \Gamma(N)\backslash\Gamma(N) m
\Gamma(N)} f(Z)|[\nu(m)^{-\frac{1}{2}}\alpha]_{\underline{k}} 
\end{equation}
(notice that $\nu(m)$ is positive and $\nu(m)^{-\frac{1}{2}}\alpha$ belongs to 
${\rm Sp}_4(\R)$) and for any positive integer $n$, put  
\begin{equation}\label{Hecke-ope-def}
T(n):=\sum_{m\in \G(N)\backslash \Delta(N)/\G(N)\atop \nu(m)=n}T_m.
\end{equation}
We call the factor $\nu(m)^{\frac{k_1+k_2}{2}-3}$ in the formula (\ref{hecke-ope}) 
the normalizing factor of $T_m$ for the weight $\lambda_{\underline{k}}$. 
We also consider the same actions on $S_{\underline{k}}(\G(N)):=S_{\underline{k}}(\G(N),
\textbf{1})$. 
For $t_1={\rm diag}(1,1,p,p),\ t_2={\rm diag}(1,p,p^2,p)$, put $T_{i,p}:=T_{t_i}\ i=1,2$ and 
fix $\widetilde{S}_{p,1}, \widetilde{S}_{p,p}\in Sp_4(\Z)$ so that 
$\widetilde{S}_{p,1}\equiv {\rm diag}(p^{-1},1,p,1)\ 
{\rm mod}\ N$ and $\widetilde{S}_{p,p}\equiv {\rm diag}(p^{-1},p^{-1},p,p)\ {\rm mod}\ N$. 
Put $S_{p}:=\widetilde{S}_{p,p}T_{pI_4}=p^{(k_1+k_2-6)}\widetilde{S}_{p,p}$ and note that 
it commutes with any Hecke operator. 
Then we see that 
\begin{equation}
T(p)=T_{1,p},\quad T^2_{1,p}-T(p^2)-p^2 S_{p}=p\{T_{2,p}+(1+p^2)S_{p}\}.
\end{equation}

Since the group $\G(N)$ contains the subgroup consists of 
$\left(\begin{array}{cc}
I_2& NS \\
0 & I_2
\end{array}
\right), S={}^tS\in M_2(\Z)$, for a given $F\in M_{\underline{k}}(\G(N))$, we have the Fourier expansion 
\begin{equation}
F(q)=\ds\sum_{T\in \mathcal{S}(\Z)_{\ge 0}}A_F(T)q^T_N,\ 
q^T_N=e^{\frac{2\pi\sqrt{-1}}{N}{\rm tr}(TZ)},\ A_F(T)\in V_{\uk}(\C) 
\end{equation}
 where 
$\mathcal{S}(\Z)_{\ge 0}$ is the subset of $M_2(\Q)$ consisting of all symmetric matrices 
$\left(\begin{array}{cc}
a& \frac{b}{2} \\
\frac{b}{2} & c
\end{array}
\right),\ a,b,c\in \Z$ which are semi-positive. We say $A_F(T)$ is the $T$-th Fourier coefficient of $F$. 

In terms of Fourier coefficients, for any subring $R$ of $\C$ we define 
$$M_{\underline{k}}(\Gamma(N),R):=\{F\in M_{\underline{k}}(\Gamma(N))\ |\ A_F(T)\in V_{\underline{k}}(R) \mbox{ for all $T\in \mathcal{S}(\Z)_{\ge 0}$} \}$$
and $S_{\underline{k}}(\Gamma(N),R)$ as well. Finally for any discrete subgroup $\G$ of $Sp_4(\Z)$ with finite index  and a character $\chi:\G\lra \C^\times$ so that 
${\rm Ker}(\chi)$ contains a principal congruence subgroup $\G(N')$ for some $N'>0$, 
put $$M_{\underline{k}}(\G,\chi,R)=M_{\underline{k}}(\Gamma,\chi)\cap M_{\underline{k}}(\Gamma(N'),R),\ 
S_{\underline{k}}(\G,\chi,R)=S_{\underline{k}}(\Gamma,\chi)\cap S_{\underline{k}}(\Gamma(N'),R).$$
Here we take the intersection inside of  $M_{\underline{k}}(\Gamma(N'))$ which includes $M_{\underline{k}}(\G,\chi)$ 
and it is as well for $S_{\underline{k}}(\G,\chi,R)$. 
We should remark that the Hecke operators do not preserve  $M_{\underline{k}}(\G,\chi)$ for a general $\G$ and $\chi$ (cf. p.465-466 of \cite{manni&top}). 

\subsection{A formula for Hecke operators}\label{integrality}

The finite group $Sp_4(\Z/N\Z)\simeq Sp_4(\Z)/\G(N)$ acts on  
$M_{\underline{k}}(\G(N))$ by $F\mapsto F|[\tilde{\gamma}]_{\underline{k}}$ if 
we fix a lift $\tilde{\gamma}$ of $\gamma\in Sp_4(\Z/N\Z)$. 
We denote this action by the same notation $F|[\gamma]_{\underline{k}}$. 
This action does not depend on the choice of lift. 
The diagonal subgroup of $Sp_4(\Z/N\Z)$ is isomorphic to $(\Z/N\Z)^\times\times (\Z/N\Z)^\times$ by sending 
$S_{a,b}:={\rm diag}(a^{-1},b^{-1},a,b)$ to 
$(a,b)$ and it also acts on $M_{\underline{k}}(\G(N))$ through the action of $Sp_4(\Z/N\Z)$. 
We have the character decomposition   
\begin{equation}\label{character}
M_{\underline{k}}(\G(N))=\bigoplus_{\chi_1,\chi_2:(\Z/N\Z)^\times\lra\C^\times}M_{\underline{k}}(\G(N), \chi_1,\chi_2),
\end{equation}
where  $M_{\underline{k}}(\G(N), \chi_1,\chi_2)=\{F\in M_{\underline{k}}(\G(N))\ |\ F|[S_{a,1}]_{\underline{k}}=\chi_1(a)F \ {\rm and}\  F|[S_{a,a}]_{\underline{k}}=\chi_2(a)F \}$ 
(see p.466 of \cite{manni&top}). 
It is easy to see that the Hecke operators preserve $M_{\underline{k}}(\G(N), \chi_1,\chi_2)$ (cf. p.465, Lemma 3.1 of \cite{manni&top}). 
We should remark that if the space is non-zero, the weight $(k_1,k_2)$ has to satisfy the parity condition 
\begin{equation}\label{parity}
\chi_2(-1)=(-1)^{k_1+k_2}. 
\end{equation}
Throughout this paper we always assume this parity condition.  

Let $F(q)=\ds\sum_{T\in \mathcal{S}(\Z)_{\ge 0}}A_F(T)q^T_N\in M_{\underline{k}}(\G(N), \chi_1,\chi_2)$ be an eigenform for all $T(p^i),\ p\nmid N,\ i\in\N$ with eigenvalues $\lambda_F(p^i)$, i.e.,
\begin{equation}\label{eigen}
T(p^i)F=\lambda_F(p^i)F.
\end{equation}

Let us study the relation between $\lambda_F(p^i)$ and $A_F(T)$. 
For a non-negative integer $\beta$, let $R(p^\beta)$ be a set of matrices 
$\left(\begin{array}{cc}
u_1& u_2 \\
u_3 & u_4
\end{array}
\right)$ of $\G^{1}(N):=\{g\in SL_2(\Z)\ |\ g\equiv 1_2\ {\rm mod}\ N \}$ whose first rows $(u_1,u_2)$ run 
over a complete set of representatives modulo the equivalence relation: 
$$(u_1,u_2)\sim (u'_1,u'_2) \Longleftrightarrow [u_1:u_2]=[u'_1:u'_2]\ {\rm in}\ \mathbb{P}^1(\Z/p^\beta\Z).$$
Put $T(p^i)F=\ds\sum_{T\in \mathcal{S}(\Z)_{\ge 0}}A_F(p^i;T)q^T_N$ 
where $A_F(p^i;T):=A_{T(p^i)F}(T)$. For simplicity we write 
$\rho_j={\rm Sym}^{j}{\rm St}_2$ for $j\ge 0$ and $UT{}^tU=\left(\begin{array}{cc}
a_U& \frac{b_U}{2} \\
\frac{b_U}{2}  & c_U
\end{array}
\right)$ for $T\in \mathcal{S}(\Z)_{\ge 0}$ and $U\in R(p^\beta)$. 
By using Proposition 3.1 of \cite{evdokimov} and the calculations done at p.439-440 loc.cit., we have 
\begin{eqnarray}
\label{hecke-fourier} 
\lambda_F(p^i)A_F(T)=A_F(p^i;T)=\sum_{\alpha+\beta+\gamma=i\atop \alpha,\beta,\gamma\ge 0}
\chi_1(p^\beta)\chi_2(p^\gamma)p^{\beta(k_1-2)+\gamma(k_1+k_2-3)}\times \nonumber\\
\sum_{U\in R(p^\beta)\atop {a_U\equiv 0\ {\rm mod}\ p^{\beta+\gamma}\atop 
b_U\equiv c_U\equiv 0\ {\rm mod}\ p^{\gamma}}}
\rho_{k_1-k_2}((\left(
\begin{array}{cc}
1& 0 \\
0  & p^\beta
\end{array}
\right)U)^{-1})
A_F\left(p^\alpha\left(
\begin{array}{cc}
a_Up^{-\beta-\gamma}& \frac{b_Up^{-\gamma}}{2} \\
\frac{b_Up^{-\gamma}}{2}  & c_Up^{\beta-\gamma}
\end{array}
\right)\right). 
\end{eqnarray}
\begin{rmk}\label{rmk-on-integrality}
Fix an isomorphism $\bQ_p\simeq \C$ and put $R=\overline{\Z}_p$. 
By the above formula, $M_{\underline{k}}(\G(N),R)$ is stable under 
the action of $T(p^i)$ for any $i\ge 0$ provided if $k_2\ge 2$. 
If $k_2=1$, this would be false.  
\end{rmk}

\subsection{Siegel modular forms with a general weight}\label{general-weight} 
Recall that $\lambda_{\underline{k}}=\lambda_{(k_1,k_2)}= \rho_{k_1-k_2}\otimes \det^{k_2}{\rm St}_2$ where 
$\rho_{k_1-k_2}={\rm Sym}^{k_1-k_2}{\rm St}_2$ and 
let $\lambda_{\underline{k}'}=\lambda_{(k'_1,k'_2)}$ be another weight.  
As in the previous subsection, we can consider a Siegel modular form $F$ of the weight 
$\lambda_{\underline{k},\underline{k}'}:=\lambda_{(k_1,k_2)}\otimes \lambda_{(k'_1,k'_2)}$ 
with respect to $\G(N)$. We denote by $M_{\lambda_{\underline{k},\underline{k}'}}(\G(N))$ 
the space of such Siegel modular forms. We can also define  
$M_{\underline{k},\underline{k}'}(\G(N), \chi_1,\chi_2)$ in the obvious manner. 
It is  known that $\lambda_{\underline{k},\underline{k}'}$ decomposes into the irreducible representations as follows:
\begin{equation}\label{decomp-tensor}
\lambda_{\underline{k},\underline{k}'}=\lambda_{(k_1,k_2)}\otimes \lambda_{(k'_1,k'_2)}\simeq 
\bigoplus_{j=0}^\mu \lambda_{(k_1+k'_1-j,k_2+k'_2+j)},\ \mu=\min\{k_1-k_2,k'_1-k'_2\}.
\end{equation} 
Notice that clearly the highest weight that is $(k_1+k_2)+(k'_1+k'_2)$ is preserved under this decomposition. 
Therefore as in (\ref{hecke-ope}) we may define the Hecke action on $F$ by 
\begin{equation}\label{hecke-action-general}
T_m F(Z):=\nu(m)^{\frac{k_1+k_2+k'_1+k'_2}{2}-3}\ds\sum_{\alpha\in \Gamma(N)\backslash\Gamma(N) m
\Gamma(N)} f(Z)|[\nu(m)^{-\frac{1}{2}}\alpha]_{\underline{k},\underline{k'}},\ m\in \Delta_n(N)
\end{equation} 
where $F(Z)|[\nu(m)^{-\frac{1}{2}}\alpha]_{\underline{k},\underline{k'}}=
\lambda_{\underline{k},\underline{k}'}(\nu(m)^{-\frac{1}{2}}\alpha)J((\nu(m)^{-\frac{1}{2}}\alpha,z)^{-1})F(\alpha Z)$. 

Let $F=\ds\sum_{T\in \mathcal{S}(\Z)_{\ge 0}}A_F(T)q^T_N\in M_{\underline{k},\underline{k}'}(\G(N), \chi_1,\chi_2)$ be an eigenform for all $T(p^i),\ p\nmid N,\ i\in\N$ with eigenvalues $\lambda_F(p^i)$. 
Let $T(p^i)F=\ds\sum_{T\in \mathcal{S}(\Z)_{\ge 0}}A_F(p^i;T)q^T_N$. 
As in (\ref{hecke-fourier}), we have 
\begin{eqnarray}
\label{hecke-fourier1} 
\lambda_F(p^i)A_F(T)=A_F(p^i;T)=\sum_{\alpha+\beta+\gamma=i\atop \alpha,\beta,\gamma\ge 0}
\chi_1(p^\beta)\chi_2(p^\gamma)p^{\beta(k_1+k'_1-2)+\gamma(k_1+k'_1+k_2+k'_2-3)}\times \nonumber\\
\sum_{U\in R(p^\beta)\atop {a_U\equiv 0\ {\rm mod}\ p^{\beta+\gamma}\atop 
b_U\equiv c_U\equiv 0\ {\rm mod}\ p^{\gamma}}}
(\rho_{k_1-k_2}\otimes \rho_{k'_1-k'_2})((\left(
\begin{array}{cc}
1& 0 \\
0  & p^\beta
\end{array}
\right)U)^{-1})
A_F\left(p^\alpha\left(
\begin{array}{cc}
a_Up^{-\beta-\gamma}& \frac{b_Up^{-\gamma}}{2} \\
\frac{b_Up^{-\gamma}}{2}  & c_Up^{\beta-\gamma}
\end{array}
\right)\right). 
\end{eqnarray}
According to (\ref{decomp-tensor}), we have the decomposition preserving Hecke actions:
$$M_{\underline{k},\underline{k}'}(\G(N), \chi_1,\chi_2)\simeq \bigoplus_{j=0}^\mu 
M_{(k_1+k'_1-j,k_2+k'_2+j)}(\G(N), \chi_1,\chi_2).$$
\subsection{Adelic forms}\label{adelic-forms}
In this section we refer to \cite{borel&jacquet} and \cite{taylor-thesis}. Recall $G=GSp_4$. 
Let $\A$ be the adele ring of $\Q$ and $\A_f=\widehat{\Z}\otimes_\Z\Q$ the finite adele of $\Q$.  
For a positive integer $N$, 
let $K(N)$ be the group consisting of the elements 
$g\in G(\widehat{\Z})$ such that $g\equiv I_4\ {\rm mod}\ N$. Then we see that  $\G(N)=Sp_4(\Q)\cap K(N)$ and $\nu(K(N))=I_4+N\widehat{\Z}$. 
Then it follows from the strong approximation theorem for $Sp_4$ that 
\begin{equation}
\label{sat2} G(\A)=\coprod_{1\le a < N \atop (a,N)=1}G(\Q)G(\R)^+ d_a K(N)=\coprod_{1\le a < N \atop (a,N)=1}G(\Q)Z_G(\R)^+Sp_4(\R)d_a K(N)
\end{equation}
where $d_a$ is the diagonal matrix such that $(d_a)_p={\rm diag}(a,a,1,1)$ if $p|N$, $(d_a)_p=I_4$ otherwise. 

Let $I:=\sqrt{-1}I_2\in\mathcal{H}_2$ and $U(2)={\rm Stab}_{Sp_4(\R)}(I)$. 
For any open compact subgroup $U$ of $G(\widehat{\Z})$, 
we let $\mathcal{A}_{\underline{k}}(U)^\circ$ denote the subspace of functions 
$\phi:G(\Q)\backslash G(\A)\lra V_{\underline{k}}(\C)$ such that 

\begin{enumerate}
\item $\phi(gu u_\infty)=\lambda_{\underline{k}}(J(u_\infty,I)^{-1})\phi(g)$ for all $g\in G(\A)$, $u\in U$, and $u_\infty\in U(2)Z_G(\R)^+$ where $Z_G(\R)^+=\R_{>0}I_4$.  
\item for $h\in G(\A_f)$, the function 
$$\phi_h:\mathcal{H}_2\lra V_{\underline{k}}(\C),\ \phi_h(Z)=\phi_h(g_\infty I):=\lambda_{\underline{k}}(J(g_\infty,I))\phi(hg_\infty)$$
is holomorphic where $Z=g_\infty I,\ g_\infty\in GSp^+_4(\R)=Z_G(\R)^+Sp_4(\R)$ (note that this definition is independent of the choice of $g_\infty$),
\item for $g\in G(\A)$, $\ds\int_{N_P(\Q)\backslash N_P(\A)}\phi(ng)dn=0 $ for 
the unipotent radical $N_P$ of any $\Q$-parabolic subgroup $P$ of $G$ and $dn$ 
is a Haar measure on $N_P(\Q)\backslash N_P(\A)$. 
\end{enumerate}

We define similarly $\mathcal{A}_{\underline{k}}(U)$ by omitting the last condition (3). 

Let $\G(N)_a:=Sp_4(\Q)\cap d_a^{-1}K(N)d_a$. Note that $\G(N)_a=\G(N)$ for each $a$. 
Then we have the isomorphism 
\begin{equation}\label{isom}
\mathcal{A}_{\underline{k}}(K(N))\stackrel{\sim}{\lra} \bigoplus_{1\le a < N\atop (a,N)=1}M_{\underline{k}}(\G(N)_a), \quad \phi\mapsto (\phi_{d_a}).
\end{equation}
The inverse of this isomorphism is given as follows: Let $F=(F_a)$ be an element of RHS which is a system of Siegel modular forms such that $F_a\in M_{\underline{k}}(\Gamma(N)_a)$ for each $a$. For each $g\in G(\A)$, there exists a unique $d_a$ such that
$g=rd_a g_\infty k$ with $r\in G(\Q)$, $g_\infty\in Sp_4(\R)Z_G(\R)^+$, and $k\in K(N)$. Then we define the function  
$$\phi_{F}(g)=\lambda_{\underline{k}}(J(g_\infty,I))^{-1}F_a(g_\infty I).$$
This gives rise to the inverse of the above isomorphism. 
We also have the isomorphism 
$\mathcal{A}_{\underline{k}}(K(N))^\circ\simeq \bigoplus_{1\le a < N\atop (a,N)=1}S_{\underline{k}}(\G(N)_a)$ 
as well (cf. \cite{borel&jacquet} for checking the cuspidality). 

Now we restrict the isomorphism (\ref{isom}) to a subspace, using the character decomposition (\ref{character}). 
Given two Dirichlet characters $\chi_i : (\Z/N\Z)^\times \lra \C^\times$, $i = 1, 2$, associate 
the characters $\chi'_i:\A_f\lra \C^\times$  via global class field theory. 

Define
$\widetilde{\chi} : T(\A_f) \lra \C^\times$ by 
$$\tilde\chi'(\diag(*,*, c, d) =\chi'_1(d^{-1}c)^{-1}\chi'_2(d)^{-1}.$$
Choose $F = (F_a)$ from RHS of (\ref{isom}) which satisfies 
$F|[S_{z,z}]_{\underline{k}} =(F_a|[S_{z,z}]_{\underline{k}}) = (\chi_2(z)F_a) = \chi_2(z)F$ and
$F|[S_{z,1}]_{\underline{k}} = \chi_1(z)F$. If we write $g\in G(\A)$ as $g = rz_\infty d_a g_\infty k \in G(\A)$ and take $z_f \in T(\A_f )$, then we define the automorphic function attached to $F$ by
$$
\phi_F(g z_f) = \lambda_{\underline{k}}(J(g, I))^{-1} F_a(g_\infty I)\tilde\chi(z_f).
$$
Then this gives rise to the isomorphism of the subspaces 
\begin{equation*}
\mathcal{A}_{\underline{k}}(K(N), \tilde\chi)\stackrel{\sim}{\lra} \bigoplus_{1\le a < N\atop (a,N)=1} M_{\underline{k}}(\G(N)_a, \chi_1,\chi_2).
\end{equation*}

We now compute the actions of $\diag(1,1,-1,-1)\in G(\R)$ on $\phi_F\in \mathcal{A}_{\underline{k}}(K(N), \tilde\chi)$ as follows:
Let $h=(\diag(1, 1,-1,-1), I_{\A_f})=\diag(1,1,-1,-1)(I_{G(\R)}, (\diag(1,1,-1,-1))_p)\in G(\Q)(G(\R)\times G(\A_f))$,
where $I_{\A_f}$ (resp. $I_{G(\R)}$) is the identity element.
Then we have
\begin{eqnarray*}
\phi_F(gh) =\phi_F((r\cdot\diag(1, 1,-1,-1))z_\infty d_a g_\infty k \cdot \diag(1,1,-1,-1))_p)) \\
= \lambda_{\underline{k}}(J(g, I))^{-1} F_a(g_\infty I)\tilde\chi(\diag(1,1,-1,-1))_p) =\chi_2(-1) \phi_F(g),
\end{eqnarray*}
since we have assumed $\chi_2(-1) = (-1)^{k_1+k_2}$. Hence we have
\begin{equation}\label{sgn}
\phi_F(\diag(1, 1,-1,-1), I_{\A_f})=(-1)^{k_1+k_2}.
\end{equation}

\begin{rmk} The space $M_{\underline{k}}(\Gamma(N))$ is 
embedded into 
$\ds\bigoplus_{1\le a < N\atop (a,N)=1} M_{\underline{k}}(\G(N)_a)$ by $F\mapsto (F|[\gamma_a]_{\underline{k}})_a$. 
So given a cusp form $F\in M_{\underline{k}}(\Gamma(N))$, we obtain $\phi_F\in \mathcal{A}_{\underline{k}}(K(N))$ which under the isomorphism (\ref{isom}), corresponds to $(F|[\gamma_a]_{\underline{k}})_a$, 
and $\phi_F$ gives rise to a cuspidal representation $\pi_F$.
Conversely, given a cuspidal representation $\pi$ of $GSp_4/\Q$, there exists $N>0$ and $\phi\in \mathcal{A}_{\underline{k}}(K(N))$
which spans $\pi$. Under the isomorphism (\ref{isom}), $\phi$ corresponds to $(F_a)_{1\le a < N\atop (a,N)=1}$. 
For any $a$, let $\pi_{F_a}$ be the cuspidal representation associated to $F_a$. Then $\pi$ and $\pi_{F_a}$ have the same Hecke eigenvalues for $p\nmid N$, and hence in the same $L$-packet.
\end{rmk}

We now study the Hecke operators on $\mathcal{A}_{\underline{k}}(K(N))$ and its relation to classical Hecke operators. 
Let $\phi$ be an element of  $\mathcal{A}_{\underline{k}}(K(N))$ and $F=(F_a)_a$ be the corresponding element of RHS via the above isomorphism (\ref{isom}). 
For any prime $p\nmid N$ and  $\alpha \in G(\Q)\cap T(\Q_p)$, we define the Hecke action with respect to $\alpha$ 
$$\widetilde{T}_\alpha \phi(g):= \ds\int_{G(\A_f)}([K(N)_p\alpha K(N)_p]\otimes 1_{K(N)^p})\phi(g g_f) dg_f$$
where $dg_f$ is the Haar measure on $G(\A_f)$ so that vol$(K)=1$. Here $K(N)_p$ is the $p$-component of $K(N)$ and  
 $K(N)^p$ is the subgroup of $K(N)$ consists of trivial $p$-component. 
 
Then by using (\ref{sat2}), we can easily see that 
 \begin{equation}\label{local-global}
T_\alpha F(Z)=\nu(\alpha)^{\frac{k_1+k_2}{2}-3}\widetilde{T}_{\alpha^{-1}} \phi(g)
\end{equation}
where $g=rz_\infty g_a g_\infty k$ as above and $Z=g_\infty I$.  
From this relation, up to the factor of $\nu(\alpha)^{\frac{k_1+k_2}{2}-3}$, the isomorphism (\ref{isom}) preserves Hecke eigenforms in both sides. 
We turn to explain the relation to classical Siegel eigenforms. Let $G\in S_{\underline{k}}(\G(N))$ be a Siegel cusp form which is a Hecke eigenform. 
Then it is easy to see that $(G)_a$ is an eigenform of $\ds\bigoplus_{1\le a < N\atop (a,N)=1}S_{\underline{k}}(\G(N)_a)$. Hence we have the 
Hecke eigenform of  $\mathcal{A}_{\underline{k}}(K(N))$ corresponding to $G$. 
The above things are easy to generalize to all open compact subgroup $U\subset G(\widehat{\Z})$. We omit the details. 

The group $G(\A)$ acts on  $\ds\varinjlim_{U}\mathcal{A}_{\underline{k}}(U)$  (also on  $\ds\varinjlim_{U}\mathcal{A}_{\underline{k}}(U)^\circ$) by 
right translation: 
$$(h\cdot \phi)(g):=\phi(gh), \ {\rm for}\ g,h\in G(\A).$$
For an open compact subgroup $U\subset G(\widehat{\Z})$, we say $U$ is of level $N$ if $N$ is the minimum positive integer so that 
$U$ contains $K(N)$. For such $U$ of level $N$ and $\phi\in \mathcal{A}_{\underline{k}}(U)$ which is an eigenform for all $T_\alpha,\ \alpha\in G(\Q)\cap T(\Q_p)$ and $p{\not|} N$, we denote by $\pi_\phi$ the irreducible maximal subquotient of the representation of $G(\A)$ 
generated by $g\cdot \phi,\ g\in G(\A)$. Then $\pi_\phi$ is an automorphic representation in the sense of \cite{borel&jacquet}. Further if $\phi \in 
\mathcal{A}_{\underline{k}}(U)^0$, then we see that $\pi_\phi$ is a cuspidal automorphic representation.

\subsection{Geometric Siegel modular forms}\label{geometry}
We will discuss arithmetic properties of Siegel modular forms by using the modular interpretation of Siegel modular varieties. 
Henceforth we assume that $N\ge 3$.  

For any $\Z[\frac{1}{N}]$-scheme $T$, we consider the triple $(A,\lambda,\phi)/T$ such that 
\begin{enumerate}
\item $A$ is an abelian $T$-scheme of relative dimension 2, 
\item $\lambda:A\stackrel{\sim}{\lra} A^\vee$ is a principal polarization where $A^\vee$ is the dual abelian scheme of $A$ (cf. Chapter I of \cite{c&f}).  
\item $\phi$ is an isomorphism $A[N]\simeq ({\mu_N})^2\oplus (\Z/N\Z)^2$ of  group schemes over $T$ so that the composition of $\phi$ and 
the symplectic pairing on $({\mu_N})^2\oplus (\Z/N\Z)^2$ defined by $J$ is the Weil pairing $A[N]\times A[N]\lra \mu_N$ defined by $\lambda$. 
\end{enumerate}
Here we remark that in Definition 6.1, p.121  in Chapter IV  of \cite{c&f}, 
the authors fix an identification between $\mu_N$ and $\Z/N\Z$ over $\Z[\zeta_N,1/N]$ 
where $\zeta_N$ is a primitive $N$-th root of unity. To consider such an identification is equivalent to 
choosing $\zeta_N$ and the number of the possible choices is the number of the 
connected components of $S_{K(N)}(\C)$ below. 
However, we do not fix any identification between $\mu_N$ and $\Z/N\Z$ to work on any 
$\Z[\frac{1}{N}]$- scheme $T$ as a base scheme.  
The results in Section 6 in Chapter IV of \cite{c&f} explain 
each component of $S_{K(N)}(\C)$ is irreducible.  

For given two triple $(A,\lambda,\phi)$ and $(A',\lambda',\phi')$ we define  the equivalence relation as follows. 
We denote by  $(A,\lambda,\phi)\sim (A',\lambda',\phi')$ if there exists an $T$-isomorphism $f:A\lra A'$ as abelian $T$-schemes such that 
$$f^\vee\circ \lambda'\circ f=\lambda,\ \phi'\circ f=\phi$$
where $f^\vee:A'^\vee\lra  A^\vee$ is the dual of $f$.   
Then the functor from the category of $\Z[\frac{1}{N}]$-schemes to the category of sets defined by  
$$Sch_{ \Z[\frac{1}{N}]}\lra Sets,\ T\mapsto \{(A,\lambda,\phi)/T\}/\sim$$
is representable by a $\Z[\frac{1}{N}]$-scheme $S_{K(N)}$ (cf. \cite{c&f}). 
It is well known that analytic description of  $S_{K(N)}(\C)$ is 
$$S_{K(N)}(\C)=G(\Q)\backslash G(\A)/Z_G(\R)^+U(2)K(N)=\coprod_{1\le a < N\atop (a,N)=1}\G(N)_a\backslash \mathcal{H}_2=
\coprod_{1\le a < N\atop (a,N)=1}\G(N)\backslash \mathcal{H}_2$$
which is a smooth quasi-projective 3-fold. Here the smoothness follows from the neatness of $K(N)$. 

We now turn to define geometric Siegel modular forms by using an arithmetic toroidal compactification of Siegel moduli scheme $S_{K(N)}$ (cf. \cite{c&f}).  
Let $\overline{S}_{K(N)}$ be an arithmetic toroidal compactification of $S_{K(N)}$ over $\Z[\frac{1}{N}]$ and $\pi:\mathcal{G}\lra \overline{S}_{K(N)}$ be 
the universal semi-abelian scheme with the zero section $e:\overline{S}_{K(N)}\lra \mathcal{G}$. We define the Hodge bundle on $\overline{S}_{K(N)}$ by 
\begin{equation}\label{hodge}
\E:=e^\ast \Omega^1_{\mathcal{G}/\overline{S}_{K(N)}}
\end{equation}
which is called the ``canonical extension"  while $\E(-\mathcal{C})$ with  
$\mathcal{C}:=\overline{S}_{K(N)}-S_{K(N)}$ is called the ``sub-canonical extension" 
(cf. Section 4 of \cite{l&s2}). 
The canonical extensions relate to (geometric modular forms) and the sub-canonical extensions relate to 
cusp forms.   
Recall the algebraic representation $\lambda_{\underline{k}}$ in Section \ref{class}. For such $\lambda_{\underline{k}}$, we associate 
the automorphic (coherent) sheaf 
$$\omega_{\underline{k}}:={\rm Sym}^{k_1-k_2}\E\otimes_{\O_{\overline{S}_{K(N)}}} {\rm det}^{k_2}\E.$$ 
For any $\Z[\frac{1}{N}]$-algebra $R$, we define the space of geometric Siegel modular forms over $R$ by 
$$GM_{\underline{k}}(K(N),R):=H^0(\overline{S}_{K(N)}\otimes R,\omega_{\underline{k}}\otimes R)=H^0(S_{K(N)}\otimes R,\omega_{\underline{k}}\otimes R).$$
The last equality follows from Koecher's principle (cf. \cite{c&f}). 

Then the space of geometric Siegel cusp forms are defined by 
$$GS_{\underline{k}}(K(N),R):=H^0(\overline{S}_{K(N)}\otimes R,\omega_{\underline{k}}(-\mathcal{C})\otimes R).$$
If $R=\C$ we have $$GM_{\underline{k}}(K(N),\C)=\bigoplus_{1\le a < N\atop (a,N)=1}M_{\underline{k}}(\G(N)_a)$$
and it is as well for $GS_{\underline{k}}(K(N),\C)$. 

Consider Mumford's semi-abelian scheme 
\begin{equation}\label{mum}
\mathcal{A}^{{\rm Mum}}:=\mathbb{G}^2_m/
\langle\left(\begin{array}{c}
q_{11}\\
q_{12}
\end{array}
\right),\ 
\left(\begin{array}{c}
q_{12}\\
q_{22}
\end{array}
\right)
 \rangle_\Z
 \end{equation}
 as a principal semi-abelian scheme with level $N$-structure over $T_N:=\coprod_{1\le a < N\atop (a,N)=1}{\rm Spec}R_N$ where $R_N:=\Z[\frac{1}{N}][q^{\pm \frac{1}{N}}_{12}][[q^{\frac{1}{N}}_{11},q^{\frac{1}{N}}_{22}]][q^{-1}_{11},q^{-1}_{22}]$ 
 (cf. p. 38-39 of \cite{ichikawa1}). 
By the universality of $\mathcal{G}/\overline{S}_{K(N)}$, there exists a morphism $\iota:T_N \lra \overline{S}_{K(N)}$ such that 
$$\mathcal{A}^{{\rm Mum}}\simeq \mathcal{G}\times_{\overline{S}_{K(N)}} T_N.$$
Then $\iota^\ast \omega_{\underline{k}}$ is canonically trivialized and it is isomorphic to the constant sheaf 
$\bigoplus_{1\le a < N\atop (a,N)=1}V_{\underline{k}}(R_N)$. 
There is a natural map $ \omega_{\underline{k}}\lra \iota_\ast\iota^\ast  \omega_{\underline{k}}$ and 
by taking global sections after base change to a $\Z[\frac{1}{N}]$-algebra $R$, 
we have a map 
$$GM_{\underline{k}}(K(N),R)\lra \bigoplus_{1\le a < N\atop (a,N)=1}V_{\underline{k}}(R_N\otimes_{\Z[\frac{1}{N}]}R).$$
The image of $F\in GM_{\underline{k}}(K(N),R)$ under this map can be written as 
\begin{equation}\label{q-exp}
F(q):=\Big(\sum_{T\in \mathcal{S}(\Z)_{\ge 0}}
A_{F_a}(T)q^T_N\Big)_{1\le a < N\atop (a,N)=1}
=\sum_{T\in \mathcal{S}(\Z)_{\ge 0}}A_F(T)q^T_N,\ 
A_F(T):=(A_{F_a}(T))_{1\le a < N\atop (a,N)=1}
\end{equation}
where $q^T_N=q^{\frac{a}{N}}_{11}q^\frac{b}{N}_{12}q^\frac{c}{N}_{22}$ for $T=\left(\begin{array}{cc}
a & \frac{b}{2}\\
\frac{b}{2} & c 
\end{array}
\right)$. This expression is the $q$-expansion of $F$ and the $q$-expansion principle (cf. \cite{c&f}) tells us that the above map is injective, hence the  
coefficients $A_F(T)$ determine $F$. Furthermore, it follows from this that for any $\Z[\frac{1}{N}]$-subalgebra $R'$ of $R$ and 
$F\in GM_{\underline{k}}(K(N),R)$, $F$ belongs to $GM_{\underline{k}}(K(N),R')$ if and only if $A_F(T)\in V_{\underline{k}}(R')$ for all $T\in \mathcal{S}(\Bbb Z)_{\geq 0}$.

As in Section 1.1.6 of \cite{s&u}, we can define geometric Hecke operators $T(p^i)^{{\rm geo}},\ p\nmid N$ on $GM_{\underline{k}}(K(N),R)$ and $GS_{\underline{k}}(K(N),R)$, which coincide with, the Hecke 
operators $T(p^i)$ on classical Siegel modular forms if  $R=\C$ (see (1.1.6.a) of \cite{s&u}).  
Then, using (\ref{hecke-fourier}) and (\ref{q-exp}), we have the similar formula for the Fourier coefficients of $T(p^i)^{{\rm geo}}F$ in terms of $A_{F_a}(T)$ and hence of $A_F(T)$. 
We often identify $T(p^i)^{{\rm geo}}$ with $T(p^i)$. 

The following theorem is due to Lan and Suh (\cite{l&s1},\cite{l&s2}). 
\begin{thm}\label{vanishing-lan}The notation being as above. Let $p\ge 5$ be a prime not dividing $N$.  
Put $\overline{S}_{N,p}:=\overline{S}_{K(N)}\otimes \bF_p$ and  $\mathcal{C}=\overline{S}_{N,p}-S_{N,p}$. 
Assume  $k_2>3$ and $p>3+(k_1-k_2)$. 
Then  
\begin{enumerate}
\item $H^i(\overline{S}_{N,p},\mathcal{F})=0$ for $\mathcal{F}\in \{\omega_{\underline{k}},\ 
\omega_{\underline{k}}(-\mathcal{C})  \}$ and $i=1,2$;
\item any element in $H^0(\overline{S}_{N,p},\omega_{\underline{k}}(-\mathcal{C}))$ is liftable to a form in characteristic zero with the same weight;
\item $H^3(\overline{S}_{N,p},\omega_{\underline{k}}\otimes \omega^{-l})=0$ for 
any $k_2-\ell-3>0$.  
\end{enumerate}
\end{thm} 
\begin{proof}The claims follow from Theorem 8.13 of \cite{l&s2} and the Serre duality  
(see also Theorem 4.1 of \cite{l&s1}). 
\end{proof}

\section{$\theta$-operators and $\theta$-cycles}\label{theta-operators}
In this section we will define various $\theta$-operators on the space of mod $p$ Siegel modular forms of 
degree 2 and then study $\theta$-cycles which are an analogue of \cite{katz},\cite{joch} in the case of elliptic modular forms.
\subsection{Ordinary locus}\label{ordinary-locus} Let us keep the notation in previous section.  
Let $S_{N,p}$ be a connected component of $S_{K(N)}\otimes \bF_p$. Let $\mathcal{A}\stackrel{f}{\lra}S_{N,p}$ be the universal 
abelian surface with the zero section $e:S_{N,p}\lra \mathcal{A}$. As in previous section, 
the Hodge vector bundle $\mathcal{E}:=e^\ast \Omega_{\mathcal{A}/S_{N,p}}=f_\ast \Omega_{\mathcal{A}/S_{N,p}}$ is a locally free sheaf on $S_{N,p}$ of 
rank 2. Its determinant $\omega:=\det(\mathcal{E})$ is an ample line bundle on  $S_{N,p}$. 
These are nothing but base changes to $S_{N,p}$ of the objects we have defined in Section \ref{geometry}. 
For any scheme $X$ over $\bF_p$, we denote by $F_{{\rm abs}}$ the absolute Frobenius map on $X$ induced by $F^\ast_{{\rm abs}}:\O_X\lra \O_X, a\mapsto a^p$. 
For any sheaf $\mathcal{F}$ on $X$, put $\mathcal{F}^{(p)}=\mathcal{F}\otimes_{\O_X,F^\ast_{{\rm abs}}}\O_X$.   
As usual let us consider the following commutative diagram: 
$$
\xymatrix{
\mathcal{A} \ar@/_/[ddr]_f \ar@{>}[dr]|{F} \ar@/^/[drr]^{F_{{\rm abs}}} \\
& \mathcal{A}^{(p)}
\ar[d]^{f^{(p)}} \ar[r]^{{\rm proj}} & \mathcal{A} \ar[d]_f \\
& S_{N,p} \ar[r]^{F_{{\rm abs}}}  & S_{N,p}} 
$$
where $F=F_{\mathcal{A}/S_{N,p}}$ stands for the relative Frobenius map where we often drop the subscript from the notation. 

By Messing (see \cite{m&m}), the Hodge to de Rham spectral 
sequence 
$$E^{s,t}_{1}=R^tf_\ast \Omega^s_{\mathcal{A}/S_{N,p}}\Longrightarrow \mathbb{H}^{s+t}_{{\rm dR}}(\mathcal{A}/S_{N,p})$$
degenerates at $E_1$ (this fact is true for any abelian scheme over any scheme). 
Applying this to $f$ and $f^{(p)}$, we have 
\begin{equation}\label{exact0}
\xymatrix{
  0 \ar[r] & f_\ast \Omega^1_{\mathcal{A}/S_{N,p}} \ar[r]^{d_1} &  \mathbb{H}^{1}_{{\rm dR}}(\mathcal{A}/S_{N,p}) 
   \ar[r]^{d_2}  & R^1f_\ast \O_{\mathcal{A}}  \ar[r] & 0 \\ 
    0 \ar[r] &  \ar[u]^{F^\ast} f^{(p)}_\ast \Omega^1_{\mathcal{A}^{(p)}/S_{N,p}} \ar[r]^{d^{(p)}_1} &  \ar[u]^{F^\ast}  \mathbb{H}^{1}_{{\rm dR}}(\mathcal{A}^{(p)}/S_{N,p}) 
   \ar[r]^{d^{(p)}_2}  &  \ar[u]^{F^\ast}  R^1f^{(p)}_\ast \O_{\mathcal{A}^{(p)}}  \ar[r] & 0 
}
\end{equation}
where $F^\ast$ stands for the pull-back of the relative Frobenius map $F$.   
By working locally on $S_{N,p}$, it easy to see that $F^\ast$ is zero on $ f^{(p)}_\ast \Omega^1_{\mathcal{A}^{(p)}/S_{N,p}}$. 
Hence the map $d^{(p)}_2$ induces a map 
\begin{equation}\label{frob1}
Fr^\ast: R^1f^{(p)}_\ast \O_{\mathcal{A}^{(p)}}\lra 
\mathbb{H}^{1}_{{\rm dR}}(\mathcal{A}/S_{N,p})
\end{equation}
and also a map 
\begin{equation}\label{hasse-matrix}
d_2\circ Fr^\ast: R^1f^{(p)}_\ast \O_{\mathcal{A}^{(p)}}\lra R^1f_\ast \O_{\mathcal{A}}
\end{equation}
This gives by Serre duality  rise to a map 
\begin{equation}\label{frob2}
\omega^{-p}  \lra  \omega^{-1}.
\end{equation}
Hence, one has a non-zero global section $H_{p-1}$ on $S_{N,p}$ of the 
sheaf ${\rm Hom}(\omega^{-p}, \omega^{-1})=
\omega^{p-1}$. The section $H_{p-1}$ is called the Hasse invariant and regared as a mod $p$ Siegel modular form 
of weight $p-1$ with ``level one'' which means that the finite group $Sp_4(\Z)/\G(N)$ 
acts trivially on $H_{p-1}$. 

We denote by $S^h_{N,p}$ the locus so that the Hasse invariant $H_{p-1}$ is non-zero. 
Then the following is known-well. 
\begin{thm}\label{locus}
The locus $S^h_{N,p}$ is a maximal subscheme so that  
the Frobenius map $(\ref{frob1})$ gives a splitting of the exact sequence $(\ref{exact0})$.  
\end{thm}
\begin{proof}
For each geometric point $x={\rm Spec}\bF_p$ of $S_{N,p}$ and the corresponding abelian variety 
$A_x$ over $\bF_p$, $H_{p-1}(x)\neq 0$ if and only if $d_2\circ Fr^\ast$ induces  
an isomorphism $H^1(A^{(p)}_x,\mathcal{O}_{A^{(p)}_x})\lra H^1(A_x,\mathcal{O}_{A_x})$ on 
the fiber at $x$.  
\end{proof}

To end this subsection we recall the de Rham pairing. 
Since $\mathcal{A}$ is endowed with a principal polarization $\lambda:\mA\stackrel{\sim}{\lra} \mA^\vee$, there is 
a canonical isomorphism 
\begin{equation}\label{canonical-isom} 
\mathbb{H}^{1}_{{\rm dR}}(\mathcal{A}/S_{N,p})\simeq \mathbb{H}^{1}_{{\rm dR}}(\mathcal{A}^\vee/S_{N,p})
\simeq \mathbb{H}^{1}_{{\rm dR}}(\mathcal{A}/S_{N,p})^\vee. 
\end{equation}  
As for the second isomorphism, the map induced by $\lambda$ is a 
fiber-wise isomorphism by Section 5.1 of \cite{BBM} and it explains the claim. 
This induces a canonical alternating  pairing 
\begin{equation}\label{pairing}
\langle\ ,\ \rangle_{{\rm dR}}: 
\mathbb{H}^{1}_{{\rm dR}}(\mathcal{A}/S_{N,p}) \times 
\mathbb{H}^{1}_{{\rm dR}}(\mathcal{A}/S_{N,p}) \simeq 
\mathbb{H}^{1}_{{\rm dR}}(\mathcal{A}/S_{N,p}) \times 
\mathbb{H}^{1}_{{\rm dR}}(\mathcal{A}/S_{N,p})^\vee \lra \O_{S_{N,p}}
\end{equation}  
Let us consider 
\begin{equation}\label{exact1}
\xymatrix{
  0 \ar[r] & \ar[d] f_\ast \Omega^1_{\mathcal{A}/S_{N,p}} \ar[r]^{d_1} & \ar[d]^\simeq_{{\rm using}\atop \langle,\rangle_{{\rm dR}}}  \mathbb{H}^{1}_{{\rm dR}}(\mathcal{A}/S_{N,p}) 
   \ar[r]^{d_2}  & \ar[d] R^1f_\ast \O_{\mathcal{A}}  \ar[r] & 0 \\ 
    0  & \ar[l](f_\ast \Omega^1_{\mathcal{A}/S_{N,p}})^\vee  & \ar[l]_{d^\vee_1} 
\mathbb{H}^{1}_{{\rm dR}}(\mathcal{A}/S_{N,p})^\vee 
     &\ar[l]_{d^\vee_2} (R^1f_\ast \O_{\mathcal{A}})^\vee  & \ar[l] 0. 
}
\end{equation}  
Then by chasing diagram, the left vertical map is zero, namely for any local sections $\w_1,\w_2\in f_\ast \Omega^1_{\mathcal{A}/S_{N,p}}$, one has $\langle\w_1 ,\w_2 \rangle_{{\rm dR}}=0$.    
Then one has a canonical isomorphism  
\begin{equation}\label{duality}
R^1f_\ast \O_{\mathcal{A}}\stackrel{d_2\atop \sim}{\longleftarrow}
\mathbb{H}^{1}_{{\rm dR}}(\mathcal{A}/S_{N,p})/f_\ast \Omega^1_{\mathcal{A}/S_{N,p}}\stackrel{\sim}{\lra}
\mathbb{H}^{1}_{{\rm dR}}(\mathcal{A}/S_{N,p})/(R^1f_\ast \O_{\mathcal{A}})^\vee\stackrel{d^\vee_1\atop \sim}{\lra}
   (f_\ast \Omega^1_{\mathcal{A}/S_{N,p}})^\vee.
\end{equation} 
 This gives rise to a canonical non-degenerate pairing 
\begin{equation}\label{pairing2}
\langle\ ,\ \rangle_{{\rm dual}}: 
f_\ast \Omega^1_{\mathcal{A}/S_{N,p}}\times R^1f_\ast \O_{\mathcal{A}} \lra \O_{S_{N,p}}.
\end{equation}   
For any local section $\w_1\in  f_\ast \Omega^1_{\mathcal{A}/S_{N,p}}$ and a lift 
$\widetilde{\eta}_1\in \mathbb{H}^{1}_{{\rm dR}}(\mathcal{A}/S_{N,p})$ of its dual $\eta_1\in R^1f_\ast \O_{\mathcal{A}}$ via 
Serre duality (namely $\langle\w_1 ,\eta_1 \rangle_{{\rm dual}}=1$).  Then we have  $\langle \w_1 , \widetilde{\eta}_1\rangle_{{\rm dR}}=\langle\w_1 ,\eta_1 \rangle_{{\rm dual}}=1.$  

Take a local basis $\omega_1,\omega_2\in  f_\ast \Omega^1_{\mathcal{A}/S_{N,p}}$ and the 
corresponding basis 
$\eta_1,\eta_2\in R^1f_\ast \O_{\mathcal{A}}$ via Serre duality. 
Fix lifts $\widetilde{\eta}_i,i=1,2$ of $\eta_i$ to $ \mathbb{H}^{1}_{{\rm dR}}(\mathcal{A}/S_{N,p})$. 
Let us denote by $\eta^{(p)}_i$ the image of $\eta_i\otimes 1$ under 
the Frobenius semilinear map 
\begin{equation}\label{pull-back-frob}
F^\ast_{{\rm abs}} \mathbb{H}^{1}_{{\rm dR}}(\mathcal{A}^{(p)}/S_{N,p})=
 \mathbb{H}^{1}_{{\rm dR}}(\mathcal{A}/S_{N,p})^{(p)}\lra  \mathbb{H}^{1}_{{\rm dR}}(\mathcal{A}^{(p)}/S_{N,p}).
\end{equation} 
Clearly $\eta^{(p)}_i\in R^1f_\ast \O_{\mathcal{A}^{(p)}}$ for $i=1,2$ and they make up a basis. Then one has 
\begin{equation}\label{frob-action1}
Fr^\ast(\eta^{(p)}_1,\eta^{(p)}_2)=(\omega_1,\omega_2,\widetilde{\eta}_1,\widetilde{\eta}_2)
\left(
\begin{array}{c}
B\\
A
\end{array}
\right),\ 
B=\left(
\begin{array}{cc}
b_{11} & b_{12} \\
b_{21} & b_{22} 
\end{array}
\right),\ 
A=\left(
\begin{array}{cc}
a_{11} & a_{12} \\
a_{21} & a_{22} 
\end{array}
\right)\in M_2(\O_{S_{N,p}}). 
\end{equation}
The matrix $A$ also satisfies $d_2\circ Fr^\ast(\eta^{(p)}_1,\eta^{(p)}_2)=(\eta_1,\eta_2)A$ by functoriality of the 
Hodge filtration (see (\ref{frob1}) for $d_2\circ Fr^\ast$).

\subsection{The Gauss-Manin connection}\label{Gauss-Manin1}
In this subsection we will recall the basic facts for the Gauss-Manin connection and the Kodaira-Spencer map 
in our setting. We will try to make everything as explicit as possible.   
We refer  to \cite{katz2},\cite{kob} for basic references.  

For our smooth family $\mathcal{A}/S_{N,p}$, one can associate an integrable connection which is so called 
the Gauss-Manin connection:
$$\nabla:\mathbb{H}^{1}_{{\rm dR}}(\mathcal{A}/S_{N,p})\lra \mathbb{H}^{1}_{{\rm dR}}(\mathcal{A}/S_{N,p})
\otimes_{\O_{S_{N,p}}}\Omega^1_{S_{N,p}}.$$
It is well-known that, for example, see \cite{hartshorne} that $Hom_{\O_{S_{N,p}}}(\Omega^1_{S_{N,p}},\O_{S_{N,p}})=
(\Omega^1_{S_{N,p}})^\vee$ is  naturally identified with 
 the sheaf of derivatives on $\O_{S_{N,p}}$ which is denoted by $Der(\O_{S_{N,p}})$.  
For a given local section $D$ of $Der(\O_{S_{N,p}})$ one has 
\begin{equation}\label{derivative}
\xymatrix{
\mathbb{H}^{1}_{{\rm dR}}(\mathcal{A}/S_{N,p})\ar[r]^{\nabla\hspace{12mm}}  \ar@/_/[dr]_{\nabla_D} 
&\ar[d]^{1\otimes D}  \mathbb{H}^{1}_{{\rm dR}}(\mathcal{A}/S_{N,p})\otimes_{\O_{S_{N,p}}}\Omega^1_{S_{N,p}}  \\
&  \mathbb{H}^{1}_{{\rm dR}}(\mathcal{A}/S_{N,p}).
} 
\end{equation}
Since $\nabla$ is integrable, for any $D_1,D_2\in Der(\O_{S_{N,p}})$ satisfying $D_1D_2=D_2D_1$, one has 
$\nabla_{D_1}\nabla_{D_2}=\nabla_{D_2}\nabla_{D_1}$ on $\mathbb{H}^{1}_{{\rm dR}}(\mathcal{A}/S_{N,p})$. 

Choose a local coordinate $x_{11},x_{12},x_{22}$ of $S_{N,p}$, then locally 
$\Omega^1_{S_{N,p}}=\langle dx_{11},dx_{12},dx_{22}\rangle_{\O_{S_{N,p}}}$. Take the dual local basis 
$\ds\frac{\partial}{\partial x_{11}},\frac{\partial}{\partial x_{12}},\frac{\partial}{\partial x_{22}}$ of $Der(\O_{S_{N,p}})$.  
It follows from this that 
\begin{equation}\label{local-exp}
\nabla(m)=\sum_{1\le i\le j\le 2}\nabla(\frac{\partial}{\partial x_{ij}})(m)dx_{ij}\ {\rm for}\ m\in \mathbb{H}^{1}_{{\rm dR}}(\mathcal{A}/S_{N,p}).
\end{equation} 
One would use another basis $d_{11},d_{12},d_{22}$ of $\Omega^1_{S_{N,p}}$ and their dual basis 
$D_{11},D_{12},D_{22}$ of $Der(\O_{S_{N,p}})$. 
Then one would see locally 
$$\nabla(m)=\ds\sum_{1\le i\le j\le 2}\nabla_{D_{ij}}(m)d_{ij},\ m\in \mathbb{H}^{1}_{{\rm dR}}(\mathcal{A}/S_{N,p}).$$ However,  
the commutativity of $D_{ij}$'s does not necessarily hold in general.  

Recall the Kodaira-Spencer isomorphism (as $O_{S_{N,p}}$-modules)
\begin{equation}\label{ks}
KS:{\rm Sym}^2\E\stackrel{\sim}{\lra}\Omega^1_{S_{N,p}}
\end{equation}
which has the property that for any local basis $\omega_1,\omega_2$ of $\E$, 
\begin{equation}\label{ks1}
\left\{\begin{array}{rl}
KS(\omega^{\otimes 2}_1)&=\langle \omega_1,\nabla\omega_1 \rangle_{{\rm dR}},\\
KS(\frac{1}{2}(\omega_1\otimes\omega_2+\omega_2\otimes\omega_1))&=
\langle \omega_1,\nabla\omega_2 \rangle_{{\rm dR}},\\ 
KS(\omega^{\otimes 2}_2)&=\langle \omega_2,\nabla\omega_2 \rangle_{{\rm dR}}. 
\end{array}\right.
\end{equation}
Let us explain how we gain (\ref{ks}) and (\ref{ks1}). For each local section $D$ of $Der(\O_{S_{N,p}})$, one can associate 
a  
$\O_{S_{N,p}}$-module homomorphism $\psi_D$ :
$$
\xymatrix{
f_\ast \Omega^1_{\mathcal{A}/S_{N,p}}\ar[r]^{\nabla_{D}\hspace{5mm}}  \ar@/_/[drrr]_{\psi_D} &\ar[r]^{(\ref{exact1})\atop \simeq} 
\mathbb{H}^{1}_{{\rm dR}}(\mathcal{A}/S_{N,p}) & \ar[r]^{d^\vee_1} 
\mathbb{H}^{1}_{{\rm dR}}(\mathcal{A}/S_{N,p})^\vee &\ar[d]^{\simeq}_{(\ref{duality})} (f_\ast \Omega^1_{\mathcal{A}/S_{N,p}})^\vee  \\
&  & &  R^1f_\ast \O_{\mathcal{A}}.
} 
$$
Thus we have 
\begin{equation}\label{ks2}
Der(\O_{S_{N,p}})\lra Hom_{\O_{S_{N,p}}}(f_\ast \Omega^1_{\mathcal{A}/S_{N,p}},R^1f_\ast \O_{\mathcal{A}}),D\mapsto \psi_D.
\end{equation}
We now study the image of this map. Let us fix a local basis $\w_1,\w_2$ of $\E=f_\ast \Omega^1_{\mathcal{A}/S_{N,p}}$ and take 
the local basis $\eta_1,\eta_2$ of $R^1f_\ast \O_{\mathcal{A}}$ via the duality (\ref{duality}). We also 
fix lifts $\widetilde{\eta}_i,\ i=1,2$ of $\eta_i$ to $\mathbb{H}^{1}_{{\rm dR}}(\mathcal{A}/S_{N,p})$. 
For $1\le i,j\le 2$, we denote by $t_{ij}$ the morphism sending $\omega_i$ to $\eta_j$ and $\omega_{i'},i'\not=i$ to 0. 
Then $\{t_{ij}\}_{1\le i,j\le 2}$ makes up a local basis of the RHS of (\ref{ks2}). On the other hand, if we write 
\begin{equation}\label{action1}
(\nabla_{D}\w_1,\nabla_{D}\w_2)=(\omega_1,\omega_2,\widetilde{\eta}_1,\widetilde{\eta}_2)
\left(
\begin{array}{c}
\ast\\
A_D
\end{array}
\right),\  
A_D=\left(
\begin{array}{cc}
a^{D}_{11} & a^{D}_{12} \\
a^{D}_{21} & a^{D}_{22} 
\end{array}
\right)\in M_2(\O_{S_{N,p}}), 
\end{equation}
then by Leibniz rule and using $\langle \omega_1,\omega_2 \rangle_{{\rm dR}}=0$,   
$$\begin{array}{rl}
0=& \nabla_{D}\langle \omega_1,\omega_2 \rangle_{{\rm dR}}=\langle  \nabla_{D}\omega_1,\omega_2 \rangle_{{\rm dR}}+
\langle \omega_1, \nabla_{D}\omega_2 \rangle_{{\rm dR}}\\ 
=&-a^{D}_{21}+a^{D}_{12}. 
\end{array}
$$ 
It follows from this that $\psi_D(\omega_i)=\ds\sum_{j=1,2}a^{D}_{ij}\eta_j$ and therefore we have 
$$\psi_D=a^{D}_{11}t_{11}+a^{D}_{12}(t_{12}+t_{21})+a^{D}_{22}t_{22}.$$

As seen before, if we work locally on $S_{N,p}$ and for $1\le k\le l\le 2$, put $a^{(kl)}_{ij}:=
a^{D_{kl}}_{ij}$ for the dual basis $D_{kl}$ of $d_{kl}$, then 
\begin{equation}\label{local-exp2}
\begin{array}{l}
\langle \omega_1,\nabla \w_1 \rangle_{{\rm dR}}=\ds\sum_{1\le k\le l\le 2}a^{(kl)}_{11}d_{kl},\\
\langle \omega_1,\nabla \w_2 \rangle_{{\rm dR}}=\ds\sum_{1\le k\le l\le 2}a^{(kl)}_{12}d_{kl}=
\langle \omega_2,\nabla \w_1 \rangle_{{\rm dR}},\\
\langle \omega_1,\nabla \w_1 \rangle_{{\rm dR}}=\ds\sum_{1\le k\le l\le 2}a^{(kl)}_{22}d_{kl}.
\end{array}
\end{equation}
 
The images of $t_{11},(t_{12}+t_{21}),t_{22}$ under a canonical identification 
\begin{equation}\label{ks3}
\begin{array}{c}
Hom_{\O_{S_{N,p}}}(f_\ast \Omega^1_{\mathcal{A}/S_{N,p}},R^1f_\ast \O_{\mathcal{A}})=
(f_\ast \Omega^1_{\mathcal{A}/S_{N,p}})^\vee\otimes_{\O_{S_{N,p}}}R^1f_\ast \O_{\mathcal{A}}\\
\stackrel{(\ref{duality})\atop \sim}{\lra} 
(f_\ast \Omega^1_{\mathcal{A}/S_{N,p}})^\vee\otimes_{\O_{S_{N,p}}}(f_\ast \Omega^1_{\mathcal{A}/S_{N,p}})^\vee= 
\E^\vee\otimes_{\O_{S_{N,p}}} \E^\vee  
\end{array}
\end{equation}
are given by $\eta^\vee_1\otimes \eta^\vee_1, (\eta^\vee_1\otimes \eta^\vee_2+\eta^\vee_2\otimes \eta^\vee_1),\ 
\eta^\vee_2\otimes \eta^\vee_2$ respectively where $\eta^\vee_i$ stands for the dual element corresponding to $\eta_i$ via (\ref{duality}). 
Let us introduce the formal symbol  $(Sym^2\E)^\ast$ to  be the $\O_{S_{N,p}}$-submodule of $\E\otimes_{\O_{S_{N,p}}} \E$ generated by 
$\eta^\vee_1\otimes \eta^\vee_1, (\eta^\vee_1\otimes \eta^\vee_2+\eta^\vee_2\otimes \eta^\vee_1),\ 
\eta^\vee_2\otimes \eta^\vee_2$. 
Composing (\ref{ks2}), (\ref{ks3}) and taking dual, one has the Kodaira-Spencer map 
$${\rm Sym}^2\E\simeq ((Sym^2\E)^\ast)^\vee \lra  (Der(\O_{S_{N,p}}))^\vee=((\Omega^1_{S_{N,p}})^\vee)^\vee=\Omega^1_{S_{N,p}}.$$
Since the formation is compatible with base change, by working over $\C$ (see \cite{harris}), 
this map gives an isomorphism (see also Section 9 in Chapter III of \cite{c&f}). 
 
Then it follows from this that 
$$\begin{array}{rl}
(\eta^\vee_1\otimes \eta^\vee_1)^\vee &\mapsto  \langle \w_1,\nabla \w_1 \rangle_{{\rm dR}} \\
(\eta^\vee_1\otimes \eta^\vee_2+\eta^\vee_2\otimes \eta^\vee_1)^\vee &\mapsto 
\langle \w_1,\nabla \w_2 \rangle_{{\rm dR}}+\langle \w_2,\nabla \w_1 \rangle_{{\rm dR}}=2\langle \w_1,\nabla \w_2 \rangle_{{\rm dR}}  \\ 
(\eta^\vee_2\otimes \eta^\vee_2)^\vee &\mapsto \langle \w_2,\nabla \w_2 \rangle_{{\rm dR}}. 
\end{array}
.$$
This explains exactly how we obtained the Kodaira-Spencer map (\ref{ks}). 
\subsection{Small $\theta$-operators and a big $\Theta$ operator}\label{theta}
In this subsection we will introduce the ``small" $\theta$-operators as in \cite{katz}. 
Contrary to the case $GL_2/\Q$ (the case of elliptic modular forms), there are various kinds of theta operators which will be 
denoted by $\theta, \theta_1,\theta_2,\theta_3$ and applying $\theta,\theta_1$ we will 
also define the ``big" $\Theta$ operator which is a characteristic $p$ geometric counter part of Bocherer-Nagaoka's operator 
in the classical case (see \cite{b&n}). 
Throughout this section we always assume $p\ge 5$. Let us keep the notation in Section \ref{Gauss-Manin1}. 

Let $\w_1,\w_2$ be a local basis of $\E$. 
By Kodaira-Spencer isomorphism, $\langle \w_i,\nabla \w_j \rangle_{{\rm dR}},\ 
1\le i\le j\le 2$ 
make up a local basis of $\Omega^1_{S_{N,p}}$. We denote by $D_{ij}$ the dual of $\langle \w_i,\nabla \w_j \rangle_{{\rm dR}}$. 
Then by (\ref{local-exp2}) we have  
\begin{equation}\label{value} 
\langle \w_i,\nabla_{D_{k\ell}}(\w_j) \rangle_{{\rm dR}}=\delta_{ik}\delta_{\ell j}\cdot 1 
\ {\rm for}\  
1\le i\le j\le 2\ {\rm and}\  1\le k\le \ell\le 2,
\end{equation}
where $1$ stands for the multiplicative identity of $\mathcal{O}_{S_{N,p}}$ and 
$\delta_{ij}$ is the Kronecker's delta. 

\begin{lem}\label{omega}Keep the notation above. 
Then the following facts hold. 
\begin{enumerate}
\item $ \nabla_{D_{11}}(\w_2)= \nabla_{D_{22}}(\w_1)=0$,
\item $ \nabla_{D_{11}}(\w_1)=\nabla_{D_{12}}(\w_2)$ and 
$\nabla_{D_{22}}(\w_2)=\nabla_{D_{12}}(\w_1)$,
\item The elements $\omega_1,\omega_2, \nabla_{D_{11}}(\w_1), 
\nabla_{D_{22}}(\w_2)$ make up a local basis of 
$\mathbb{H}^{1}_{{\rm dR}}(\mathcal{A}_{\mathcal{X}}/\mathcal{X})$. 
\end{enumerate}
\end{lem}
\begin{proof}We first show (1). Applying (\ref{action1}), (\ref{local-exp2}), 
and (\ref{value}) for $D_{11}$, one has 
$A_{D_{11}}=
\left(\begin{array}{cc}
1 & 0 \\
0 & 0 
\end{array}
\right)$. Hence $\langle \nabla_{D_{11}}(\w_2), \widetilde{\eta}_i \rangle_{{\rm dR}}=0$ for $i=1,2$. 
Further by (\ref{value}),  $\langle \nabla_{D_{11}}(\w_2),\w_i \rangle_{{\rm dR}}=0$ for $i=1,2$. 
Then it follows from non-degeneracy of the de Rham pairing that 
$\nabla_{D_{11}}(\w_2)=0$. 
Similarly, one has  $\nabla_{D_{22}}(\w_1)=0$. 

Next we prove (2). By (\ref{value}), clearly  
$\langle \nabla_{D_{11}}(\w_1)-\nabla_{D_{12}}(\w_2),\w_i \rangle_{{\rm dR}}=0$ for $i=1,2$. 
Applying (\ref{action1}) for $D_{12},D_{22}$, one has 
$A_{D_{12}}=
\left(\begin{array}{cc}
0 & 1 \\
1 & 0 
\end{array}
\right),\ 
A_{D_{22}}=
\left(\begin{array}{cc}
0 & 0 \\
0 & 1 
\end{array}
\right)$. Thus we see that 
$$\langle \nabla_{D_{11}}(\w_1)-\nabla_{D_{12}}(\w_2),\widetilde{\eta}_1 \rangle_{{\rm dR}}=1-1=0,\ 
\langle \nabla_{D_{11}}(\w_1)-\nabla_{D_{12}}(\w_2),\widetilde{\eta}_1 \rangle_{{\rm dR}}=0-0=0.$$
The second equality is the same as above. 

For the last claim, let us assume 
$a_1\omega_1+a_2\omega_2+a_3\nabla_{D_{11}}(\w_1)+a_4\nabla_{D_{22}}(\w_2)=0, a_i\in \O_{S_{N,p}}$. 
First we apply $\langle \ast,\w_i \rangle_{{\rm dR}},\ i=1,2$ and next do  
$\langle \ast,\widetilde{\eta}_i \rangle_{{\rm dR}},\ i=1,2$. Then one has 
$a_1=a_2=a_3=a_4=0$. 
\end{proof}
By using the basis of Lemma \ref{omega}-(3), as in (\ref{frob-action1}) we consider 
\begin{equation}\label{ab}
\begin{array}{c}
F^\ast((\nabla_{D_{11}}(\w_1))^{(p)},(\nabla_{D_{22}}(\w_2))^{(p)})=(\omega_1,\omega_2,\nabla_{D_{11}}(\w_1),\nabla_{D_{22}}(\w_2))
\left(
\begin{array}{c}
B\\
A
\end{array}
\right),\\
 B=\left(
\begin{array}{cc}
b_{11} & b_{12} \\
b_{21} & b_{22} 
\end{array}
\right),\ 
A=\left(
\begin{array}{cc}
a_{11} & a_{12} \\
a_{21} & a_{22} 
\end{array}
\right)\in M_2(\O_{S_{N,p}}). 
\end{array}
\end{equation}
Here we use $F^\ast$ not $Fr^\ast$. 
As in (\ref{action1}), let us fix lifts $\phi_i,\ i=1,2$ of $\eta_i$ (here we use the symbol $\phi_i$ instead of $\widetilde{\eta}_i$). 
Since $\langle \w_i,\phi_i \rangle_{{\rm dR}}=1,\ i=1,2$ and (\ref{value}), we see that 
$A$ is the Hasse-Matrix with respect to $\eta_1,\eta_2$ and we have 
\begin{equation}\label{important}
(w_1,w_2)B+(\nabla_{D_{11}}(\w_1),\nabla_{D_{22}}(\w_2))A=
(\phi_1,\phi_2)A. 
\end{equation}
Put $$\nabla_{11}:=\nabla_{D_{11}},\ \nabla_{12}:=\nabla_{D_{12}},\ \nabla_{22}:=\nabla_{D_{22}}$$
for simplicity. These are not necessarily commutative even though $\nabla$ is integrable. 
  
The relation between $A$ and $B$ is revealed by the following lemma:
\begin{lem}\label{vanishing}Let $S$ be a scheme over $\bF_p$ and 
let $f:\mathcal{A}\lra S$ be an abelian scheme.  
Let $\nabla^{(p)}$ be the Gauss-Manin connection associated to $f^{(p)}:\mathcal{A}^{(p)}\lra S$. 
Then for any local section $h$ of $\mathbb{H}^{1}_{{\rm dR}}(\mathcal{A}/S)$ and any  $D\in Der(\O_{S})$, we have 
$$\nabla_{D}(F^\ast(h^{(p)}))=0.$$
\end{lem}
\begin{proof}
By the compatibility of the Gauss-Manin connection with the relative Frobenius map $F$ and the base change $F_{{\rm abs}}$, one has the following commutative diagram:
$$
\xymatrix{
\ar[d]^{F^\ast_{{\rm abs}}\nabla^{(p)}}   \mathbb{H}^{1}_{{\rm dR}}(\mathcal{A}/S)^{(p)}  
\ar[r]^{h\otimes 1\mapsto h^{(p)}} & \ar[d]^{\nabla^{(p)}} 
\mathbb{H}^{1}_{{\rm dR}}(\mathcal{A}^{(p)}/S) 
   \ar[r]^{F^\ast}  & \ar[d]^{\nabla} \mathbb{H}^{1}_{{\rm dR}}(\mathcal{A}/S)    \\ 
\ar[d]^{1\otimes D^{(p)}}  \mathbb{H}^{1}_{{\rm dR}}(\mathcal{A}/S)^{(p)}
\otimes_{\O_{S^{(p)}}}\Omega^1_{S^{(p)}}  \ar[r]^\alpha &\ar[d]^{1\otimes D}   \mathbb{H}^{1}_{{\rm dR}}(\mathcal{A}^{(p)}/S)
\otimes_{\O_{S}}\Omega^1_{S}  
   \ar[r]^{F^\ast}  &\ar[d]^{1\otimes D}     \mathbb{H}^{1}_{{\rm dR}}(\mathcal{A}/S)\otimes_{\O_{S}}\Omega^1_{S} \\
    \mathbb{H}^{1}_{{\rm dR}}(\mathcal{A}/S)^{(p)}  \ar[r] &
\mathbb{H}^{1}_{{\rm dR}}(\mathcal{A}^{(p)}/S) 
   \ar[r]^{F^\ast}  &  \mathbb{H}^{1}_{{\rm dR}}(\mathcal{A}/S) 
}
$$
where three left horizontal arrows are defined as in (\ref{pull-back-frob}). 
Let us observe the left middle  horizontal arrow (which has been denoted by $\alpha$). Since 
the map $\Omega^1_{S^{(p)}}\lra \Omega^1_{S}$ defining $\alpha$ is locally given by $df\mapsto df^p=pf^{p-1}df=0$,  
it vanishes, so does the map $\alpha$. 
The element $\nabla_{D}(F^\ast(h^{(p)}))$ has been factoring through $\alpha$. 
Hence this gives us the claim. 
\end{proof}

\begin{prop}\label{ab-lemma}
The following relations for $A$ and $B$ hold:
\begin{enumerate}
\item $$\begin{array}{lll}
\nabla_{11}(a_{11})=-b_{11} & \nabla_{12}(a_{11})=-b_{21} & \nabla_{22}(a_{11})=0\\
\nabla_{11}(a_{12})=-b_{12} & \nabla_{12}(a_{12})=-b_{22} & \nabla_{22}(a_{12})=0\\
\nabla_{11}(a_{21})=0 & \nabla_{12}(a_{21})=-b_{11} & \nabla_{22}(a_{21})=-b_{21}\\
\nabla_{11}(a_{22})=0 & \nabla_{12}(a_{22})=-b_{12} & \nabla_{22}(a_{22})=-b_{22}, 
\end{array}
$$
\item for any $1\le i,j\le 2$ and $1\le k<l\le 2$, $\nabla_{kl}(b_{ij})=0$ if 
all $\nabla_{ij}$'s commute with each other. 
\end{enumerate}
\end{prop}
\begin{proof}
By Leibniz rule and Lemma \ref{vanishing}, for $1\le k< l\le 2,\ 1\le i,j\le 2$, 
$$\nabla_{kl}\langle F^\ast((\nabla_{ii}\w_i)^{(p)}),\w_j \rangle_{{\rm dR}}=
\langle \nabla_{kl}F^\ast((\nabla_{ii}\w_i)^{(p)}),\w_j \rangle_{{\rm dR}}+
\langle F^\ast((\nabla_{ii}\w_i)^{(p)}),\nabla_{kl}\w_j \rangle_{{\rm dR}}$$
$$=\langle F^\ast((\nabla_{ii}\w_i)^{(p)}),\nabla_{kl}\w_j \rangle_{{\rm dR}}.$$
It follows from this that we have $3\times 4=12$ relations in the first claim. 

The second claim follows from the first one with the commutativity of the differentials. For instance, 
$$\nabla_{11}(b_{11})=-\nabla_{11}\nabla_{12}(a_{21})=-\nabla_{12}\nabla_{11}(a_{21})=0$$
and  
$$\nabla_{12}(b_{12})=-\nabla_{12}\nabla_{11}(a_{12})=\nabla_{11}(b_{22})=-\nabla_{11}\nabla_{22}(a_{22})=
-\nabla_{22}\nabla_{11}(a_{22})=0.$$ 
The remaining cases are left to readers. 
\end{proof}

\subsubsection{Scalar valued case}
Fix a local basis $\omega_1,\omega_2$ of $\E$. 
Recall the Hodge filtration $\E=f_\ast\Omega^1_{\mathcal{A}/S_{N,p}}\subset \mathbb{H}^{1}_{{\rm dR}}(\mathcal{A}/S_{N,p})$ and $\w=\det(\E)$. 
Since we work locally on $S^h_{N,p}$, by Theorem \ref{locus}, the Frobenius map splits the Hodge filtration. 
Hence we have obtained a canonical decomposition $\mathbb{H}^{1}_{{\rm dR}}(\mathcal{A}/S_{N,p})=\E\oplus U,\ U={\rm Im}(Fr^\ast).$
We denote by $R(U)=R_{\underline{k}}(U)$ for $\underline{k}=(k_1,k_2)$ 
a canonical direct factor, which is induced from one of 
$\mathbb{H}^{1}_{{\rm dR}}(\mathcal{A}/S_{N,p})$ as above, so that 
$${\rm Sym}^{k_2}(\bigwedge^2 \mathbb{H}^{1}_{{\rm dR}}(\mathcal{A}/S_{N,p}))\otimes_{\O_{S_{N,p}}}
{\rm Sym}^{k_1-k_2}\mathbb{H}^{1}_{{\rm dR}}(\mathcal{A}/S_{N,p}))=
\w_{\underline{k}}\oplus R(U)$$
which is induced from the above decomposition of 
$\mathbb{H}^{1}_{{\rm dR}}(\mathcal{A}/S_{N,p})$.  
For simplicity put $\mathbb{H}^{1}_{{\rm dR}}:=\mathbb{H}^{1}_{{\rm dR}}(\mathcal{A}/S_{N,p})$ and write ${\rm Sym}^k(\mathbb{H}^{1}_{{\rm dR}})={\rm Sym}^k(\E)
\oplus R'_{k}$ with $R'_{k}=R_{(k,0)}(U)$ for $k\ge 1$.  
Then we have an operator $\widetilde{\theta}$ which is given by 
\begin{equation}\label{scalar-case}
\xymatrix{
\w^{\otimes k}=\w_{(k,k)}  \ar[r]^{\hookrightarrow\hspace{5mm}}  \ar@/_/[ddrr]_{\widetilde{\theta}} & 
{\rm Sym}^k(\bigwedge^2 \mathbb{H}^{1}_{{\rm dR}})\ar[r]^{\nabla\hspace{12mm}}&  
{\rm Sym}^k(\bigwedge^2 \mathbb{H}^{1}_{{\rm dR}})\otimes_{\O_{S_{N,p}}}\Omega^1_{S_{N,p}}
\ar[d]^{KS^{-1}}   \\
&  &\ar[d]^{\mod R_{(k,k)}(U)}   {\rm Sym}^k(\bigwedge^2 \mathbb{H}^{1}_{{\rm dR}})\otimes_{\O_{S_{N,p}}}{\rm Sym}^2\E\\ 
&  &   \w^{\otimes k}\otimes_{\O_{S_{N,p}}}{\rm Sym}^2\E. 
} 
\end{equation}
Similarly we define operators $\widetilde{\theta}_i,i=1,2,3$: 
\begin{equation}\label{vector-case1}
\xymatrix{
&  \w^{\otimes k}\otimes_{\O_{S_{N,p}}}{\rm Sym}^2\E  \ar[r]^{\hookrightarrow\hspace{10mm}}  \ar@/_/[dddl]_{\widetilde{\theta}_1}
\ar@/_/[ddd]_{\widetilde{\theta}_2}  \ar@/_/[dddr]_{\widetilde{\theta}_3}   
&{\rm Sym}^k(\bigwedge^2 \mathbb{H}^{1}_{{\rm dR}})\otimes 
{\rm Sym}^2(\mathbb{H}^{1}_{{\rm dR}})\ar[d]^{KS^{-1}\circ \nabla}    \\
&  &\ar[d]^{\mod R_{(k,k)}(U)\otimes R'_2}   {\rm Sym}^k(\bigwedge^2 \mathbb{H}^{1}_{{\rm dR}})\otimes 
{\rm Sym}^2(\mathbb{H}^{1}_{{\rm dR}})\otimes_{\O_{S_{N,p}}}{\rm Sym}^2\E \ \  \\ 
&  & \ar[dll]_{p_1}  \ar[dl]^{p_2} \ar[d]^{p_3}\hspace{14mm}  \w^{\otimes k}\otimes_{\O_{S_{N,p}}}{\rm Sym}^2\E\otimes_{\O_{S_{N,p}}}{\rm Sym}^2\E  \\
 \w^{\otimes k+2} & \w^{\otimes k+1}\otimes_{\O_{S_{N,p}}} {\rm Sym}^2\E  & \w^{\otimes k}\otimes_{\O_{S_{N,p}}}{\rm Sym}^4\E
} 
\end{equation}
where the three arrows $p_1,p_2,p_3$ are given by an explicit, but non-canonical direct decomposition 
\begin{equation}\label{decomp}
{\rm Sym}^2\E\otimes_{\O_{S_{N,p}}}{\rm Sym}^2\E=
\w^{\otimes 2} \oplus (\w\otimes_{\O_{S_{N,p}}}{\rm Sym}^2\E) \oplus  
{\rm Sym}^4\E
\end{equation}
as ${\rm Aut}_{\O_{S_{N,p}}}(\E)$-modules. 
We can globally construct this decomposition by using 
Schur functors (cf. Section 2.2.4 of \cite{EFGMM}). 
Further each component is characterized by a highest weight vector. 
Therefore, we may work locally on $S_{N,p}$ to understand the global 
decomposition (\ref{decomp}) explicitly. 

Fix a local basis $e_1,e_2$ of $\E$. Put $u_{i}=e^{i}_1e^{2-i}_2,i=0,1,2$ 
with the convention $e^0_1=e^0_2:=1$ 
 which make up a local basis of 
${\rm Sym}^2\E$. To avoid confusion, we define additional symbols  $v_{i}=e^{i}_1e^{2-i}_2,i=0,1,2$ which 
play the same role. Then $\{u_i\otimes v_j\}_{0\le i,j\le 2}$ gives a local basis of 
${\rm Sym}^2\E\otimes_{\O_{S_{N,p}}}{\rm Sym}^2\E$. 
As in Appendix A, for any $x=\ds\sum_{0\le i,j\le 2}a_{ij}u_i\otimes v_j$ we define 
\begin{equation}\label{decom-explicit}
\begin{array}{rl}
p_1(x)&=\Big(\ds\frac{1}{3}a_{20}-\frac{1}{6}a_{11}+\frac{1}{3}a_{02}\Big)
(e_1\wedge e_2)^2 \\
p_2(x)&=\ds\frac{1}{2}(a_{21}-a_{12})e^{2}_1(e_1\wedge e_2)+(a_{20}-a_{02})e_1e_2(e_1\wedge e_2)
+\frac{1}{2}(a_{10}-a_{01})e^{2}_2(e_1\wedge e_2) \\
 p_3(x)&=a_{22}e^4_1+(a_{21}+a_{12})e^3_1e_2+(a_{20}+a_{11}+a_{02})e^2_1e^2_2+(a_{10}+a_{01})e_1e^3_2+a_{00}e^4_2  
\end{array}
\end{equation}    
which give (\ref{decomp}) in local basis $\{u_i\otimes v_j\}_{0\le i,j\le 2}$. We have used the assumption $p\ge 5$  
to get this decomposition. Otherwise it is no longer decomposable as $GL_2(\O_{S_{N,p}})$-modules. 

We now compute the images of theta operators. 
Recall that we have worked locally on the ordinary locus $S^h_{N,p}$.  
Therefore, the Hasse matrix $A$ in (\ref{important}) is invertible on $S^h_{N,p}$ and we may put 
$C=BA^{-1}=\left(\begin{array}{cc}
c_{11} & c_{12} \\
c_{21} & c_{22}
\end{array}
\right)$. 
Then we have  
\begin{equation}\label{c}
c_{11}=\frac{b_{11}a_{22}-b_{12}a_{21}}{\det(A)},\ 
c_{12}=\frac{-b_{11}a_{12}+b_{12}a_{11}}{\det(A)},\ 
c_{21}=\frac{b_{21}a_{22}-b_{22}a_{21}}{\det(A)},\ 
c_{22}=\frac{-b_{21}a_{12}+b_{22}a_{11}}{\det(A)}. 
\end{equation}
 By (\ref{important}) and Lemma \ref{omega}-(2), we have   
\begin{equation}\label{important1}
\begin{array}{r}
\nabla_{12}(\w_2)=\nabla_{11}(\w_1)=\phi_1-(c_{11}w_1+c_{12}w_2),\\
 \nabla_{12}(\w_1)=\nabla_{22}(\w_2)=\phi_2-(c_{21}w_1+c_{22}w_2). 
 \end{array}
\end{equation}
For simplicity, put  $\w^2_1=\w^{\otimes 2}_1,\ 
\w_1\w_2=\frac{1}{2}(\w_1\otimes\w_2+\w_2\otimes\w_1),\ \w^2_2=\w^{\otimes 2}_2$ (see (\ref{ks1})). 
Put $d_{ij}:=\langle \w_i,\nabla \w_j \rangle_{{\rm dR}}$ for $1\le i\le j\le 2$.  
Then we have 
\begin{prop}\label{first-step}Keep the notations above. Let $F(\omega_1\wedge \omega_2)^k$ be 
a local section of $\w^k$. Then 
$$\widetilde{\theta}(F(\omega_1\wedge \omega_2)^k)=
\sideset{^t}{}{\mathop{%
 \begin{pmatrix} 
\nabla_{11}F-kc_{11}F \\ 
\nabla_{12}F-k(c_{12}+c_{21})F\\
\nabla_{22}F-kc_{22}F 
 \end{pmatrix}}}
\left(
\begin{array}{c}
(\omega_1\wedge \omega_2)^k\w^2_1  \\
(\omega_1\wedge \omega_2)^k\w_1\w_2  \\
(\omega_1\wedge \omega_2)^k\w^2_2   
\end{array}
\right).
$$
\end{prop}
\begin{proof}
By definition, 
$$
\begin{array}{rl}
KS^{-1}\circ\nabla(F(\omega_1\wedge \omega_2)^k)&=\ds\sum_{1\le i\le j\le 2}(\nabla_{ij}F(\omega_1\wedge \omega_2)^k)
KS^{-1}(d_{ij})\\
&=\ds\sum_{1\le i\le j\le 2}(\nabla_{ij}F(\omega_1\wedge \omega_2)^k)
\w_i\w_j. 
\end{array}
$$ 
Therefore we have only to compute $\nabla_{ij}F(\omega_1\wedge \omega_2)^k$ modulo $R(U)$. 
By direct computation with (\ref{important1}), 
$$
\begin{array}{rl}
\nabla_{11}(\omega_1\wedge \omega_2)&=
\nabla_{11}(\omega_1)\wedge \omega_2+\omega_1\wedge \nabla_{11}(\omega_2)\\
&=\phi_1\wedge w_2-c_{11}\w_1\wedge\w_2 \equiv -c_{11}\w_1\wedge\w_2 \mod R(U).
\end{array}
$$
It follows from this that 
$$
\begin{array}{rl}
\nabla_{11}(F(\omega_1\wedge \omega_2)^k)&=
\nabla_{11}(F)(\omega_1\wedge \omega_2)^k+kF(\omega_1\wedge \omega_2)^{k-1}\nabla_{11}(\omega_1\wedge \omega_2)\\
&\equiv (\nabla_{11}F-kc_{11}F)(\omega_1\wedge \omega_2)^k \mod R(U).
\end{array}
$$
We will have the same results for remaining cases by using 
$$\nabla_{12}(\omega_1\wedge \omega_2)\equiv -(c_{12}+c_{21})\w_1\wedge\w_2 \mod R(U),\ 
\nabla_{22}(\omega_1\wedge \omega_2)\equiv -c_{22}\w_1\wedge\w_2 \mod R(U).$$ 
\end{proof}
The following lemma will be used later. We omit the proof. 
\begin{lem}\label{useful}Keep the notations above. 
Then under modulo $R(U)$, 
$$
\begin{array}{l}
\nabla_{11}(\omega^2_1)\equiv -2c_{11}\w^2_1-2c_{12}\w_1\w_2 , 
\nabla_{12}(\omega^2_1)\equiv  -2c_{21}\w^2_1-2c_{22}\w_1\w_2 ,
\nabla_{22}(\omega^2_1)=0 , \\
\nabla_{11}(\omega_1\omega_2)\equiv -c_{11}\w_1\w_2-c_{12}\w^2_2, 
\nabla_{12}(\omega_1\omega_2)\equiv-c_{11}\w^2_1-(c_{12}+c_{21})\w_1\w_2-c_{22}\w^2_2,\\
\nabla_{22}(\omega_1\omega_2)\equiv -c_{21}\w^2_1-c_{22}\w_1\w_2, \\
\nabla_{11}(\omega^2_2)\equiv 0, 
\nabla_{12}(\omega^2_2)\equiv -2c_{11}\w_1\w_2-2c_{12}\w^2_2, 
\nabla_{22}(\omega^2_2)\equiv  -2c_{21}\w_1\w_2-2c_{22}\w^2_2. 
\end{array}
$$
\end{lem}
Next we compute the image of each $\widetilde{\theta}_i,i=1,2,3$. 
Let 
$F_{11}(\omega_1\wedge \omega_2)^k\w^2_1,\ 
F_{12}(\omega_1\wedge \omega_2)^k\w_1\w_2,\ 
F_{22}(\omega_1\wedge \omega_2)^k\w^2_2$ be 
local sections of $\w^k\otimes_{\O_{S_{N,p}}}{\rm Sym}^2\E$. Put 
$$h=F_{11}(\omega_1\wedge \omega_2)^k\w^2_1+ 
F_{12}(\omega_1\wedge \omega_2)^k\w_1\w_2+ F_{22}(\omega_1\wedge \omega_2)^k\w^2_2.$$ 
Let $u_i,v_j,0\le i,j\le 2$ be the basis to define (\ref{decom-explicit}) with respect to $e_1=\w_1,e_2=\w_2$. 
We denote by $\ds\sum_{0\le i,j\le 2}a_{ij}(\w_1\wedge\w_2)^k u_i\otimes v_j,\ a_{ij}\in \O_{S^h_{N,p}} $ the image of $KS^{-1}\circ \nabla(h)$ under 
the projection to $\w^{\otimes k}\otimes_{\O_{S_{N,p}}}{\rm Sym}^2\E\otimes_{\O_{S_{N,p}}}{\rm Sym}^2\E$. 
\begin{prop}\label{step2}Keep the notations above. 
Each $a_{ij}$ is given as follows:
$$
\begin{array}{ll}
a_{22}=\nabla_{11}(F_{11})-(k+2)c_{11}F_{11}, & a_{21}=\nabla_{12}(F_{11})-(k(c_{12}+c_{21})+2c_{21})F_{11}-c_{11}F_{12},\\
 a_{12}=\nabla_{11}(F_{12})-(k+1)c_{11}F_{12}-2c_{12}F_{11},&  \\
a_{20}=\nabla_{22}(F_{11})-kc_{22}F_{11}-c_{21}F_{12}, & a_{11}=\nabla_{12}(F_{12})-(k+1)(c_{12}+c_{21})F_{12}-2c_{22}F_{11}-2c_{11}F_{22},  \\
a_{02}=\nabla_{11}(F_{22})-kc_{11}F_{22}-c_{12}F_{12}, & \\
a_{10}=\nabla_{22}(F_{12})-(k+1)c_{22}F_{12}-2c_{21}F_{22}, & 
a_{01}=\nabla_{12}(F_{22})-(k(c_{12}+c_{21})+2c_{12})F_{22}-c_{22}F_{12},  \\
a_{00}=\nabla_{22}(F_{22})-(k+2)c_{22}F_{22}. 
\end{array}
$$
\end{prop}
\begin{proof}
By definition, for $1\le i\le j\le 2$, 
$$
\begin{array}{rl}
KS^{-1}\circ\nabla(F_{ij}(\omega_1\wedge \omega_2)^k\w_i\w_j)&=\ds\sum_{1\le k\le l\le 2}(\nabla_{kl}F(\omega_1\wedge \omega_2)^k\w_i\w_j)
KS^{-1}(d_{kl})\\
&=\ds\sum_{1\le i\le j\le 2}(\nabla_{kl}F_{ij}(\omega_1\wedge \omega_2)^k\w_i\w_j)\otimes \w_k\w_l. 
\end{array}
$$
Therefore we have only to compute 
$$\nabla_{kl}F_{ij}(\omega_1\wedge \omega_2)^k\w_i\w_j=
(\nabla_{kl}F_{ij})(\omega_1\wedge \omega_2)^k\w_i\w_j+F_{ij}(\nabla_{kl}(\omega_1\wedge \omega_2)^k)\w_i\w_j
+(\omega_1\wedge \omega_2)^k\nabla_{kl}(\w_i\w_j).$$ 
By uisng Lemma \ref{useful} for the third term (see the proof of Proposition \ref{first-step} for the second term), 
we easily see the claim. 
\end{proof}
\begin{prop}\label{step3}Keep the notation above. Put 
$$\widetilde{\Theta}:=\widetilde{\theta}_1\circ \widetilde{\theta}:\w^k\lra \w^{k+2}$$ 
which is 
defined on $S^h_{N,p}$. Let $F(\omega_1\wedge \omega_2)^k$ be 
a local section of $\w^k$. 
Assume that $\nabla_{ij}$'s commute with each other. 
Then 
$$
\begin{array}{rl}
(\det A)\cdot\widetilde{\Theta}(F)&=\ds\frac{2}{3}\det A\cdot
\det\begin{pmatrix} 
\nabla_{11} &  \frac{1}{2}\nabla_{12} \\
\frac{1}{2}\nabla_{12}& \nabla_{22}  
 \end{pmatrix}(F)+
 \frac{2k(2k-1)}{9}F\cdot\det\begin{pmatrix} 
\nabla_{11} &  \frac{1}{2}\nabla_{12} \\
\frac{1}{2}\nabla_{12}& \nabla_{22}  
 \end{pmatrix}(\det A) \\
 &+\ds\frac{2k-1}{3}\Bigg(\nabla_{11}(F)\nabla_{22}(\det A)-\frac{1}{2}\nabla_{12}(F)\nabla_{12}(\det A)+
\nabla_{22}(F)\nabla_{11}(\det A)\Bigg)
\end{array}
$$
where $\det\begin{pmatrix} 
\nabla_{11} &  \frac{1}{2}\nabla_{12} \\
\frac{1}{2}\nabla_{12}& \nabla_{22}  
 \end{pmatrix}(F):=(\nabla_{11}\nabla_{22}-\frac{1}{4}\nabla^2_{12})F$. 
\end{prop}
\begin{proof}
By the assumption of the commutativity of $\nabla_{ij}$, we see that 
$$\langle  \nabla_{11} \w_1,\nabla_{22} \w_2 \rangle_{{\rm dR}}=0.$$
It follows from this that $c_{12}=c_{21}$. 
By Proposition \ref{first-step} and \ref{step2}, as a first step we have 
$$
\begin{array}{rl}
\widetilde{\Theta}(F)&=\ds\frac{2}{3} 
\det\begin{pmatrix} 
\nabla_{11} &  \frac{1}{2}\nabla_{12} \\
\frac{1}{2}\nabla_{12}& \nabla_{22}  
 \end{pmatrix}(F) + F\Big\{-\frac{k}{3}(\nabla_{22}c_{11}-\nabla_{12}c_{12}+\nabla_{11}c_{22})+
 \frac{2k(k-1)}{3}(c_{11}c_{22}-c^2_{12}) \Big\} \\
&+\ds\frac{2k-1}{3}(c_{12}\nabla_{12}(F)-c_{11}\nabla_{22}(F)-c_{22}\nabla_{11}(F)). 
\end{array}
$$
By the assumption of the commutativity of $\nabla_{ij}$'s and Proposition \ref{ab-lemma}, one has 
\begin{equation}\label{c's}
\begin{array}{l}
\nabla_{22}c_{11}-\nabla_{12}c_{12}+\nabla_{11}c_{22}=-\ds\frac{\det B}{\det A}.
\end{array}
\end{equation}
It follows from this that the second term becomes $\ds\frac{k(2k-1)}{3}\cdot \frac{\det B}{\det A}$. 
On the other hand, by Proposition \ref{ab-lemma} again, 
$\det\begin{pmatrix} 
\nabla_{11} &  \frac{1}{2}\nabla_{12} \\
\frac{1}{2}\nabla_{12}& \nabla_{22}  
 \end{pmatrix}(\det A)=\ds\frac{3}{2}\det B$. Putting these together, the second term becomes one 
 as in the claim. The third term is computed by using 
\begin{equation}\label{formula1}
\nabla_{11}(\det A)=-c_{11}\det A ,\ \nabla_{12}(\det A)=-2 c_{12}\det A,\ \nabla_{22}(\det A)=-c_{22}\det A .
\end{equation}
 This gives us the claim. 
\end{proof}

Put 
\begin{itemize}
\item $\Theta:=H_{p-1}\cdot \widetilde{\Theta}$,
\item $\theta:=H_{p-1}\cdot \widetilde{\theta}$, and
\item $\theta_i:=H_{p-1}\cdot \widetilde{\theta}_i,\ i=1,2,3$. 
\end{itemize}
We will check the holomorphy is preserved under these operators. 
However, to apply Proposition \ref{step2} and Proposition \ref{step3}, we require that 
the dual basis $D_{ij},1\le i\le j\le 2$ commute with each other. 
So far we do not know how to do that because we work on the field of 
positive characteristic. 
To overcome this situation we move to the localization at each point 
apply to the local deformation of each point which is called 
the generic square zero deformation by Katz (see p. 180-185 of \cite{kob}). 

Let us fix a point $x\in S_{N,p}(\bF_p)$ and denote by $A_x$ the corresponding abelian variety over $\bF_p$. 
Let $\widehat{\O}_{S_{N,p},x}=\bF_p[[t_{11},t_{12},t_{22}]]$ be the completion of $\O_{S_{N,p},x}$ along $x$ with the 
three variables and  $m_{x}$ be its maximal ideal. Put $R=\widehat{\O}_{S_{N,p},x}/m^2_x$. 
Fix a basis $\w'_1,\w'_2$ of $H^0(A_x,\Omega^1_{A_x})$. 
Let $\eta'_i\in H^1(A_x,\mathcal{O}_{A_x})$ be the dual element of $\w'_i$. 
Let $s_{ij}\in {\rm Hom}_{\bF_p}(H^0(A_x,\Omega^1_{A_x}),H^1(A_x,\mathcal{O}_{A_x}))$  be the element sending 
$\w'_i$ to $\eta'_j$ and $\w'_{k}$ to 0 if $k\not=i$.  
Let $\mathcal{X}$ be the deformation of $A_x$ to $R$ corresponding to  
$$s_{11}t_{11}+s_{12}t_{12}+s_{22}t_{22}\in 
{\rm Hom}_{\bF_p}(H^0(A_x,\Omega^1_{A_x}),H^1(A_x,\mathcal{O}_{A_x}))\otimes_{\bF_p} m_R$$ 
which we  call it the generic zero square deformation
(see  SUBLEMMA 8 at p.180 of \cite{kob}). 
We denote by $f_{\mathcal{X}}:\mathcal{X}\lra {\rm Spec}\hspace{0.3mm}R$ the corresponding abelian scheme. 
Put $D_{ij}=\ds\frac{\partial}{\partial t_{ij}}\in Der(R),\ 1\le i,j\le 2$. 
\begin{prop}\label{local-beha}
Let $\w_{i,\X}$ be any lift of $\w_i$ to ${f_{\X}}_\ast\Omega^1_{\X/{\rm Spec}\hspace{0.3mm}R}$. 
Then it holds that 
$$\langle \w_{i,\X},\nabla_{D_{k\ell}}(\w_{j,\X}) \rangle_{{\rm dR}}=\delta_{ik}\delta_{\ell j},\ 1\le i\le j\le 2,\ 
1\le k\le \ell\le 2.$$
\end{prop}
\begin{proof}
By universality of $\mathcal{A}/S_{N,p}$, 
there is a map ${\rm Spec}\hspace{0.5mm}R\lra S_{N,p}$ such that $\X=\mathcal{A}\times_{S_{N,p}}{\rm Spec}\hspace{0.5mm}R$ and $f_{\X}$ is the base change of the structure map $\mathcal{A}\lra S_{N,p}$ 
of  $\mathcal{A}/S_{N,p}$. Clearly $\nabla_{D_{k\ell}}$'s commute with each other and 
by formation of the construction, we can apply 
the results in Section \ref{ordinary-locus} to $\X/{\rm Spec}\hspace{0.3mm}R$. 
This means that if we define the functions $a_{ij,\X},b_{ij,\X}\in R$ as in (\ref{ab})  for $\X/{\rm Spec}\hspace{0.3mm}R$, 
these are nothing but the images of $a_{ij},b_{ij}\in \O_{S_{N,p}}$ under 
the map $\O_{S_{N,p}}\lra R$ coming from the universality. Then  $a_{ij,\X},b_{ij,\X}\in R$ play the 
same role as  $a_{ij},b_{ij}\in \O_{S_{N,p}}$ do.  
\end{proof}

\begin{prop}\label{theta}Keep the notations above. Recall the Hasse invariant $H_{p-1}=\det A\in H^0(S_{N,p},\w^{p-1})$.   
Recall $\Theta:=H_{p-1}\cdot \widetilde{\Theta}$, $\theta:=H_{p-1}\cdot \widetilde{\theta}$ and $\theta_i:=H_{p-1}\cdot \widetilde{\theta}_i,\ i=1,2,3$. 
Let $\ell$ be a rational prime coprime to $pN$. 
Then the following properties are satisfied:
\begin{enumerate}
\item Let $f\in GM_{(k,k)}(\G(N),\bF_p)$. Assume $\Theta(f)$ and $\theta(f)$ are not identically zero, then 
\begin{enumerate}
\item $\Theta(f)\in GM_{(k+p+1,k+p+1)}(\G(N),\bF_p)$ and 
$T(\ell^i)\Theta(f)=\ell^{2i}T(\ell^i)f$, 
\item  $\theta(f)\in GM_{(k+p+1,k+p-1)}(\G(N),\bF_p)$ and 
$\lambda_{\theta(f)}(\ell^i)=\ell^{i}\lambda_{f}(\ell^i)$ where $\lambda_G(\ell^i)$ stands for 
the Hecke eigenvalue of $G$ for $T(\ell^i)$ $($see the last part of Section \ref{geometry} with 
$($\ref{eigen}$))$, 
\end{enumerate}
\item Let $f\in GM_{(k+2,k)}(\G(N),\bF_p)$ and $i\in \{1,2,3\}$. Assume that $\theta_i(f)$ is not identically zero, then
 $\theta_i(f)\in   GM_{(k+p+i,k+p+2-i)}(\G(N),\bF_p)$ and 
$\lambda_{\theta_i(f)}(\ell^i)=\ell^{i}\lambda_{f}(\ell^i)$.  
\end{enumerate}
\end{prop}
\begin{proof}The weight correspondence in each case is obvious by definition. For 
the Hecke eigenvalues under theta operators are computed later in (\ref{hecke-formula}) in more general setting. 

What we have to do is to check that the holomorphy is preserved under 
these operators. 
Except for $\Theta$, By Proposition \ref{step2} the holomorphy is obvious since $\det(A)c_{ij}$'s are holomorphic, hence elements in the 
total space $S_{N,p}$. 
For $\Theta$, by using the generic square zero deformation, Proposition \ref{step3}, and 
Proposition \ref{local-beha}, we see that 
$\Theta(f)=\det(A)\cdot \widetilde{\Theta}(F)$ does not have any denominator. 
The point here is a priori we know $\det(A)^2\cdot \widetilde{\Theta}(F)$ is holomorphic on $S_{N,p}$ without using 
localization, but 
thanks to the symmetry $c_{12}=c_{21}$ after the localization, we see $\det(A)\cdot \widetilde{\Theta}(f)$ is holomorphic. This gives us the claims.   
\end{proof}
\begin{rmk}\label{theta0}
The construction of each of $\theta$, $\theta_2$, and $\theta_3$ works even when $p=3$.   
\end{rmk}
To study the $q$-expansion of the image of mod $p$ Siegel modular forms under the above operators we need 
to compute the Hasse invariant of Mumford's semi-abelian scheme. 
Recall Mumford's semi-abelian scheme $\pi:\mathcal{A}^{{\rm Mum}}\lra T_N$ (\ref{mum}) and consider its (relative) base change to $\bF_p$ which 
is denoted by $\pi_p:\mathcal{A}^{{\rm Mum}}_p=\mathcal{A}^{{\rm Mum}}\otimes\bF_p\lra T_{N,p}:=T_N\otimes\bF_p$. 
Fix the canonical invariant forms $\w_{{\rm can},1}=\ds\frac{dt_1}{t_1}, \w_{{\rm can},2}=\ds\frac{dt_2}{t_2}$ on $\mathcal{A}^{{\rm Mum}}_p$ which make up 
a basis of ${\pi_p}_\ast \Omega^1_{\mathcal{A}^{{\rm Mum}}_p/T_{N,p}}$. Recall the periods $q_{ij},1\le i\le j\le 2$ of this abelian variety.   
\begin{prop}\label{hasse-qexp}The following statements are satisfied:
\begin{enumerate}
\item for any $1\le i\le j\le 2$, $\langle \w_{{\rm can},i},\nabla\w_{{\rm can},j}\rangle_{{\rm dR}}=dq_{ij}$ where $\nabla$ is 
the Gauss-Manin connection with respect to $\pi_p$. 
\item 
The Hasse matrix of the Mumford's semi-abelian scheme  
$\mathcal{A}^{{\rm Mum}}\otimes\bF_p$ is the identity matrix with respect to the canonical basis 
$\w_{{\rm can},1}=\ds\frac{dt_1}{t_1}, \w_{{\rm can},2}=\ds\frac{dt_2}{t_2}$. 
In particular, the $q$-expansion $($ in the sense of  $($\ref{q-exp} $))$ of the Hasse invariant $H_{p-1}$ satisfies $H_{p-1}(q)=1$.
\end{enumerate} 
\end{prop}
\begin{proof}
The first claim follows from the fact that the Gauss-Manin connection is of formation compatible with any base change. In fact, 
before tensoring $\bF_p$, we once moved to $p$-adic setting and then apply Theorem 1, p.434 of \cite{gerritzen}. 

We now prove the second claim. 
Let $\eta_{{\rm can},i}$ be the element of $R^1{\pi_p}_\ast\O_{\mathcal{A}^{{\rm Mum}}_p}$  
corresponding to $\w_{{\rm can},i}$. Then it is well-known (cf. p.167 of \cite{goren}) that 
for any local sections $x\in R^1{\pi_p}_\ast\O_{\mathcal{A}^{{\rm Mum}}_p}$ and 
$y\in {\pi_p}_\ast \Omega^1_{\mathcal{A}^{{\rm Mum}}_p/T_{N,p}}$, one has 
$$\langle F^\ast(x^{(p)}),y \rangle_{{\rm dual}}=\langle x,V^\ast(y)^{(p^{-1})} \rangle^{(p)}_{{\rm dual}}$$
where $V:(\mathcal{A}^{{\rm Mum}}_p)^{(p)}\lra \mathcal{A}^{{\rm Mum}}_p$ stands for the Verschiebung morphism and 
$V^\ast(\ast)^{(p^{-1})}$ is defined by the composition of the pullback 
$V^\ast:{\pi_p}_\ast \Omega^1_{\mathcal{A}^{{\rm Mum}}_p/T_{N,p}}\lra 
{(\pi_p)^{(p)}}_\ast \Omega^1_{(\mathcal{A}^{{\rm Mum}}_p)^{(p)}/T_{N,p}}$ and 
the dual of the absolute Frobenius $F^\vee_{{\rm abs}}:{(\pi_p)^{(p)}}_\ast \Omega^1_{(\mathcal{A}^{{\rm Mum}}_p)^{(p)}/T_{N,p}}
\lra {\pi_p}_\ast (\Omega^1_{\mathcal{A}^{{\rm Mum}}_p/T_{N,p}})^{(p)}$ with respect to the polarization $\lambda$. 
The superscript of $\langle x,V^\ast(y)^{(p^{-1})} \rangle_{{\rm dual}}$ means 
the image of $\langle x,V^\ast(y)^{(p^{-1})} \rangle_{{\rm dual}}$ under the natural map $\O^{(p)}_{T_{N,p}}=
\O_{T_{N,p}}\otimes_{\O_{T_{N,p}},F_{{\rm abs}}} \O_{T_{N,p}}\lra \O_{T_{N,p}}$. 
Since we have the identifications between tangent spaces ${\rm Tan}(\mathcal{A}^{{\rm Mum}}_p):={{\pi_p}_\ast \Omega^1}^\vee_{\mathcal{A}^{{\rm Mum}}_p/T_{N,p}}={\rm Tan}(\mathcal{A}^{{\rm Mum}}_p[p])$,  
by using invariant forms $\w_1,\w_2$, it is naturally 
identified with ${\rm Tan}(\mu_p/T_{N,p})^{\oplus 2}$ where $\mu_p/T_{N,p}=
{\rm Spec}\hspace{0.5mm} \bF_p[x]/(x^p-1)\times_{{\rm Spec}\hspace{0.5mm} \bF_p} T_{N,p}$. 
Here we implicitly used a canonical base $\ds\frac{dx}{x}$ of $\Omega^1_{\mu_p/T_{N,p}}$ to identify. 
Hence the dual ${}^t V^\ast(\ast)^{(p^{-1})}$ is naturally identified with 
the map on ${\rm Tan}(\mu_p/T_{N,p})^{\oplus 2}$ which also comes from Vershiebung on 
$(\mu_p/T_{N,p})^{\times 2}$. 
However, it is well-known that it induces the identity map. Hence $V^\ast(\ast)^{(p^{-1})}=1$ with respect to 
$\w_{{\rm can},1},\w_{{\rm can},2}$. 
Plugging this into above formula, one has 
$$\langle F^\ast(\eta^{(p)}_{{\rm can},i}),\w_{{\rm can},j} \rangle_{{\rm dual}}=\langle \eta_{{\rm can},i},
V^\ast(\w_{{\rm can},j})^{(p^{-1})}  \rangle^{(p)}_{{\rm dual}}=
\langle \eta_{{\rm can},i},\w_{{\rm can},j} \rangle^{(p)}_{{\rm dual}}=\delta_{ij}$$
where $\delta_{ij}$ is the Kronecker's delta function. Then by non-degeneracy of the pairing, 
one has $F^\ast(\eta^{(p)}_{{\rm can},i})=\eta_{{\rm can},i}$ which gives us the claim.   
\end{proof}

\subsubsection{Vector valued case}\label{cycle-vector}
Let us first assume that $\underline{k}=(k_1,k_2),k_1\ge k_2\ge 1$ satisfies $p>k_1-k_2+2$. 
Under this assumption we have a non-canonical decomposition by using 
Schur functors:
\begin{equation}\label{decom1}
Sym^{k_1-k_2}\E\otimes_{\O_{S_{N,p}}}Sym^{2}\E= (Sym^{k_1-k_2-2}\E\otimes _{\O_{S_{N,p}}}\w^{\otimes 2})\oplus 
(Sym^{k_1-k_2}\E\otimes _{\O_{S_{N,p}}}\w)\oplus Sym^{k_1-k_2+2}\E
\end{equation}
as $Aut_{\O_{S_{N,p}}}(\E)$-modules (see Appendix A). 
When $k_1-k_2=1$, the decomposition  
\begin{equation}\label{decom-when1}
\E\otimes_{\O_{S_{N,p}}}Sym^{2}\E= 
Sym^{3}\E\oplus \E\otimes _{\O_{S_{N,p}}}\w
\end{equation}
is given by sending $\ds\sum_{1\le i\le 2\atop 0\le j  \le 2}a_{ij}\w_i\otimes v_j,\ v_j=\w^j_1\w^{2-j}_2$ to 
\begin{equation}\label{n=1}
\{a_{12}\w^3_1+(a_{11}+a_{22})\w^2_1\w_2+(a_{10}+a_{21})\w_1\w^2_2+a_{20}\w^3_2
\}
\oplus \Big((\ds\frac{a_{11}}{3}-\ds\frac{2a_{22}}{3})\w_1+(\ds\frac{2a_{10}}{3}-\ds\frac{a_{21}}{3})\w_2\Big)(\w_1\wedge\w_2). 
\end{equation}
As in (\ref{vector-case1}), one has 
\begin{equation}\label{vector-case2}
\xymatrix{
&  \w_{\underline{k}}=\w^{\otimes k_2}\otimes_{\O_{S_{N,p}}}{\rm Sym}^{k_1-k_2}\E  
  \ar@/_/[dl]^{\widetilde{\theta}^{\underline{k}}_1}
\ar@/_/[d]_{\widetilde{\theta}^{\underline{k}}_2}  \ar@/_/[dr]^{\widetilde{\theta}^{\underline{k}}_3}   
 &   \\
 \w_{(k_1,k_2+2)} & \w_{(k_1+1,k_2+1)}  & \w_{(k_1+2,k_2)}. 
} 
\end{equation}
The definition of each of these operators depends 
on the equations (\ref{decom1}) through  (\ref{decom4}) in Appendix A.  
The operator $\widetilde{\theta}^{\underline{k}}_1$ is used to reduce $k_1-k_2$ in the 
proof of Theorem \ref{reduction} below. 
Further, the theta cycle of  $\widetilde{\theta}^{\underline{k}}_2$ is studied in 
Section \ref{first-kind-theta}. Some combinations of compositions of 
these operators and also $\widetilde{\theta}^{\underline{k}}_3$ will be studied in \cite{yam}. 
 
\begin{prop}\label{theta-vector}Keep the notations above. Recall the Hasse invariant $H_{p-1}=\det A\in H^0(S_{N,p},\w^{p-1})$.   
Put  $\theta_i=\theta^{\underline{k}}_i:=H_{p-1}\cdot \widetilde{\theta}^{\underline{k}}_i,\ i=1,2,3$. 
If $f\in GM_{\underline{k}}(\G(N),\bF_p)$, then
\begin{enumerate}
\item $\theta^{\underline{k}}_1(f)\in   GM_{(k_1+p-1,k_2+p+1)}(\G(N),\bF_p)$, 
$\theta^{\underline{k}}_2(f)\in   GM_{(k_1+p,k_2+p)}(\G(N),\bF_p)$, and 

\noindent
$\theta^{\underline{k}}_3(f)\in   GM_{(k_1+p+1,k_2+p-1)}(\G(N),\bF_p)$,  
\item 
 if $f$ is a Hecke eigenform, then $\lambda_{\theta^{\underline{k}}_j(f)}(\ell^i)=\ell^{i}\lambda_{f}(\ell^i),\ j=1,2,3$ for any rational prime $\ell$ not dividing $pN$ 
provided if it is not identically zero. 
\end{enumerate} 
\end{prop}
\begin{proof}We first show that for any 
$f\in GM_{\underline{k}}(\G(N),\bF_p)$, $H_{p-1}\cdot KS^{-1}\circ\nabla(f) \mod R(U)$ is a Siegel modular form $G$ of 
weight $\w_{(k_1+p-1,k_2+p-1)}\otimes_{S_{N,p}}Sym^2\E$, namely it belongs to 
$$H^0(S_{N,p},\w_{(k_1+p-1,k_2+p-1)}\otimes_{S_{N,p}}Sym^2\E).$$  
To see this we compute $G$ explicitly. 
Choose a local basis $\w_1,\w_2$ of $\E$. Then $\delta_n:=\w^{(k_1-k_2)-n}_1\w^n_2(\w_1\wedge\w_2)^{k_2},\ 
0\le n\le k_1-k_2$  
make up a local basis of $\w_{\underline{k}}$.  Using this, we have a local expression $f=\ds\sum_{i=0}^{k_1-k_2}
f_i\delta_i$. We proceed as Proposition \ref{first-step}:  
$$
\begin{array}{rl}
KS^{-1}\circ\nabla(f)&=\ds\sum_{n=0}^{k_1-k_2}\ds\sum_{1\le i\le j\le 2}\nabla_{ij}(f_n\delta_n)\otimes 
KS^{-1}(d_{ij}),\ d_{ij}=\langle \w_i,\nabla\w_j \rangle_{{\rm dR}}\\
&=\ds\sum_{n=0}^{k_1-k_2}\ds\sum_{1\le i\le j\le 2}\nabla_{ij}(f_n\delta_n)\otimes  
\w_i\w_j\\
&=\ds\sum_{n=0}^{k_1-k_2}\ds\sum_{1\le i\le j\le 2}(\nabla_{ij}(f_n)\delta_n+f_n\nabla_{ij}(\delta_n))\otimes  
\w_i\w_j.  
\end{array}
$$ 
Therefore we have only to compute $\nabla_{ij}(f\delta_n)$ modulo $R(U)$. 
By direct computation, one has 
\begin{equation}\label{delta}
\begin{array}{l}
\nabla_{11}(\delta_n)\equiv -(k_1-n)c_{11}\delta_n-(k_1-k_2-n)c_{12}\delta_{n+1},\\ 
\nabla_{12}(\delta_n)\equiv -nc_{11}\delta_{n-1}-((k_1-n)c_{21}+(k_2+n)c_{12})\delta_{n}-(k_1-k_2-n)c_{22}\delta_{n+1},\\
\nabla_{22}(\delta_n)\equiv -nc_{21}\delta_{n-1}-(k_2+n)c_{22}\delta_{n}.
\end{array}
\end{equation} 
It follows from this with (\ref{c}) that $G$ is holomorphic since so is $\det(A)c_{ij},\ 1\le i,j\le 2$ (see also (\ref{c})). 

Next we compute the Hecke action on $G=H_{p-1}\cdot KS^{-1}\circ\nabla(f) \mod R(U)$. Assume $f$ is a Hecke eigenform with $q$-expansion 
$f(q)=\ds\sum_{T\in \mathcal{S}(\Z)_{\ge 0}}A_f(T)q^T_N=\ds\sum_{n=0}^{k_1-k_2}
\sum_{T\in \mathcal{S}(\Z)_{\ge 0}}A_{f_n}(T)q^T_N\delta_n$ (Here we discuss about only one component of RHS of (\ref{q-exp})). 
Note that $A_f(T)=\ds\sum_{n=0}^{k_1-k_2}A_{f_n}(T)\delta_n\in {\rm Sym}^{k_1-k_2}St_2(\bF_p)$.  
Since the Hasse matrix for Mumford semi-abelian scheme is identity with respect to 
the canonical basis related to the periods $q_{ij}$, $b_{ij}(q)=0$, hence all $c_{ij}(q),\ 1\le i\le j\le2$ are zero. 
This means that 
$$G(q)=\sum_{T\in \mathcal{S}(\Z)_{\ge 0}}A_G(T)q^T_N,\ A_G(T):=\ds\sum_{n=0}^{k_1-k_2}
\frac{1}{N}A_{f_n}(T)\delta_n\otimes T=\frac{1}{N}A_f(T)\otimes T
$$
where we identify $T=\left(
\begin{array}{cc}
a_{11}& \frac{1}{2}a_{12} \\
\frac{1}{2}a_{12}  & a_{22}
\end{array}
\right)$ with the vector  $a_{11}e^2_1+a_{12}e_1e_2+a_{22}e^2_2$ in $Sym^2St_2(\Z)$. 

Applying the  formula (\ref{hecke-fourier1}) with $(k'_1,k'_2)=(2+(p-1),p-1)$, we have  
\begin{equation}\label{hecke-formula}
\begin{array}{l} 
A_G(\ell^i;T)\\
\\
=\ds\sum_{\alpha+\beta+\gamma=i\atop \alpha,\beta,\gamma\ge 0}
\chi_1(\ell^\beta)\chi_2(\ell^\gamma)\ell^{\beta((k_1+2+p-1)-2)+\gamma((k_1+2+p-1)+k_2+p-1-3)}\times \\
 \ds\sum_{U\in R(\ell^\beta)\atop {a_U\equiv 0\ {\rm mod}\ \ell^{\beta+\gamma}\atop 
b_U\equiv c_U\equiv 0\ {\rm mod}\ \ell^{\gamma}}}
(\rho_{k_1-k_2}\otimes \rho_2)((\left(
\begin{array}{cc}
1& 0 \\
0  & \ell^\beta
\end{array}
\right)U)^{-1})\Bigg\{
A_G\left(\ell^\alpha\left(
\begin{array}{cc}
a_U\ell^{-\beta-\gamma}& \frac{b_U\ell^{-\gamma}}{2} \\
\frac{b_U\ell^{-\gamma}}{2}  & c_U\ell^{\beta-\gamma}
\end{array}
\right)\right)\\ \otimes \rho_2\Bigg( \ell^\alpha\left(
\begin{array}{cc}
a_U\ell^{-\beta-\gamma}& \frac{b_U\ell^{-\gamma}}{2} \\
\frac{b_U\ell^{-\gamma}}{2}  & c_U\ell^{\beta-\gamma}
\end{array}
\right)\Bigg)T\Bigg\}\\
\\
=\ell^{\alpha+\beta+\gamma}\ds\sum_{\alpha+\beta+\gamma=i\atop \alpha,\beta,\gamma\ge 0}
\chi_1(\ell^\beta)\chi_2(\ell^\gamma)\ell^{\beta(k_1-2)+\gamma(k_1+k_2-3)}\times \\
 \ds\sum_{U\in R(\ell^\beta)\atop {a_U\equiv 0\ {\rm mod}\ \ell^{\beta+\gamma}\atop 
b_U\equiv c_U\equiv 0\ {\rm mod}\ \ell^{\gamma}}}
\rho_{k_1-k_2}((\left(
\begin{array}{cc}
1& 0 \\
0  & \ell^\beta
\end{array}
\right)U)^{-1})
A_f\left(\ell^\alpha\left(
\begin{array}{cc}
a_U\ell^{-\beta-\gamma}& \frac{b_U\ell^{-\gamma}}{2} \\
\frac{b_U\ell^{-\gamma}}{2}  & c_U\ell^{\beta-\gamma}
\end{array}
\right)\right)\otimes T\\
\\
=\ell^i \lambda_f(\ell^i)A_G(T).
\end{array}
\end{equation}
\end{proof}
We now turn to give a proof of Proposition \ref{theta}. 
\begin{proof}We have only to show the claim (a)-(1). Other cases have been done already by Proposition \ref{theta-vector}. 
By the proof of Proposition \ref{theta-vector}, we have 
$$\Theta(f)(q)=\frac{2}{3}\sum_{T\in \mathcal{S}(\Z)_{\ge 0}}\frac{1}{N^2}\det(T)A_f(T)q^T_N.$$
Then the claim follows easily by the formula (\ref{hecke-formula}). 
\end{proof}
\begin{rmk}\label{nagaoka-wei}In \cite{b&n}, B\"ocherer and Nagaoka studied ``Ramanujan differential" for 
a mod $p$ Siegel modular form of general degree $g$ which comes from a classical form. 
When $g=2$, we extend their result for which the forms are not necessarily coming from classical forms. 

On the other hand, 
for any mod $p$ Siegel modular form $F$ of a parallel weight $k$ of degree 2 with $f(q)=\ds\sum_{T\in \mathcal{S}(\Z)_{\ge 0}}A_f(T)q^T_N$, 
Weissauer asked if  the formal series $\ds\sum_{T\in \mathcal{S}(\Z)_{\ge 0}}\frac{1}{N}TA_f(T)q^T_N$ is the $q$-expansion of a vector valued Siegel modular form \cite{wei3}. 
Proposition \ref{theta} gives an affirmative answer to his question. That is nothing but 
the $q$-expansion of the Siegel modular form $\theta(f)$ of weight $(k+p+1,k+p-1)$. 
\end{rmk}

\begin{rmk}\label{remark-decomp}Let $R$ be a domain with characteristic $p>0$. 
Let $\rho_k={\rm Sym}^k St_2(R)$ and we denote by $V(k)$ its representation space.
We also denote by $V(k,\ell)$ the representation space of $\rho_k\otimes \det^\ell St_2$ and 
$V(k)^{p^s}$ the $s$-th Frobenius twist of $V(k)$, that is the action is given by 
$\rho_k(\begin{pmatrix} 
a^{p^s}& b^{p^s} \\
c^{p^s}& d^{p^s}
\end{pmatrix})$ for $\begin{pmatrix} 
a& b \\
c& d
\end{pmatrix}\in GL_2(R)$. 
When $r:=k_1-k_2=p-1$ or $r=p-2$, by direct calculation, we have the following exact sequence 
$$0\lra V(p-3)\lra V(p-1)\otimes V(2)\lra V(1)^{(p)}\otimes V(1)\oplus V(p-1,1)\oplus V(p-3,2)\lra 0$$
and  
$$0\lra V(p-2)\lra V(p-2)\otimes V(2)\lra V(1)^{(p)}\oplus V(p-2,1)\oplus V(p-4,2)\lra 0.$$
The projection to $V(r-2,2)$ splits and it is given by the same basis as in case $r<p-2$. So the definition of  
$\theta^{(k_1,k_2)}_1$ for $k_1-k_2=p-1$ and $k_1-k_2=p-2$ still makes sense. This fact will be used later.  
\end{rmk}

Finally we explain how we reduce  the difference $k_1-k_2$  of  
a general weight $\underline{k}=(k_1,k_2)$ of a mod $p$ Siegel modular form by 
using the Verschiebung map which is the Cartier dual of the Frobenius map. 
The following construction has been reformulated conceptually in \cite{EFGMM}. 

Recall $d_2\circ Fr^\ast:R^1f^{(p)}_\ast\O_{\mathcal{A}^{(p)}}\lra R^1f_\ast\O_{\mathcal{A}}$ (see  (\ref{hasse-matrix})). 
Combining this with the pull-back $$F^\ast_{{\rm abs}}:R^1(f\circ {\rm proj})_\ast\O_{\mathcal{A}^{(p)}}
=(R^1f_\ast\O_{\mathcal{A}})^{(p)}\lra 
R^1f^{(p)}_\ast\O_{\mathcal{A}^{(p)}},$$  
we have a canonical element in $Hom_{\mathcal{O}_{S_{N,p}}}((R^1f_\ast\O_{\mathcal{A}})^{(p)},R^1f_\ast\O_{\mathcal{A}})$. 
By using the principal polarization, we also have a canonical element in 
$Hom_{\mathcal{O}_{S_{N,p}}}({\mathcal{E}^\vee}^{(p)},\mathcal{E}^\vee)$ and an identification $\mathcal{E}^\vee\simeq\mathcal{E}\otimes 
\omega^{-1}$ (cf. p.127, Exercise 5.16-(b) of \cite{hartshorne}) yields an element in 
$Hom_{\mathcal{O}_{S_{N,p}}}(\mathcal{E}^{(p)},\mathcal{E}\otimes \omega^{p-1})$. 
Let us denote this element by $V$. If we fix a local basis of $R^1f_\ast\O_{\mathcal{A}}$  and using it we express the Hasse matrix $A$ for $d_2\circ Fr^\ast$, then 
\begin{equation}\label{expressV}
V={}^t \widetilde{A}
\end{equation}
where $\widetilde{A}$ stands for the adjugate matrix of $A$. 
 An advantage of considering $V$ is that 
it is defined entirely over $S_{N,p}$ because we do not use the splitting of the 
Hodge filtration by the Frobenius map.  
The map $V$ is naturally extended to an $\mathcal{O}_{S_{N,p}}$-linear map 
$Sym^k \E^{(p)}=(Sym^k \E)^{(p)}\lra Sym^k\E\otimes \omega^{k(p-1)}$. 
Then it induces an injective morphism, which we also write $V$ again,   
\begin{equation}\label{ver}
H^0(\overline{S}_{N,p},{\rm Sym}^{k} \E^{(p)}\otimes \w^j)\stackrel{V}{\hookrightarrow} 
H^0(\overline{S}_{N,p},{\rm Sym}^{k} \E\otimes \w^{j+k(p-1)})=GM_{(k+j+k(p-1),j+k(p-1))}(\G(N),\bF_p)
\end{equation}
since $S^h_{N,p}$ is dense open in $S_{N,p}$. 
\begin{prop}\label{Hecke-equivalence}
The map $V$ is Hecke equivariant. Namely, the images of 
Hecke eigenforms outside $pN$ are also Hecke eigenforms outside $pN$ and further, 
the eigenvalues are never be changed.
\end{prop}
\begin{proof}
As in the proof of Proposition \ref{hasse-qexp}, Verschiebung map acts trivially on 
the holomorphic one forms on Mumford's semiabelian variety. 
The claim follows from this.  
\end{proof}
It is known that for any integer $k\ge 0$, as algebraic representations 
$\rho_k={\rm Sym}^kSt_2$ of $SL_2/\bF_p$  (and also the tensor product of such representations) can be obtained by successive extension of the following representation 
\begin{equation}\label{sl2}
\bigotimes_{i=1}^r\rho^{(p^{i})}_{m_i},\ m_i\in \{0,1,\ldots,p-1\},\ r\in \Z_{\ge 0}
\end{equation}
(see \cite{winter}). Applying $V$ and taking a successive extension repeatedly, we have the following: 
\begin{thm}\label{reduction}For any $f\in GM_{(k_1,k_2)}(\G(N),\bF_p)$, there exist an integer $m,\ 0\le m\le p-2$ 
and  $G\in GM_{(k'_1,k'_2)}(\G(N),\bF_p)$ with $k'_1\ge k'_2\ge 1$ and $p>k'_1-k'_2$ such that 
$\lambda_f(\ell^i)=\lambda_{G}(\ell^i)$ 
for $\ell\not|pN$.  
\end{thm}
\begin{proof}By Steinberg's Theorem (see Theorem 3.9 of \cite{herzig}) we see that  
each irreducible constituent of $\rho_{\underline{k}}$ is isomorphic to 
$\bigotimes_{i=1}^r{\rm Sym}^{m_i}\mathcal{E}^{(p^{t_i})}\otimes \w^{k_2},\ m_i\in \{0,1,\ldots,p-1\},
\ t_i\in \Z_{\ge 0}$. 
Applying $V$ many times if necessary, we may switch $F$ with a form of weight 
$\ds\bigotimes_{i=1}^r\rho_{m_i}\otimes \w^{s},\ m_i\in \{0,1,\ldots,p-1\},\ s,t_i\in \Z_{\ge 0}$ with $s\ge k_2$ whose highest weight (without the determinant part) is strictly lower than one before. 
This weight can be re-written as a successive extension of some weight as in (\ref{sl2}). By using the induction on the highest weight as 
a representation $SL_2/\bF_p$ and hitting $V$ again if necessary we 
can reduce the weight to $\rho_m\otimes  \w^{n},\ 0\le m\le p-1,\ n\in \Z_{> 0}$. 
By Proposition \ref{Hecke-equivalence} $V$ does not change Hecke eigenvalues and 
 the resulting form has the mod $p$ Galois representation which is the same as $\br_{F,p}$.     
\end{proof}

\subsection{Partial Hasse invariants on Ekedahl-Oort stratification}

In this subsection we will recall the Ekedahl-Oort stratification of Satake 
compactification $S^\ast_{N,p}$ of $S_{N,p}$ which 
consists of four strata and define partial Hasse invariants on the Zariski closures on each stratum. In \cite{oort2} (see Claim-2) after Theorem (4.1)), Oort proved that each stratum is quasi-affine by using the so called Raynaud's trick,  
constructing a kind of Hasse invariant on each stratum. 

To reduce the weight of a given mod $p$ Siegel modular form as in \cite{Edix} for elliptic modular forms, we have to extend these invariants to the Zariski closure of 
each stratum in $S^\ast_{N,p}$. 
Many people (\cite{Kos}, \cite{GK}, \cite{boxer}, \cite{GS},\cite{GS1}) have studied some  extensions of these invariants to the Zariski closure of each stratum.  
To carry out a similar thing in our setting we mimic the argument done in \cite{GS} and use the theory of 
``Dieudonn\'e $C(R)$-modules'' as in 
Section 2 of \cite{GO}. Here a commutative ring $R$ will  be specified later. 

A key is to compute the local behavior of the partial Hasse invariant on each stratum by using the corresponding Dieudonn\'e $C(R)$-module and it turns out to be holomorphic along the boundary locus in the Zariski closure. 

Henceforth we will use the convention of \cite{oort2}, \cite{GO} (see also \cite{oort1} which would be a friendly reference for readers). 

\subsubsection{Ekedahl-Oort stratification} 
Let us recall elementary sequences and final sequences in the case $g=2$. 
We refer to Section 2 of \cite{oort1} or Section 5 of \cite{oort2}. 
An elementary sequence 
 $\vp=(\vp(1),\vp(2))\in \Z^2_{\ge 0}$ with the convention $\vp(0):=0$ is 
 completely given by a 
 combinatorial way. For two elementary sequences $\vp_1,\vp_2$, 
 $\vp_1\prec \vp_2$ is defined by the lexicographic order on $\Z^2_{\ge 0}$. 
 We list all elementary sequences and their  
 invariants $f=f(\vp)$, $a=a(\vp)$, $|\vp|:=\vp(1)+\vp(2)$ (see p.49 of \cite{oort1} ) 
 as below (remark that the first row is ordered by $\prec$):
\begin{table}[htbp]
\begin{center}
{\renewcommand\arraystretch{2}
\begin{tabular}{|c|c|c|c|c|}
\hline
$\vp$  & \hspace{3mm} $(0,0)$\hspace{3mm}  &\hspace{3mm} $(0,1)$\hspace{3mm} &\hspace{3mm} $(1,1)$\hspace{3mm} 
&\hspace{3mm} $(1,2)$\hspace{3mm}    \\
\hline
$f=f(\vp)$  & 0  &  0 & 1  & 2    \\
\hline
$a=a(\vp)$  & 2  &  1 & 1  & 0    \\
\hline
$|\vp|$  & 0  &  1 & 2  & 3    \\
\hline
\end{tabular}}
\end{center}
\caption{Elementary sequences.}
\label{table1}
\end{table}

Next we recall final sequences. 
A final sequence is a certain vector $\psi=(\psi(1),\psi(2),\psi(3),\psi(4))$ of $\Z^4_{\ge 0}$ with the convention $\psi(0):=0$. 
By definition there is the one-to-one correspondence between final sequences and elementary sequences.  
Let us list all $\psi$ for each $\vp$ as below. 
\begin{table}[htbp]
\label{table2}
\begin{center}
{\renewcommand\arraystretch{2}
\begin{tabular}{|c|c|c|c|c|}
\hline
$\vp=(\vp(1),\vp(2))$  & \hspace{3mm} $(0,0)$\hspace{3mm}  &\hspace{3mm} $(0,1)$\hspace{3mm} &\hspace{3mm} $(1,1)$\hspace{3mm} 
&\hspace{3mm} $(1,2)$\hspace{3mm}    \\
\hline
$\psi=(\psi(1),\psi(2),\psi(3),\psi(4))$  & (0,0,1,2)  &  (0,1,1,2) & (1,1,2,2)  & (1,2,2,2)    \\
\hline
\end{tabular}}
\end{center}
\caption{Final sequences.}
\label{table2}
\end{table}

Let $\Phi$ be the set of all elementary sequences and $\Psi$ be the set of 
all final sequences. 
Let $A=(A,\phi,\lambda)$ be an element of $S_{N,p}(\bF_p)$ which represents 
a principal polarized abelian surface $A$ over $\bF_p$ with level structure $\phi$. Let us put $G=G_A=A[p]$ and often drop the 
subscript $A$ if the dependence is obvious from the context. 
Let $F=F_G:G\lra G^{(p)}$ be the relative Frobenius map and $V=V_G: G^{(p)}\lra G$ Verschiebung map. 
They satisfy 
$$F\circ V=[p]_{G^{(p)}},\ V\circ F=[p]_{G}.$$
Further we have Im$V={\rm Ker}F$ and Im$F={\rm Ker}V$. For any subgroup scheme $H\subset G$ we write 
\begin{equation}\label{rule}
V(H):={\rm Im}(V:H^{(p)}\lra H),\ F^{-1}(H):=F^{-1}(H^{(p)}).
\end{equation}
The principal polarization $\lambda:A\lra A^\vee$ induces 
the isomorphism 
\begin{equation}\label{isom-cartier}
\lambda:G_A\stackrel{\sim}{\lra} G_{A^\vee}= G^D_A
\end{equation}
where $G^D_A$ is the Cartier dual of $G_A$. By using this we have a canonical non-degenerate pairing 
$$\langle\ ,\ \rangle:G\times G\stackrel{\sim}{\lra} G\times G^D\lra \mu_p.$$
For a subgroup scheme $H\subset G$, we define 
$$H^\perp=\{g\in G\ |\  \langle g,h \rangle=1 {\rm\ for\ any\ }h\in H \}$$
If $H={\rm Im}V$, it follows from the adjoint property of 
$F$ and $V$ on the corresponding Dieudonn\'e module (cf. (\ref{adjoint})) that $H^\perp=H$, hence $H$ is maximal isotropic. 
By using the polarization $\lambda$, we can check that 
the Cartier duality of $F_G$  is given by 
\begin{equation}\label{dual-p}
(F_G)^D=V_{G^D}:(G^{(p)})^D=(G^D)^{(p)}\lra G^D.
\end{equation} 
\begin{dfn}\label{final-filtration} For $G=A[p]$ as above, a final sequence for $G$ is 
a filtration 
$$0=G_0\subset G_1 \subset G_2=H={\rm Im}(V)
\subset G_3 \subset G_4=G$$
such that there exists a final sequence $\psi\in \Psi$ satisfying the conditions:
for $j\in\{0,1,2,3,4\}$  
\begin{itemize}
\item {\rm rank}$(G_j)=p^j$, 
\item $(G_j)^\perp=G_{4-j}$, 
\item ${\rm Im}(V:G^{(p)}_{j}\lra G)=G_{\psi(j)}$, 
\item $F^{-1}(G_j)=G_{2+j-\psi(j)}$. 
\end{itemize}

\end{dfn}

We now study the Ekedahl-Oort stratification of $S_{N,p}$. 
By Oort (\cite{oort2}, see also Theorem 2.7 of \cite{oort1}), for each 
$A=(A,\phi,\lambda)\in S_{N,p}(\bF_p)$, $G=A[p]$ has a final filtration 
and $\psi\in\Psi$ associated to it is independent of the filtration. 
Recall that the set $\Psi$ is naturally identified with the set $\Phi$ consisting of elementary sequences. 
Hence $A$ gives rise to a unique elementary sequence $ES(A):=\varphi$. 
By using this facts, for each $\vp\in \Phi$, one can define 
$$S_{\varphi}:=\{A\in S_{N,p}(\bF_p)\ |\ ES(A)=\varphi \}.$$
If $A\in S_{\varphi}$, the invariants $f(\varphi)$ and $a(\varphi)$ are 
$f$-rank of $A$ and $a$-number of $A$ respectively (cf. Section 2.10 and 
Section 2.11 of \cite{oort1}). 
By a fundamental result of Oort (see Section 1 of \cite{oort1}), each $S_\vp$ is a locally closed subset of $S_{N,p}$ and by Corollary 5.4 of \cite{oort1}, 
${\rm dim}S_{\varphi}=|\varphi|$.  
Note that $S_{(1,2)}=S^h_{N,p}$. 
Let us consider the Zariski closure $\overline{S}_{\vp}$ of $S_{\vp}$ in $S_{N,p}$.  
Oort proved that 
$$\overline{S}_{\vp}=\coprod_{\vp'\prec\vp}S_{\vp'}.$$ 
Clearly  $\overline{S}_{(1,2)}=S_{N,p}$. 
By \cite{kob} and \cite{KO} the following holds: 
\begin{enumerate}
\item $\overline{S}_{(1,1)}$ is a normal surface whose singularities are exactly at all superspecial points. These singularities are all $A_1$-type and they are analytically  
isomorphic to $t_{11}t_{22}-t^2_{12}=0$, 
\item Only for $N\ge 3$, $\overline{S}_{(0,1)}$ consists of smooth projective curves  which are isomorphic to $\mathbb{P}^1_{\bF_p}$. The singularities of  
$\overline{S}_{0,1}$ are exactly at the superspecial points and each component contains $p^2+1$ superspecial points.  Moreover, at each singular point there are $p+1$ irreducible components
passing through and intersecting transversely. 
\item $\overline{S}_{(0,0)}=S_{(0,0)}$ is a finite set which consists of the superspecial abelian surfaces over $\bF_p$. 
\end{enumerate}
We next consider the same kind of stratification for Satake compactification $S^\ast_{N,p}$ of $S_{N,p}$. 
\begin{dfn}\label{sequence-for-semi} (cf. p.363, Section (6.1) of \cite{oort2} or 
p.123, Definition 6.4 of \cite{c&f})
Let $X$ be a semi abelian surface over $\bF_p$ with level 
structure $\phi$ with respect to $\G(N)$ defined as follows:
\begin{enumerate}
\item $X$ is given by 
an extension of group schemes
$$0\lra (\mathbb{G}_m)^r\lra X\lra A\lra 0,\ r\in\{1,2\}$$
where $A$ is a principal polarized abelian variety of dimension $2-r$ with 
the level structure $\phi_A:A[N]\simeq (\Z/N\Z)^{2g-2r}$ preserving 
the Weil pairing on $A[N]$ induced from the polarization and the standard 
alternating pairing on $(\Z/N\Z)^{4-2r}$;
\item the level structure is an isomorphism $\phi:X[N]\simeq (\Z/N\Z)^{4-r}$ 
compatible with $\phi_A$ via the natural surjection $X[N]\lra A[N]$ and a fixed surjection $(\Z/N\Z)^{4-r}\lra 
(\Z/N\Z)^{4-2r}$. Further, $\phi$ preserves the degenerate alternating pairings  
which are given by the pullback of the Weil pairing on $A[N]$ and 
the standard alternating pairing on $(\Z/N\Z)^{4-2r}$. 
\end{enumerate}
For such $X$ we associate an elementary sequence $\varphi=:ES(X)$ as follows:
$$(\vp(1),\vp(2))=
\left\{
\begin{array}{ll}
(1,2) &\ {\rm if}\ r=2\ {\rm or}\ r=1\ {\rm and}\ A\ {\rm is\ ordinary}  \\
(1,1) &\ {\rm if}\ r=1\ {\rm and}\ A\ {\rm is\ supersingular}.
\end{array}
\right.
$$
\end{dfn}
As was done for $S_{N,p}$, we define the stratum $S^\ast_{\vp}$ to be the set of semi-abelian surfaces 
over $\bF_p$ with an elementary sequence $\vp$. 
Let $\overline{S}^\ast_{\vp}$ be the Zariski closure of $S^\ast_{\vp}$ in $S^\ast_{N,p}$. 
Clearly, $\overline{S}^\ast_{(1,2)}=S^\ast_{N,p}$ and it is easy to see that 
$$\overline{S}^\ast_{\vp}=\coprod_{\vp'\prec\vp}S^\ast_{\vp'}.$$
Since $\overline{S}_{\vp'}$ with $\vp'\prec (0,1)$ is closed, $\overline{S}^\ast_{\vp}=
\overline{S}_{\vp}$ and they never intersect with the boundary $S^\ast_{N,p}\setminus S_{N,p}$.  

To define partial Hasse invariants in the next subsection, 
it is better to work with canonical filtrations instead of final filtrations on $G=A[p]$. 
Let us explain what the canonical filtrations are (see (2.2) of \cite{oort1}).  
Consider $V^i(G)={\rm Im}(G^{(p^i)}\stackrel{V^{(p^{i-1})}}{\lra} G^{(p^{i-1})}
\stackrel{V^{(p^{i-2})}}{\lra}\cdots 
\stackrel{V}{\lra} G)$ which gives rise to a decreasing filtration on $G$.  
Take all $F^{-j}(V^i(G)),\ j=1,\ldots,$ for each $i$. 
Then one has a filtration 
$$\cdots \subset F^{-n^{(i)}_1}(V^i(G))\subset 
F^{-n^{(i)}_1}(V^i(G))\subset \cdots \subset F^{-n^{(i)}_{r_i}}(V^i(G))\subset 
 F^{-n^{(i)}_{r_i}}(V^{i-1}(G))\subset \cdots$$
 for some integers $0<n_1<\cdots <n_{r_i}$. Clearly $F^{-n^{(i)}_j}(V^i(G))\supset V(G)$. 
This second filtration gives one between $G$ and $V(G)$. 
We denote by $\{V^i\}$ the first procedure and $\{F^{-j}\}$ the second procedure symbolically. 
Repeating again first $\{V^i\}$ and next $\{F^{-j}\}$, then at some step, 
the filtration between $0$ and $V(G)$ is stabilized by $\{V^i\}$ and 
the filtration between $V(G)$ and $G$ is stabilized by $\{F^{-j}\}$.  
We write such a filtration as follows:
$$0=N_0\subset  \cdots \subset N_r=V(G)\subset\cdots \subset N_s=G$$
where $r,s$ are non-negative integers. 
Then there exist functions 
$$\rho:\{0,\ldots,s\}\lra\Z_{\ge 0},\ 
v:\{0,\ldots,s\}\lra\{0,\ldots,r\},\ 
f:\{0,\ldots,s\}\lra\{r,\ldots,s\}$$
so that 
\begin{equation}
{\rm rank}(N_i)=p^{\rho(i)},\ V(N_i)=N_{v(i)},\ F^{-1}(N_i)=N_{f(i)}. 
\end{equation}
Since $G$ has a polarization, the triple $(\rho,v,f)$ is a symmetric canonical type in the sense of 
\cite{oort2}.  
Clearly $v$ and $f$ are increasing functions. 
We write $\Gamma_s:=\{0,\ldots,s-1\}$ and define $\pi_G:\Gamma_s\lra \Gamma_s$ by 
$$\pi_G(i)=\left\{
\begin{array}{ll}
v(i) &\ {\rm if}\ v(i+1)>v(i)  \\
f(i) &\ {\rm if}\ v(i+1)=v(i) .
\end{array}
\right.
$$
We write $B_i=N_{i+1}/N_i$ for $i\in \G_s$. 
By Lemma (2.4) of \cite{oort1},  the functions $v$ and $f$ are surjective and $\pi_G$ is 
bijective. We denote by $n(G)$ the order of $\pi_G$. 
Further, $v$ and $f$ satisfy the following properties:
\begin{enumerate}
\item $v(i+1)>v(i)\Longleftrightarrow f(i+1)=f(i)$ and in this case, 
$V$ maps $B^{(p)}_i$ isomorphically onto  $B_{\pi_G(i)}$,
\item $v(i+1)=v(i)\Longleftrightarrow f(i+1)>f(i)$ and in this case, 
$F$ maps $B_{\pi_G(i)}$  isomorphically onto $B^{(p)}_i$. 
\end{enumerate}
By the recipe given at (5.6) of \cite{oort2}, one can associate a triple $\tau=(\rho,v,f)$ 
to each $\vp\in \Phi$. We list all data as below (the final row stands for a direct relation between 
the canonical filtration $\{N_i\}_i$ and the final sequence $\{G_i\}_i$ attached to each final sequence):
\begin{table}[htbp]
\begin{center}
{\renewcommand\arraystretch{2}
\begin{tabular}{|c|c|c|c|c|}
\hline
$\vp=(\vp(1),\vp(2))$  & \hspace{3mm} $(0,0)$\hspace{3mm}  &\hspace{3mm} $(0,1)$\hspace{3mm} &\hspace{3mm} $(1,1)$\hspace{3mm} 
&\hspace{3mm} $(1,2)$\hspace{3mm}    \\
\hline
$(s,r)$  & (2,1)                     &  (4,2)        & (4,2)        & (2,1)    \\
\hline
$(\rho(0),\ldots,\rho(s))$  & (0,2,4)  &  (0,1,2,3,4)  & (0,1,2,3,4)  & (0,2,4)    \\
\hline
$(v(0),\ldots,v(s))$  & (0,0,1)        &  (0,0,1,1,2)  & (0,1,1,2,2)  & (0,1,1)    \\
\hline
$(f(0),\ldots,f(s))$  & (1,2,2)        &  (2,3,3,4,4)  & (2,2,3,3,4)  & (1,1,2)    \\
\hline
$(\pi_G(0),\ldots,\pi_G(s-1))$  & (1,0)  &  (2,0,3,1)  & (0,2,1,3)         & (0,1)    \\
\hline
$n_{\varphi}$  & 2  &  4 & 2  & 1    \\
\hline
canonical filtration & $N_i=G_{2i},\atop i\in\{0,1,2\}$  &  $N_i=G_i,\atop i\in\{0,\ldots,4\}$ & 
$N_i=G_i,\atop i\in\{0,\ldots,4\}$  & $N_i=G_{2i},\atop i\in\{0,1,2\}$    \\
\hline
\end{tabular}}
\end{center}
\caption{}
\label{table-list}
\end{table}

\subsubsection{Definition of partial Hasse invariants on Ekedahl-Oort stratification}\label{phi} 

\

We are now in position to recall the partial Hasse invariant $H_\vp$ defined by Oort on each 
stratum $S^\ast_{\vp}$ for $\vp\in\Phi\setminus \{(0,0)\}$.  
Recall $\omega:=\det(\mathcal{E})$ where $\mathcal{E}$ is the Hodge bundle on 
$\overline{S}_{N,p}:=
\overline{S}_{K(N)}\otimes \F_p$ (see (\ref{hodge})). 
The line bundle $\omega$ descends to an ample line bundle on 
$S^\ast_{N,p}$ by construction of Satake (minimal) compactification and we denote it by $\w$ again. 

For each $\varphi\in \Phi$,   
we have 
 $(\w^{-1}|_{\overline{S}^\ast_{\vp}})^\vee\simeq \w|_{\overline{S}^\ast_{\vp}}$ since 
 $\w$ is a line bundle.  

For $X=\mathcal{G}\times_{S^\ast_{N,p}}S^\ast_{\vp}$, put $G=X[p]$. We now apply the results in previous section to $G/S^\ast_{\vp}$. 

In case $\vp=(1,2)$, we may put $H_{\vp}=H_{p-1}|_{S_{(1,2)}}$ where  $H_{p-1}\in H^0(S_{N,p},\omega^{p-1})$ is 
the Hasse invariant.   

In case $\vp=(1,1)$.   
From Table \ref{table-list} we obtain  
\begin{equation}\label{comp-hasse1}
\begin{array}{c}
B_1^{(p^2)}\stackrel{F^{(p)}\atop \sim}{ \longleftarrow}
{B_2}^{(p)}\stackrel{V\atop \sim}{\lra} B_1,\  B_0^{(p)}
 \stackrel{V\atop \sim}{\lra} B_0 
\end{array}
\end{equation}
where $B_i=N_{i+1}/N_i,\ N_0=0,\ N_1=V^2(G),\ N_2={\rm Ker}(F)=V(G)$, and $N_3={\rm Ker}F^2=F^{-1}(N_2)$. 
The image or the inverse image should be understood under the convention (\ref{rule}). 
The filtration $N_0=0\subset N_1\subset N_2=V(G)$ induces 
the inclusions of tangent bundles:
$$0\subset \mathfrak{t}_{N_1}\subset \mathfrak{t}_{N_2}=\mathfrak{t}_{G}=\mathfrak{t}_{X/S^\ast_{\vp}}.$$
For $i=0,1$ put $\mathcal{L}_i=\mathfrak{t}_{N_{i+1}}/ \mathfrak{t}_{N_i}$. 
Then one has 
$\w^{-1}|_{S^\ast_{\vp}}=\det(\mathfrak{t}_{X/S^\ast_{\vp}})=
\mathcal{L}_0\otimes_{\mathcal{O}_{S^\ast_{\vp}}} \mathcal{L}_1$. 
Using this interpretation, for $\varphi=(1,1)$ we define 
\begin{equation}\label{hasse1}
H_{(1,1)}=(V\circ{F^{(p)}}^{-1})\otimes 
(V\circ V^{(p)}):
\underline{\w}^{-p^2}|_{S^\ast_{\vp}}\lra 
\underline{\w}^{-1}|_{S^\ast_{\vp}}.
\end{equation}
It follows from this that $H_{(1,1)}$ is regarded as a global no-where vanishing section of $\underline{\w}^{p^2-1}|_{S^\ast_{\vp}}$.

In case $\vp=(0,1)$.   
Similarly we obtain  
\begin{equation}\label{comp-hasse2}
\begin{array}{c}
B_0^{(p^4)}\stackrel{F^{(p^3)}\atop \sim}{ \longleftarrow}
{B_2}^{(p^3)}\stackrel{F^{(p^2)}\atop \sim}{\longleftarrow} {B_3}^{(p^2)} \stackrel{V^{(p)}\atop \sim}{\lra} B_1^{(p)}
 \stackrel{V\atop \sim}{\lra} B_0, \\
B_1^{(p^4)}\stackrel{V^{(p^3)}\atop \sim}{\lra}
{B_0}^{(p^3)}\stackrel{F^{(p^2)}\atop \sim}{\longleftarrow} {B_2}^{(p^2)} \stackrel{F^{(p)}\atop \sim}{\lla} B_3^{(p)}
 \stackrel{V\atop \sim}{\lra} B_1 
\end{array}
\end{equation}
where $B_i=N_{i+1}/N_i,\ N_0=0,\ N_1=VF^{-1}V(G),\ N_2=V(G)$, and $N_3=F^{-1}V(G)$. 
The filtration $N_0=0\subset N_1\subset N_2=V(G)$ induces 
the inclusions of tangent bundles:
$$0\subset \mathfrak{t}_{N_1}\subset \mathfrak{t}_{N_2}=\mathfrak{t}_{G}=\mathfrak{t}_{X/S^\ast_{\vp}}.$$
For $i=0,1$ put $\mathcal{L}_i=\mathfrak{t}_{N_{i+1}}/ \mathfrak{t}_{N_i}$ again. 
Then one has 
$\w^{-1}|_{S^\ast_{\vp}}=\det(\mathfrak{t}_{X/S^\ast_{\vp}})=
\mathcal{L}_0\otimes_{\mathcal{O}_{S^\ast_{\vp}}} \mathcal{L}_1$. 
Then for $\varphi=(0,1)$ we define 
\begin{equation}
H_{(0,1)}=(V\circ V^{(p)}\circ {F^{(p^2)}}^{-1}\circ {F^{(p^3)}}^{-1})\otimes 
(V\circ {F^{(p)}}^{-1}\circ {F^{(p^2)}}^{-1}\circ V^{(p^3)}):
\underline{\w}^{-p^4}|_{S^\ast_{\vp}}\lra 
\underline{\w}^{-1}|_{S^\ast_{\vp}}.
\end{equation}
It follows from this that $H_\vp$ is regarded as a global section of $\underline{\w}^{p^4-1}|_{S^\ast_{\vp}}$. 

In summary, for each $\vp\neq (0,0)$, we constructed a partial Hasse invariant 
on $S^\ast_{\vp}$ whose weight is given by $p^{n_\vp}-1$ in the following table:
\begin{table}[htbp]
\begin{center}
{\renewcommand\arraystretch{2}
\begin{tabular}{|c|c|c|c|}
\hline
$\vp$   &\hspace{3mm} $(0,1)$\hspace{3mm} &\hspace{3mm} $(1,1)$\hspace{3mm} 
&\hspace{3mm} $(1,2)$\hspace{3mm}  \\
\hline
Weight of $H_\vp$  & $p^4-1$  & $p^2-1$ & $p-1$      \\
\hline
\end{tabular}}
\end{center}
\caption{}
\label{table4}
\end{table}
\subsection{An extension of partial Hasse invariants}\label{ephi} 
Following \cite{GS} we compute the local behavior of the partial Hasse invariants and check that it can be naturally 
extended holomorphically on the Zariski closure $\overline{S}^\ast_\varphi$ for each 
$\varphi\not=(0,0)$. 
The basic tools are the covariant Dieudonne theory for $p$-divisible groups and its local deformation.  
We refer to \cite{GO}, \cite{GS}, and \cite{oort2} for the notation. The Dieudonne theory for 
a family of $p$-divisible groups was developed by Zink \cite{Zink} and the all results we are going to use follow from his theory,  
and we work as in the style done in \cite{GO} or \cite{GS}. 

For each point $X\in S^\ast_{N,p}$, by abuse of notation, 
we use the same symbol $X$ to denote the corresponding semi-abelian surface. 
 we denote by $\mathcal{D}$ the contravariant functor from  
 the category of $p$-divisible groups over $\bF_p$ to the category of 
 the Dieudonne modules over $\bF_p$.  For any subgroup scheme $H\subset X[p]$, put 
\begin{equation}\label{cov-d}
\mathbb{D}(H):=\mathcal{D}(H^D)
\end{equation} where $H^D$ is the  Cartier dual of $H$.  
Since the usual Dieudonne functor $\mathcal{D}$ is contravariant, we have 
a covariant functor $\mathbb{D}$.  

For any $\bF_p$-module $\mathcal{M}$ and a non-negative integer $n$, we define  
$\mathcal{M}^{(p^n)}=\mathcal{M}\otimes_{\bF_p,\phi^n}\bF_p$ where 
$\phi^n:\bF_p\lra \bF_p,x\mapsto x^{p^n}$.  
For an element $m\in \mathcal{M}$ we put $m^{(p^n)}=m\otimes_{\bF_p,\phi^n} 1$. In particular,  
$a^{p^n} m^{(p^n)}=m\otimes a^{p^n}=(am)\otimes 1=(am)^{(p^n)}$ for $a\in \bF_p$. 

The Frobenius map $F$ and the Verschiebung $V$ map on $X[p]$ induce $\bF_p$-linear maps 
$$\widehat{V}:=F^D:\mathbb{D}\lra \mathbb{D}^{(p)},\ 
\widehat{F}:=V^D:\mathbb{D}^{(p)}\lra \mathbb{D}.$$   
The polarization induces a natural identification $\D^{(p)}\times \D\simeq \D\times \D^{(p)}$ and defined 
a non-degenerate alternating pairing 
$\langle \ast, \ast \rangle:\D^{(p)}\times \D\simeq \D\times \D^{(p)}\lra \bF_p$  with a relation 
\begin{equation}\label{adjoint}
\langle \widehat{F}x, y \rangle=\langle x, \widehat{V}y \rangle\ {\rm for}\ x,y\in \D.
\end{equation}
Then $\widehat{V}$ is determined by $\widehat{F}$ from this relation and vice versa. 

By Table 4.3 of \cite{Hartwig} for $\varphi=(0,1)$ and Corollary in p.192 of \cite{kob} for $\varphi=(0,0)$, we have the matrices for $\widehat{F}$ 
and $\widehat{V}$ for any $X\in S^\ast_\varphi$ in a suitable basis 
(note that $\widehat{V}={}^t F$ and $\widehat{F}={}^t V$ where $F$ are given in those references and $V$ is obtained 
from $F$. Hence the 
role is interchanged due to the convention we have taken)

\begin{table}[htbp]
\begin{center}
\begin{tabular}{|c|c|c|}
\hline
$\varphi$  & $\widehat{V}$ & $\widehat{F}$   \\
\hline
$(0,1)$  & 
$\begin{pmatrix}
0 & 0 & 0 & 1 \\
1 & 0 & 0 & 0 \\
0 & 0 & 0 & 0 \\
0 & 0 & 0 & 0 
\end{pmatrix}$
  &  
$\begin{pmatrix}
0 & 0 & 0 & 0 \\
0 & 0 & -1 & 0 \\
0 & 0 & 0 & 1 \\
0 & 0 & 0 & 0 
\end{pmatrix}$
   \\
\hline
$(0,0)$  & $\begin{pmatrix}
0 & 0 & 1 & 0 \\
0 & 0 & 0 & 1 \\
0 & 0 & 0 & 0 \\
0 & 0 & 0 & 0 
\end{pmatrix}$  
&  
$\begin{pmatrix}
0 & 0 & -1 & 0 \\
0 & 0 & 0 & -1 \\
0 & 0 & 0 & 0 \\
0 & 0 & 0 & 0 
\end{pmatrix}$    \\
\hline
\end{tabular}
\end{center}
\caption{}
\label{table-fv}
\end{table}  

We now study the local deformation of $X \in S^\ast_{\varphi}$ for $\varphi=(0,1)$ 
or $(0,0)$. These are supersingular but not superspecial and superspecial cases. 
We first consider the case $\varphi=(0,1)$. 
Let $\mathcal{X}$ be the local deformation of $X$ over the formal completion 
$\widehat{\overline{S}}_{N,p}={\rm Spf}R,\ R:=\bF_p[[t_{11},t_{12},t_{22}]]$ of $\overline{S}_{N,p}$ at $X$. 
We write $\widehat{\overline{S}^\ast}_{(1,1)}$ for the formal completion of 
$\overline{S}^\ast_{(1,1)}$ at $X$. 
By (2.3) of \cite{GO}, the Frobenius map on 
the covariant Dieudonn\'e $C(R)$-module $\D(\mathcal{X})$ over $R$ 
(as in (\ref{cov-d}) it is defined by using 
the dual given by the polarization and Definition 2.2.5 and 
Theorem 2.2.7 of \cite{GO})
is given by 
$\widehat{V}=\begin{pmatrix}
t_{12} & t_{22} & 0 & 1 \\
1 & 0 & 0 & 0 \\
0 & 0 & 0 & 0 \\
0 & 0 & 0 & 0 
\end{pmatrix}$ and left upper $2\times 2$ matrix is nothing but the Hasse matrix. On 
$\widehat{\overline{S}^\ast}_{(1,1)}$ 
we must have $\det\begin{pmatrix}
t_{12} & t_{22} \\
1 & 0
\end{pmatrix}=-t_{22}=0$. Therefore, we have 
$\widehat{\overline{S}^\ast}_{(1,1)}={\rm Spf} R_{(1,1)},\ R_{(1,1)}=\bF_p[[t_{11},t_{12}]]$. 
Put $t=t_{12}$ for simplicity and then on $\overline{S}^\ast_{(1,1)}$ the Frobenius and the Vershiebung on 
$\D_1:=\D(\mathcal{X})\otimes_R\bF_p[[t_{11},t]]$ are  given by 
\begin{equation}\label{defor1}
\widehat{V}=\begin{pmatrix}
t & 0 & 0 & 1 \\
1 & 0 & 0 & 0 \\
0 & 0 & 0 & 0 \\
0 & 0 & 0 & 0 
\end{pmatrix},\ 
\widehat{F}=
\begin{pmatrix}
0 & 0 & 0 & 0 \\
0 & 0 & -1 & 0 \\
0 & 0 & t & 1 \\
0 & 0 & 0 & 0 
\end{pmatrix}
\end{equation}
where $\widehat{F}$ is computed from $\widehat{V}$ via the alternating pairing. 
Applying first (\ref{comp-hasse1}) for $\mathcal{X}$ and taking the covariant Dieudonn\'e module we have 
\begin{equation}\label{sequence1}
\D(B_1)^{(p^2)}\stackrel{\widehat{V}^{(p)}\atop \sim}{\lla}\D(B_2)^{(p)}
\stackrel{\widehat{F}\atop \sim}{\lra} \D(B_1),\ 
\D(B_0)^{(p)}\stackrel{\widehat{F}\atop \sim}{\lra}\D(B_0)
\end{equation}
Note again that the role for $F$ and $V$ on $\mathcal{X}$ is switched with $\widehat{V}$ and $\widehat{F}$ on the covariant 
Dieudonn\'e module respectively. Let $\{e_1,e_2,e_3,e_4\}$ be the basis of $\D_1$ over $R_{(1,1)}$ which satisfies 
$\langle e_i,e_{i+2} \rangle =1$ for $i=1,2$ and $\langle e_i,e_j \rangle$ for the remaining cases. 
We may assume that (\ref{defor1}) has been given in this basis. 
\begin{lem}\label{gen1} Keep the notation as above. 
It holds:
\begin{enumerate}
\item $\D(N_1)=\langle -e_2+te_3 \rangle_{R_{(1,1)}}$, $\D(N_2)=\langle e_2,e_3 \rangle_{R_{(1,1)}}$, and 
$\D(N_3)=\langle e_1-te_4,e_2,e_3 \rangle_{R_{(1,1)}}$, 
\item $\D(B_0)=\langle -e_2+te_3 \rangle_{R_{(1,1)}}$, $\D(B_1)=\langle e_3 \rangle_{R_{(1,1)}}$, and 
$\D(B_2)=\langle e_1-te_4 \rangle_{R_{(1,1)}}$. 
\end{enumerate}
\end{lem} 
\begin{proof}The first claim is done by direct computation without any difficulty and so details are omitted. 
Notice that $\D$ is covariant. Therefore, $\D(B_i)=\D(N_{i+1})/\D(N_i)$. Hence the second claim follows from (1). 
\end{proof}
Using this Lemma we compute the local behavior of $H_{(1,1)}$ along 
$\widehat{\overline{S}^\ast}_{(1,1)}$. We can view $H_{(1,1)}$ as a global section in 
$H^0(\overline{S}^\ast_{(1,1)},\omega^{p^2-1}(nS_{\overline{S}_{(0,1)}}))$ 
for a sufficiently large positive integer $n$ where ${\overline{S}_{(0,1)}}$ is 
regarded with a divisor. 
\begin{prop}\label{ph1}
The Hasse invariant $H_{(1,1)}$ is 
a section in $H^0(\overline{S}^\ast_{(1,1)},\omega^{p^2-1})$, 
namely, it has no poles along $S_{\overline{S}_{(0,1)}}(=\overline{S}^\ast_{(0,1)})$. 
Then  the support of zero locus (given by the reduced scheme 
structure) of  $H_{(1,1)}$ is given by 
$\overline{S}_{(0,1)}$. More precisely as a zero divisor, 
$${\rm div}_0(H_{(1,1)})=\sum_{ D\in \pi_0(\overline{S}_{(0,1)})}2pD$$
where $\pi_0(\overline{S}_{(0,1)})$ stands for the set of all irreducible components of 
$\overline{S}_{(0,1)}$.  
\end{prop}
\begin{proof}By (\ref{sequence1}) and Lemma \ref{gen1}  we may chase the image of a generator 
$(e_1-te_4 )^{(p)}=e^{(p)}_1-t^p e^{(p)}_4$ of  $\D(B_2)^{(p)}$. It follows that 
$$\widehat{F}(e^{(p)}_1-t^p e^{(p)}_4)=-t^p e_3$$
and 
$$\widehat{V}^{(p)}(e^{(p)}_1-t^p e^{(p)}_4)=
t^p e^{(p^2)}_1+e^{(p^2)}_2-t^p e^{(p^2)}_1=e^{(p^2)}_2\equiv t^{p^2}e^{(p^2)}_3.$$
On the other hand, for $(-e_2+te_3)^{(p^2)}=-e^{(p^2)}_2+t^{p^2}e^{(p^2)}_3 \in \D(B_0)^{(p^2)}$,  
$$\widehat{F}\circ\widehat{F}^{(p)}(-e^{(p^2)}_2+t^{p^2}e^{(p^2)}_3)=
t^{p^2}\widehat{F}(-e^{(p)}_2+t^{p}e^{(p)}_3)=t^{p^2+p}(-e_2+te_3).$$
Notice that 
According to (\ref{hasse1}) the partial Hasse invariant on $\overline{S}^\ast_{(1,1)}$ is computed as follows: 
$$(\widehat{F}\circ (\widehat{V}^{(p)})^{-1})\otimes 
(\widehat{F}\circ \widehat{F}^{(p)})(e^{(p^2)}_3\otimes (-e_2+te_3)^{(p^2)})=-t^{2p}(e_3\otimes (-e_2+te_3)).$$ 
The claim is now obvious from this.  
\end{proof} 
Next we consider the case $\varphi=(0,0)$. 
Let $\mathcal{X}$ be the local formal deformation of $X$ over the formal completion 
$\widehat{\overline{S}}_{N,p}={\rm Spf}R,\ R:=\bF_p[[t_{11},t_{12},t_{22}]]$ of $\overline{S}_{N,p}$ at $X$. 
We write $\widehat{\overline{S}^\ast}_{(0,1)}$ for the formal completion of 
$\overline{S}^\ast_{(0,1)}$ at $X$. By (2.3) of \cite{GO} again, the Frobenius map on 
the covariant Dieudonn\'e $C(R)$-module $\D(\mathcal{X})$ over $R$ 
is given by 
$\widehat{V}=\begin{pmatrix}
t_{11} & t_{12} & 1 & 0 \\
t_{12} & t_{22} & 0 & 1 \\
0 & 0 & 0 & 0 \\
0 & 0 & 0 & 0 
\end{pmatrix}$ and left upper $2\times 2$ matrix is nothing but the Hasse matrix. 
The defining equation of $\widehat{\overline{S}^\ast}_{(0,1)}=
\widehat{\overline{S}}_{(0,1)}$ is given by $t^{p+1}_{11}+t^{p+1}_{12}=0$ and so, 
for every $\zeta\in \bF_p$ a $(p+1)$-th root of $-1$,  
as explained in p.193 of \cite{kob}, the local defining equation of 
$\widehat{\overline{S}^\ast}_{(0,1)}=
\widehat{\overline{S}}_{(0,1)}$ at $X$ with a branch $\zeta$ is given by 
$$\widehat{\overline{S}^\ast}_{(0,1)}[\zeta]: t_{11}=t,\ t_{12}=\zeta t,\ t_{22}=\zeta^2 t.$$ 
Therefore, we have 
$$\widehat{\overline{S}^\ast}_{(0,1)}[\zeta]={\rm Spf}R_{(0,1)},\ R_{(0,1)}=\bF_p[[t]].$$ 
Then on $\overline{S}^\ast_{(0,1)}$ the Frobenius and the Vershiebung on 
$\D_0:=\D(\mathcal{X})\otimes_R\bF_p[[t]]$ are  given by 
\begin{equation}\label{defor2}
\widehat{V}=\begin{pmatrix}
t & \zeta t & 1 & 0 \\
\zeta t & \zeta^2 t & 0 & 1 \\
0 & 0 & 0 & 0 \\
0 & 0 & 0 & 0 
\end{pmatrix},\ 
\widehat{F}=
\begin{pmatrix}
0 & 0 & -1 & 0 \\
0 & 0 & 0 & -1 \\
0 & 0 & t & \zeta t \\
0 & 0 & \zeta t & \zeta^2 t 
\end{pmatrix}. 
\end{equation} 
Applying first (\ref{comp-hasse2}) for $\mathcal{X}$ and taking the covariant 
 Dieudonn\'e $C(R)$-module for $R=R_{(0,1)}$ we have 
\begin{equation}\label{sequence2}
\begin{array}{c}
\D(B_0)^{(p^4)}\stackrel{\widehat{V}^{(p^3)}\atop \sim}{ \longleftarrow}
{\D(B_2)}^{(p^3)}\stackrel{\widehat{V}^{(p^2)}\atop \sim}{\longleftarrow} {\D(B_3)}^{(p^2)} \stackrel{\widehat{F}^{(p)}\atop \sim}{\lra} \D(B_1)^{(p)}
 \stackrel{\widehat{F}\atop \sim}{\lra} \D(B_0), \\
\D(B_1)^{(p^4)}\stackrel{\widehat{F}^{(p^3)}\atop \sim}{\lra}
{\D(B_0)}^{(p^3)}\stackrel{\widehat{V}^{(p^2)}\atop \sim}{\longleftarrow} {\D(B_2)}^{(p^2)} \stackrel{\widehat{V}^{(p)}\atop \sim}{\lla} \D(B_3)^{(p)}
 \stackrel{\widehat{F}\atop \sim}{\lra} \D(B_1) 
\end{array}
\end{equation}
Let us introduce  a basis $\{e_1,e_2,e_3,e_4\}$ of $\D_0=\D(N_4)$ over $R_{(0,1)}$ as in previous case.  
We may assume that (\ref{defor2}) has been given in this basis. 
\begin{lem}\label{gen2} Keep the notation as above. 
The followings hold:
\begin{enumerate}
\item $\D(N_1)=\langle e_1-\zeta^{-1}e_2 \rangle_{R_{(0,1)}}$, $\D(N_2)=\langle u_1,u_2 \rangle_{R_{(0,1)}}$, and 
$\D(N_3)=\langle -e_1+\zeta e_2, u_1,u_2 \rangle_{R_{(0,1)}}$ where $u_1=-e_1+t(e_3+\zeta e_4),\ 
u_2=-e_2+t\zeta (e_3+\zeta e_4)$, 
\item $\D(B_0)=\langle e_1-\zeta^{-1}e_2 \rangle_{R_{(0,1)}}$, $\D(B_1)=\langle u_1 \rangle_{R_{(0,1)}}$,  
$\D(B_2)=\langle -e_1+\zeta e_2 \rangle_{R_{(0,1)}}$, and $\D(B_3)=\langle e_3 \rangle_{R_{(0,1)}}$. 
\end{enumerate}
\end{lem} 
\begin{proof}This can be proved as in Lemma \ref{gen1} and so the details are omitted. 
\end{proof}
Using this Lemma, we compute the order of $H_{(0,1)}$ along 
$\widehat{\overline{S}^\ast}_{(0,1)}$. As seen before, 
we can view  $H_{(0,1)}$ as a global section in
 $H^0(\overline{S}^\ast_{(0,1)},\omega^{p^4-1}(n\overline{S}_{(0,0)}))$ for 
 a sufficiently large positive integer $n$ where $\overline{S}_{(0,0)}$ is regarded 
 with a divisor.  
\begin{prop}\label{ph2}
The partial Hasse invariant $H_{(0,1)}$ is 
a section in $H^0(\overline{S}^\ast_{(0,1)},\omega^{p^4-1})$.   
The support of zero locus (given by the reduced scheme 
structure) of $H_{(0,1)}$ is exactly given by $\overline{S}^\ast_{(0,0)}=\overline{S}_{(0,0)}$. More precisely as a zero divisor on 
each irreducible component $D$ of $\overline{S}_{(0,1)}$, 
$${\rm div}_0(H_{(0,1)}|_D)=\sum_{ P\in S_{(0,0)}\cap D}(p^4-p^3-p^2+p)P.$$
\end{prop}
\begin{proof}As in the previous case we apply (\ref{sequence2}) and Lemma \ref{gen2}.  We first consider the first sequence of 
(\ref{sequence2}). We may chase the image of a basis vector  
$e^{(p^2)}_3$ of  $\D(B_3)^{(p^2)}$. It follows that 
$\widehat{F}^{(p)}(e^{(p^2)}_3)=-e^{(p)}_1+t^p(e^{(p)}_3+\zeta^p e^{(p)}_4)$ and 
$$\widehat{F}\circ \widehat{F}^{(p)}(e^{(p^2)}_3)=t^p(u_1+\zeta^p u_2)=t^p(u_1-\zeta^{-1} u_2)=t^p(-e_1+\zeta^{-1}e_2).$$
Similarly one has 
$\widehat{V}^{(p^2)}(e^{(p^2)}_3)=e^{(p^3)}_1$ and 
$$\widehat{V}^{(p^3)}\circ \widehat{V}^{(p^2)}(e^{(p^2)}_3)=t^{p^3}(e^{(p^4)}_1+\zeta^{p^3}e^{(p^4)}_2)=
t^{p^3}(e^{(p^4)}_1-\zeta^{-1}e^{(p^4)}_2)=
t^{p^3}(e_1-\zeta^{-1}e_2)^{(p^4)}.$$
Therefore, it follows that 
$$\widehat{F}\circ \widehat{F}^{(p)}\circ (\widehat{V}^{(p^3)}\circ \widehat{V}^{(p^2)})^{-1}
((e_1-\zeta^{-1}e_2)^{(p^4)})=t^{-p^3+p}.$$
Next we consider the second sequence of (\ref{sequence2}). 
It is easy to see that $\widehat{F}(e^{(p)}_3)=u_1$. In what follows we compare the images of $u^{(p^4)}_1$  
and $e^{(p)}_3$ at the relay point $\D(B_0)^{(p^2)}$. 

Note that $e_3\equiv \frac{1}{2}(e_3-\zeta e_4)$ in $\D(B_3)$.  
Then  $e^{(p)}_3\equiv(\frac{1}{2}(e_3-\zeta e_4))^{(p)}=\frac{1}{2}(e^{(p)}_3+\zeta^{-1}e^{(p)}_4)$ in $\D(B_3)^{(p)}$. 
Therefore, we have    
$$\widehat{V}^{(p)}(e^{(p)}_3)=\widehat{V}^{(p)}(\frac{1}{2}(e^{(p)}_3+\zeta^{-1}e^{(p)}_4))=\frac{1}{2}(e^{(p)}_1+\zeta^{-1}e^{(p)}_2)=
\frac{1}{2}(e_3-\zeta e_4)^{(p)}$$ and 
$$\widehat{V}^{(p^2)}(\frac{1}{2}(e^{(p)}_1+\zeta^{-1}e^{(p)}_2))=
\frac{1}{2}\Big\{t^{p^2}(e^{(p^3)}_1+\zeta^{p^2}e^{(p^3)}_2)+\zeta^{-1}t^{p^2}\zeta^{p^2}(e^{(p^3)}_1+\zeta^{p^2} e^{(p^3)}_2) \Big\}
=t^{p^2}(e^{(p^3)}_1+\zeta e^{(p^3)}_2).$$
Since $u^{(p^4)}_1=-e^{(p^4)}_1+t^{p^4}(e^{(p^4)}_3+\zeta^{p^4}e^{(p^4)}_4)=
-e^{(p^4)}_1+t^{p^4}(e^{(p^4)}_3+\zeta e^{(p^4)}_4)$, we have 
$$\widehat{F}^{p^3}(u^{p^4}_1)=t^{p^4}(-e^{(p^3)}_1-\zeta e^{(p^3)}_2)=t^{p^4-p^2}(\widehat{V}^{(p^2)}\circ\widehat{V}^{(p)})(e^{(p)}_3).$$

Put $m=u_1\otimes (e_1-\zeta^{-1}e_2)$ and then $m^{(p^4)}=u^{(p^4)}_1\otimes (e_1-\zeta^{-1}e_2)^{(p^4)}$. 
According to (\ref{hasse1}) the partial Hasse invariant on 
$\widehat{\overline{S}^\ast}_{(0,1)}$ is computed as follows: 
$$(\widehat{F}\circ (\widehat{V}^{(p^2)}\circ\widehat{V}^{(p)})^{-1}\circ\widehat{F}^{p^3})\otimes 
(\widehat{F}\circ   \widehat{F}^{(p)} \circ (\widehat{V}^{(p^3)}\circ\widehat{V}^{(p^2)})^{-1})(m^{(p^4)})=
-t^{p^4-p^3-p^2+p}m$$
The claim is now obvious from this.  
\end{proof}
By Koecher's principle, the Hasse invariant $H_{p-1}$ can be extended to a global 
section of $\omega^{p-1}$ on the Satake compactification $S^\ast_{N,p}$ or on a toroidal compactification $\overline{S}_{N,p}$. 
To end this section we give the following fact for $H_{p-1}$ which is known well (cf \cite{oort2}):
\begin{prop}\label{ph0}
\label{zero-hasse}The zero locus of $H_{p-1}$ is given by $\overline{S}^\ast_{(1,1)}$ 
with multiplicity one. 
\end{prop}
\begin{proof}Let $\mathcal{X}$ be the local deformation of $X\in \overline{S}^\ast_{(1,1)}$ over the formal completion 
$\widehat{\overline{S}}_{N,p}={\rm Spf}R,\ R:=\bF_p[[t_{11},t_{12},t_{22}]]$ of $\overline{S}_{N,p}$ at $X$. 
Then $\widehat{V}=F^\vee$ on $\D(\mathcal{X})$ is given by the formal completion 
$$\begin{pmatrix}
1 & 0 & 0 & 0  \\
 t_{11} & t_{12}  &  0 & 0  \\
0 & 0 & 0 & 0 \\
0 & 0 & 0 & 0
\end{pmatrix}.$$ Therefore, $H_{p-1}$ on $\widehat{\overline{S}}_{N,p}$ is given by 
the determinant of left upper $2\times 2$ matrix in $\widehat{V}$, that is $t_{12}$. Hence we have the claim.  
\end{proof}



\section{Weight reduction}
In this section we carry out weight reduction for mod $p$ Siegel modular forms. 
Let $\uk=(k_1,k_2),\ k_1\ge k_2\ge 1$ satisfying $p>k_1-k_2+3$. Recall 
$\w_{\uk}={\rm Sym}^{k_1-k_2}\mathcal{E}\otimes_{\bS_{N,p}} \w^{k_2}$, where 
$\mathcal{E}$ is the 
Hodge bundle on a fixed toroidal compactification $\bS_{N,p}$ and $\w=\det\mathcal{E}$. Let $\mathcal{C}=\bS_{N,p}-S_{N,p}$ be the 
boundary component. The  Hodge bundle $\mathcal{E}$ is 
extended to a vector bundle on $\bS_{N,p}$ which is denoted again by $\E$.  
By p.130, Theorem 6.7 of \cite{c&f}, the Kodaira-Spencer map gives an isomorphism 
\begin{equation}\label{KS-cpt}
\wedge^3 \Omega^1_{\bS_{N,p}}\simeq \omega^{\otimes 3}(-\mathcal C)
\end{equation}
where $\omega=\det \E$ on $\bS_{N,p}$. The line bundle $\w$ on  $\bS_{N,p}$ descends to 
$S^\ast_{N,p}$ as an ample line bundle. 
Put $D_2={\rm div}_0(H_{p-1})$ and $D_1={\rm div}_0(H_{(1,1)})$. 
By definition the support of $D_1$ is included in $S^\ast_{N,p}$.  
Applying the construction in Section \ref{ephi} to $X=
\mathcal{G}\times_{S^\ast_{N,p}}D_1$, 
one can extend $H_{(0,1)}$ to $D_1$ as a global section of $\w^{p^4-1}|_{D_1}$. 
Here $\w^{p^4-1}|_{D_1}$ stands for the pullback of $\w^{p^4-1}$ to $D_1$ with 
respect to the natural embedding from $D_1$ to $S^\ast_{N,p}$. 
We denote it by $H^{D_1}_{(0,1)}$  and put  
$D_0={\rm div}_0(H^{D_1}_{(0,1)})$. 

Then by Proposition \ref{ph1},\ref{ph2}, and \ref{ph0}, we see that 
$$D_2=\overline{S}^\ast_{(1,1)},\ 
D^{{\rm red}}_{1}=\overline{S}^\ast_{(0,1)}=\overline{S}_{(0,1)},\ 
D^{{\rm red}}_{0}=\overline{S}^\ast_{(0,0)}=\overline{S}_{(0,0)}=S_{(0,0)}$$
where the superscript ``red" stands for the reduced scheme structure. 
Thus, $D_2$ is the 2-dimensional, and irreducible, non-ordinary locus; 
$D^{{\rm red}}_{1}$ is the (connected) 1-dimensional supersingular locus, and 
$D^{{\rm red}}_{0}$ is the 0-dimensional superspecial locus. 
Let us consider the exact sequence 
$$0\lra \w_{\uk}\otimes \w^{-(p-1)}\stackrel{\times H_{p-1}}{\lra} \w_{\uk}\lra 
i_\ast\w_{\uk}|_{D_2}\lra 0$$
where  $i:D_2\hookrightarrow \overline{S}_{N,p}$ is the natural embedding, 
$\w_{\uk}|_{D_2}=i^\ast \w_{\uk}$ is the pullback of $\w_{\uk}$ to $D_2$, and 
$i_\ast\w_{\uk}|_{D_2}$ stands for the extension of $\w_{\uk}|_{D_2}$ by zero outside $D_2$.   
This gives rise to a long exact sequence 
\begin{equation}\label{long1}
\begin{array}{ccccccc} 0\lra & H^0(\overline{S}_{N,p}, \w_{\uk}\otimes \w^{-(p-1)})&\stackrel{\times H_{p-1}}{\lra}& 
H^0(\overline{S}_{N,p},\w_{\uk})&\lra & 
H^0(D_2,\w_{\uk}|_{D_2})& \\
\lra &H^1(\overline{S}_{N,p}, \w_{\uk}\otimes \w^{-(p-1)})&\lra & 
H^1(\overline{S}_{N,p},\w_{\uk})&\lra & 
H^1(D_2,\w_{\uk}|_{D_2})& \\
\lra&H^2(\overline{S}_{N,p}, \w_{\uk}\otimes \w^{-(p-1)})& \lra &  
H^2(\overline{S}_{N,p},\w_{\uk})&\lra & 
H^2(D_2,\w_{\uk}|_{D_2})& \\
\lra &H^3(\overline{S}_{N,p}, \w_{\uk}\otimes \w^{-(p-1)})&\lra  & H^3(\overline{S}_{N,p}, \w_{\uk})&\lra &0&.
\end{array}
\end{equation}
Here we use the fact that $H^i(\overline{S}_{N,p},i_\ast\w_{\uk}|_{D_2})=
H^i(D_2,\w_{\uk}|_{D_2})$ by III, Lemma 2.10 of \cite{hartshorne}. 
Note that $\mathcal{E}^\vee\simeq \mathcal{E}\otimes \w^{-1}$  
(cf. p.127, Exercise 5.16-(b) of \cite{hartshorne}). 
Then, by Theorem \ref{vanishing-lan}-(3), we have 
$$H^3(\overline{S}_{N,p},\w_{\uk}\otimes \w^{-(p-1)})=0$$ if $k_2-(p-1)-3>0$.  
Then plugging Theorem \ref{vanishing-lan} again into this long exact sequence we have the vanishing of 
cohomology:     
\begin{prop}\label{coh1}
Suppose that $k_2>3$ and $p>k_1-k_2+3$. 
The followings hold: 
\begin{enumerate}
\item the restriction $H^0(\overline{S}_{N,p},\w_{\uk})\lra  
H^0(D_2,\w_{\uk}|_{D_2})$ is surjective if $k_2>p+2$ or $k_1<p-1$, 
\item
 $H^1(D_2,\w_{\uk}|_{D_2})=0$ if $k_2>p+2$ or $k_1<p-1$.   
 \end{enumerate}
\end{prop}
\begin{proof}Applying the following duality 
$$H^{3-i}(\overline{S}_{N,p},\w_{\uk}\otimes \w^{-(p-1)})\simeq 
H^i(\overline{S}_{N,p},{\rm Sym}^{k_1-k_2}\mathcal{E}\otimes \w^{p+2-k_1}(-\mathcal{C})))$$ 
for $i=1,2$ to Theorem \ref{vanishing-lan} 
 with  the Kodaira-Spencer map (\ref{KS-cpt}),  it yields two conditions in the first claim. 
It is the same for the second condition.  
\end{proof}
Similarly, by Proposition \ref{ph1}, there is an exact sequence 
$$0\lra \w_{\uk}\otimes \w^{-(p^2-1)}|_{D_2}\stackrel{\times H_{(1,1)}}{\lra} \w_{\uk}|_{D_2}\lra 
\w_{\uk}|_{D_1}\lra 0.$$ 
This gives rise to a long exact sequence 
\begin{equation}\label{long2}
\begin{array}{ccccccc}
0\lra &H^0(D_2,\w_{\uk}\otimes \w^{-(p^2-1)}|_{D_2})&\stackrel{\times H_{(1,1)}}{\lra}& 
 H^0(D_2,\w_{\uk}|_{D_2})&\lra &
H^0(D_1,\w_{\uk}|_{D_1})& \\
\lra&H^1(D_2,\w_{\uk}\otimes \w^{-(p^2-1)}|_{D_2})&\lra & 
 H^1(D_2,\w_{\uk}|_{D_2})&\lra & H^1(D_1,\w_{\uk}|_{D_1})&  \\
\lra& H^2(D_2,\w_{\uk}\otimes \w^{-(p^2-1)}|_{D_2})&\lra&0&&.
\end{array}
\end{equation}
Since $D_2$ is a normal surface (see the explanation given right before Definition 
\ref{sequence-for-semi}), it is Cohen-Macaulay.  We then apply 
Serre duality (cf. Theorem 5.71, p.182 of \cite{KoMo} with the reflexivity of 
$\w_{\uk}$ explained in Section 7 of \cite{ghitza2}) and (\ref{KS-cpt}) to obtain  
$$H^{2-i}(D_2,\w_{\uk}\otimes \w^{-(p^2-1)}|_{D_2})\simeq H^i(D_2,
({\rm Sym}^{k_1-k_2}\mathcal{E}\otimes \w^{p^2-1-k_1}(-\mathcal{C}))|_{D_2}\otimes\mathcal{K}_{D_2})$$
where $\mathcal{K}_{D_2}$ is the dualizing sheaf of $D_2$. 
By adjunction formula (cf. Proposition 5.73, p. 182 of \cite{KoMo}) we see that 
$$\mathcal{K}_{D_2}\simeq \omega^3|_{D_2}\otimes_{\mathcal{O}_{D_2}} \mathcal{O}(D_2)|_{D_2}=\omega^{p+2}|_{D_2}$$
since the line bundle $\mathcal{O}(D_2)$ defined by the divisor $D_2$ satisfies 
$\mathcal{O}(D_2)=\mathcal{O}({\rm div}_0(H_{p-1}))=\omega^{p-1}$. 
Plugging this into the above isomorphism we have 
$$H^{2-i}(D_2,\w_{\uk}\otimes w^{-(p^2-1)}|_{D_2})\simeq H^i(D_2,
({\rm Sym}^{k_1-k_2}\mathcal{E}\otimes w^{p^2+p+1-k_1}(-\mathcal{C}))|_{D_2}).$$ 
Notice that Proposition \ref{coh1} is still true even if we replace $\w_{\uk}$ with $\w_{\uk}(-\mathcal{C})$ by 
Theorem \ref{vanishing-lan}. 
Therefore, we have the following. 
\begin{prop}\label{coh2} 
Suppose $p>k_1-k_2+3$. 
 \begin{enumerate}
 \item the natural restriction $H^0(D_2,\w_{\uk}|_{D_2})\lra 
H^0(D_1,\w_{\uk}|_{D_1})$ is surjective if $k_1<p^2+p-2$ or $k_2>p^2+2$. 
\item any element of 
 $H^1(D_1,\w_{\uk}|_{D_1})$ is liftable to an element of 
$S_{(p^2+p+1-k_2,p^2+p+1-k_1)}(\G(N))$ if $k_2>p^2+2$ or $k_1<p^2-1$ and it vanishes if $k_2> p^2+p+1$. 
 \end{enumerate}
\end{prop}
\begin{proof}By Proposition \ref{coh1}-(2), $H^1(D_2,\w_{\uk}\otimes \w^{-(p^2-1)}|_{D_2})=0$ if 
$k_2-(p^2-1)>p+2$ or $k_1-(p^2-1)<p-1$, hence $k_2>p^2+p+1$ or $k_1<p^2+p-2$. 
On the other hand, $H^1(D_2,
({\rm Sym}^{k_1-k_2}\mathcal{E}\otimes \w^{p^2+p+1-k_1}(-\mathcal{C}))|_{D_2})=0$ if 
$p^2+p+1-k_1>p+2$ or $p^2+p+1-k_2<p-1$, hence $k_2>p^2+2$ or $k_1<p^2-1$. 
Summing up   
$$H^1(D_2,\w_{\uk}\otimes \w^{-(p^2-1)}|_{D_2})=H^1(D_2,
({\rm Sym}^{k_1-k_2}\mathcal{E}\otimes \w^{p^2+p+1-k_1}(-\mathcal{C}))|_{D_2})=0$$
if $k_1<p^2+p-2$ or $k_2>p^2+2$. 

For the second claim, since $k_2>p+2$ or $k_1<p-1$, $H^1(D_2,\w_{\uk}|_{D_2})=0$ by Proposition \ref{coh1}-(2).  
It follows from  (\ref{long2}) that  
$$H^1(D_1,\w_{\uk}|_{D_1})\hookrightarrow 
H^2(D_2,\w_{\uk}\otimes \w^{-(p^2-1)}|_{D_2})\simeq 
H^0(D_2,
({\rm Sym}^{k_1-k_2}\mathcal{E}\otimes \w^{p^2+p+1-k_1}(-\mathcal{C}))|_{D_2}).$$
By   Proposition \ref{coh1}-(1), the restriction map from 
$H^0(\overline{S}_{N,p},{\rm Sym}^{k_1-k_2}\mathcal{E}\otimes \w^{p^2+p+1-k_1}(-\mathcal{C}))$ to the right hand side 
of the above sequence is 
surjective provided if $p^2+p+1-k_1>p+2$ or $p^2+p+1-k_2<p-1$. 
\end{proof}
By Proposition \ref{ph2}, there is an exact sequence 
$$0\lra \w_{\uk}\otimes \w^{-(p^4-1)}|_{D_1}\stackrel{\times H_{(0,1)}}{\lra} \w_{\uk}|_{D_1}\lra 
\w_{\uk}|_{D_0}\lra 0.$$ 
This gives rise to a long exact sequence 
\begin{equation}\label{long3}
\begin{array}{ccccccc}
0\lra & H^0(D_1, \w_{\uk}\otimes \w^{-(p^4-1)}|_{D_1})&\stackrel{\times H_{(1,1)}}{\lra}& 
 H^0(D_1, \w_{\uk}|_{D_1})&\lra& 
H^0(D_0, \w_{\uk}|_{D_0})& \\
\lra& H^1(D_1, \w_{\uk}\otimes \w^{-(p^4-1)}|_{D_1})&\lra & H^1(D_1, \w_{\uk}|_{D_1}) &\lra &0 .
\end{array}
\end{equation}
Unfortunately the curve $D_1$ is not reduced, hence non-Cohen-Macaulay. 
However, it is projective and we can apply more extensive duality theorem 
(see Theorem in Preface right after p.4 of  \cite{AK}).  
The formally defined dualizing sheaf 
$$ \mathcal{K}_{D_1}:=\omega_{D_2}|_{D_1}\otimes_{\mathcal{O}_{D_1}} \mathcal{O}(D_1)|_{D_1}=\omega^{p^2+p+1}|_{D_1}$$
where the second equality follows from 
 $\mathcal{O}(D_1)=\mathcal{O}({\rm div}_0(H_{(1,1)}))=\omega^{p^2-1}|_{D_2}$  
 plays the same role in the duality 
$$ H^1(D_1, \w_{\uk}\otimes \w^{-(p^4-1)}|_{D_1})\simeq H^0(D_1,
({\rm Sym}^{k_1-k_2}\mathcal{E}\otimes \w^{p^4-1-k_1}(-\mathcal{C}))|_{D_1}\otimes_{\mathcal{O}_{D_1}}\mathcal{K}_{D_1})$$
$$ \simeq H^0(D_1,({\rm Sym}^{k_1-k_2}\mathcal{E}\otimes \w^{p^4+p^2+p-k_1})|_{D_1})$$
Then by Proposition \ref{coh2}-(1) we have 
\begin{prop}\label{coh3}The natural map  
$$ H^0(D_1, \w_{\uk}|_{D_1})\lra 
H^0(D_0, \w_{\uk}|_{D_0})$$ is 
 surjective if $k_2>p^4+p^2+p$. 
\end{prop}
\begin{proof}By Proposition \ref{coh2} (2) $H^1(D_1, \w_{\uk}\otimes \w^{-(p^4-1)}|_{D_1})=0$ if 
$k_2-(p^4-1)>p^2+p+1$. The claim now follows from the exact sequence (\ref{long3}). 
\end{proof}
\begin{cor}\label{surj}The restriction map 
$$H^0(\overline{S}_{N,p},\w_{\uk})\lra H^0(S_{(0,0)},\w_{\uk}|_{S_{(0,0)}})$$ 
is Hecke equivariant and surjective if $k_2>p^4+p^2+p$. 
\end{cor}
\begin{proof}By the results of Ghitza (cf. the proof of Theorem 28 in \cite{ghitza1}), 
the restriction map is Hecke equivariant. 
The surjectivity follows from  Proposition \ref{coh3},    
 Proposition \ref{coh1}, and \ref{coh2}.  
\end{proof}
By Corollary \ref{surj} we will reduce the weight of mod $p$ Siegel modular forms after the study of 
the images of theta operators. 

\subsection{A non-vanishing criterion for theta operators} 
In this section, we will study the image of theta operators and give a 
sufficient condition when the image is not identically zero. 
A key point is to observe the behavior when we restrict forms to the superspecial locus $S_{(0,0)}$ in $S_{N,p}$. 

For any $\underline{k}=(k_1,k_2)$, 
let us recall the automorphic sheaf $\w_{\underline{k}}$ on $\overline{S}_{N,p}$. As Ghitza studied in \cite{ghitza1}, 
\cite{ghitza2}, we consider the space of superspecial forms: 
$$SS_{\underline{k}}:=H^0(S_{(0,0)},\w_{\underline{k}}|_{S_{(0,0)}}).$$
Let $\mathbb{T}_{N}$ be the ring generated over $\Z$ of Hecke operators  
$T(\ell^i),i\ge 0$, for all primes $\ell{\not|}N$ if $k_2\ge 2$ and  
$T(\ell^i),i\ge 0$, for all primes $\ell{\not|}pN$ otherwise. 
Suppose a $\mathbb{T}_N$ module $M$ over $\bF_p$ has the basis consisting of 
simultaneous eigenvectors for all elements in $\mathbb{T}_N$. 
For each simultaneous eigenvector $f$, we associate 
$\Phi_f:\mathbb{T}_N\lra \bF_p$ which sends $t\in \mathbb{T}_N$ to 
the eigenvalue of $f$ for $t$.  We call it the Hecke eigensystem associated to $f$. 
We denote by $HES(M)$  the set of the Hecke eigensystems associated to 
all simultaneous eigenvectors in $M$. 

For any integer $t>0$, 
put 
$$GM_t(K(N),\bF_p):=\ds\bigoplus_{(k_1,k_2)\in\Z^2_{\ge 0}\atop k_1\ge k_2\ge t}
GM_{(k_1,k_2)}(K(N),\bF_p),\ 
GS_t(K(N),\bF_p):=\ds\bigoplus_{(k_1,k_2)\in\Z^2_{\ge 0}\atop k_1\ge k_2\ge t}
GS_{(k_1,k_2)}(K(N),\bF_p),$$ and 
$SS_t:=\ds\bigoplus_{(k_1,k_2)\in\Z^2_{\ge 0}\atop k_1\ge k_2\ge t}SS_{(k_1,k_2)}$. 
Then in \cite{ghitza2}, Ghitza proved the following
\begin{thm}\label{eigensystem}For any integer $t>0$, 
$$HES(GM_t(K(N),\bF_p))=
HES(GS_t(K(N),\bF_p))=
HES(SS_t).$$
\end{thm}
\begin{proof}The proof here is essentially that given  in  \cite{ghitza1},\cite{ghitza2}. 
The first equality (essentially) follows from Theorem 6 in \cite{ghitza2}. 
It follows from the proof of Theorem 6 in \cite{ghitza2} with Serre's vanishing theorem 
that $HES(GS_t(K(N),\bF_p))\supset HES(SS_t)$ since any superspecial form is liftable to a Siegel cusp form.  
Therefore, we may prove the opposite inclusion. 

By abuse of notation we put $\omega_{\uk}:=\omega_{\uk}\otimes \bF_p$ for simplicity.   
Let $\mathcal{I}$ be the ideal sheaf of the inclusion $i:S_{(0,0)}\hookrightarrow \overline{S}_{N,p}$ with the exact sequence 
$$0\lra \mathcal{I} \lra \mathcal{O}_{\overline{S}_{N,p}}\lra i_\ast \mathcal{O}_{S_{(0,0)}}\lra 0.$$
By tensoring with $\mathcal{I}^n,\ n\in \Z_{\ge 0}$ we have 
$0\lra \mathcal{I}^{n+1} \lra \mathcal{I}^n\lra i_\ast \mathcal{O}_{S_{(0,0)}}\otimes \mathcal{I}^n \lra 0$ 
and 
$$0\lra \omega_{\uk}\otimes\mathcal{I}^{n+1} \lra \omega_{\uk}\otimes\mathcal{I}^n\lra \omega_{\uk}\otimes 
i_\ast \mathcal{O}_{S_{(0,0)}}\otimes \mathcal{I}^n \lra 0$$
as well. 
Taking the global sections we have a homomorphism 
\begin{equation}\label{rn}r_n:H^0(\overline{S}_{N,p},\omega_{\uk}\otimes\mathcal{I}^n)\lra 
H^0(\overline{S}_{N,p},\omega_{\uk}\otimes 
i_\ast \mathcal{O}_{S_{(0,0)}}\otimes \mathcal{I}^n)=H^0(S_{(0,0)}, i^\ast(\omega_{\uk}\otimes\mathcal{I}^n/\mathcal{I}^{n+1})).
\end{equation}
Notice that $r_0$ is nothing but the restriction map to $S_{(0,0)}$. 
Let us start with  a Hecke eigenform $f$ in $GM_{\uk}(K(N),\bF_p)=H^0(\overline{S}_{N,p},\omega_{\uk})$. 
If $r_0(F)\neq 0$ then we are done. Otherwise $f$ belongs to $H^0(\overline{S}_{N,p},\omega_{\uk}\otimes\mathcal{I})$. 
Testing if $r_1(f)$ is zero and repeating this procedure, eventually there exists $n\in \Z_{\ge 0}$ such that $r_n(f)\neq 0$ in 
$H^0(S_{(0,0)}, i^\ast(\omega_{\uk}\otimes\mathcal{I}^n/\mathcal{I}^{n+1}))$
by Noetherian-ness. As in the proof of Theorem 28 (see p.380. line 11-12) of \cite{ghitza1} we see that 
$i^\ast(\omega_{\uk}\otimes\mathcal{I}^n/\mathcal{I}^{n+1})=i^\ast(\omega_{\uk}\otimes {\rm Sym}^n({\rm Sym}^2\mathcal{E}))$. 
Similar to (\ref{sl2}) and the proof of Theorem \ref{reduction}, by using $V$ (note that it induces an isomorphism between the spaces in question and preserves Hecke eigenforms) there exists an irreducible constituent of the form ${\rm Sym}^m\mathcal{E}\otimes \w^k$ 
for $m,k\in \Z_{\ge 0}$ with $k\neq 0$ of    
$\omega_{\uk}\otimes {\rm Sym}^n({\rm Sym}^2\mathcal{E})$ so that 
$r_n(f)\neq 0$ belongs to it. Applying the periodic property 
(see p.820. line 2 from the bottom of \cite{ghitza2}) we may assume that $k$ is sufficiently large as it satisfies $k\ge t$. 
Hence $r_n(f)$ yields a non-trivial system in $HES(SS_t)$ which is the same as $f$.   
Hence $HES(GS_t(K(N),\bF_p))
\subset HES(SS_t)$ and we have the claim. 
\end{proof}
As a corollary we have the following:
\begin{cor}\label{cor-of-eigensystem}
Keep the notation as above. 
For any Hecke eigenform $f$ in $GM_{(k_1,k_2)}(K(N),\bF_p)$, the exists a 
Hecke eigenform $G$ 
in $GS_{(k'_1,k'_2)}(\G(N),\bF_p)$ for some weight $(k'_1,k'_2)$ with $k'_1\ge k'_2\ge 3$ such that 
\begin{enumerate}
\item the Hecke eigensystems of $f$ and $G$ are same, 
\item  the restriction of $G$ to $S_{(0,0)}$ is not identically zero. 
\end{enumerate}
\end{cor}
\begin{proof}By the proof of the previous theorem there exists $n$ such that $r_n(f)$ and it belongs to $SS_{(m,k)}$ for 
$m,k\in \Z_{\ge 0}$ with $k\neq 0$. Applying the periodic property again 
(see p.820. line 2 from the bottom of \cite{ghitza2}) there exits $t\in \Z_{\ge 0}$ such that 
the restriction map $GM_{(m+t(p^2-1),k+t(p^2-1))}(K(N),\bF_p)\lra SS_{(m+t(p^2-1),k+t(p^2-1))}\simeq SS_{(m,k)}$ is surjective and 
preserving Hecke eigenvalues. 
Let $g$ be the element in $SS_{(m+t(p^2-1),k+t(p^2-1))}$ corresponding to $r_n(f)$ under the above isomorphism. 
By Proposition 1.2.2 of \cite{a&s} we can find a Hecke eigenform $G$ in $GM_{(m+t(p^2-1),k+t(p^2-1))}(K(N),\bF_p)$ which maps to 
$g$ whose eigenvalues are the same as $G$. Hence we have the claim.
\end{proof}

We are ready to study the non-vanishing of the images of mod $p$ Siegel modular forms under theta operators. 
\begin{thm}\label{non-vanishing1}
Let  $f$ be an element in $GM_{(k,k)}(\G(N),\bF_p),\ k\ge 2$. Assume that $f$  is not identically zero on  $S_{(0,0)}$. 
The followings hold:  
\begin{enumerate}
\item if $p\not|k$, then $\theta(f)$ is not identically zero and,  
\item if $p\not|k(2k-1)$, then $\Theta(f)$ is not identically zero. 
\end{enumerate}
\end{thm}
\begin{proof}
By assumption, we may assume that $\alpha:=f|_{\{X\}}$ is non-zero for 
some $X\in S_{(0,0)}$. We work on an open neighborhood $U$ of $X$ to trivialize 
$\w^k$. 
Let $\mathcal{X}$ be the generic square zero deformation of $X$ over ${\rm Spec}\hspace{0.5mm}R$ 
(see the argument before Proposition \ref{theta}). By universality of $\mathcal{A}_U=\mathcal{A}\times_{S_{N,p}}U$ we have a map 
$\O_U\lra R$. Using this map, we have $\mathcal{A}_{\mathcal{X}}=A_{U}\times_U 
{\rm Spec}\hspace{0.5mm}R$. 

By Corollary of Lemma 12 and 13 at p.192 of \cite{kob} we have the local behaviors of 
Hasse matrix and Hasse invariant  
\begin{equation}\label{hasse-m-local}
A(\mathcal{X})\equiv \begin{pmatrix}
t_{11} & t_{12} \\
t_{12} & t_{22}
\end{pmatrix}\ {\rm mod}\ m^2_R,\  H(\mathcal{X})\equiv t_{11}t_{22}-t^2_{12}\ {\rm mod}\  m^3_R.
\end{equation}
In the notation of (\ref{ab}), $B=I_2,\ A=A(\mathcal{X})$, and $C=BA^{-1}=A^{-1}$. 
Here we changed the normalization of $H(\mathcal{X})$ by multiplying $-1$. 

In order to compute the local behavior of $\widetilde{\theta}(f)$ around $X$, we first work on the ordinary locus of $U$ to apply Proposition \ref{first-step} and 
then Proposition \ref{local-beha} and the proof of Proposition \ref{theta} 
show that $\theta(f)=H_{p-1}\cdot\widetilde{\theta}(f)$ is regular at $X$ 
and its local expansion at $X$ is given by 
$$
\sideset{^t}{}{\mathop{%
 \begin{pmatrix} 
(t_{11}t_{22}-t^2_{12})\nabla_{11}(f)-kt_{22}f \\ 
(t_{11}t_{22}-t^2_{12})\nabla_{12}(f)+2kt_{12}f\\
(t_{11}t_{22}-t^2_{12})\nabla_{22}(f)-kt_{11}f 
 \end{pmatrix}}}
\equiv k\alpha
\sideset{^t}{}{\mathop{%
 \begin{pmatrix} 
-t_{22} \\ 
2t_{12}\\
-t_{11} 
 \end{pmatrix}}}
\ {\rm mod}\ m^2_R
$$
where $\nabla_{ij}=\nabla_{D_{ij}}$ for $D_{ij}=\ds\frac{\partial}{\partial t_{ij}}$ with 
$1\le i\le j\le 2$. 
The first claim follows from this. 

As for the second claim, notice that by (\ref{hasse-m-local}) 
$$\det(A(\mathcal{X}))\equiv \nabla_{ij}(A(\mathcal{X}))\equiv 0\ {\rm mod}\ m_R,\ 
1\le i\le j\le 2$$
and 
$$\det
\begin{pmatrix}
\nabla_{11} & \frac{1}{2}\nabla_{12} \\
\frac{1}{2}\nabla_{12} & \nabla_{22}
\end{pmatrix}
(\det(A(\mathcal{X})))\stackrel{\tiny{{\rm mod}\ m_R}}{\equiv} \Big(D_{11}D_{22}-\frac{1}{4}D_{12}D_{12}\Big)
(t_{11}t_{22}-t^2_{12})=1-\frac{1}{2}=\frac{1}{2}. 
$$
By Proposition \ref{step3}, the local expansion of $\Theta(f)$ at $X$ is given by 
$$\Theta(f)(\mathcal{X})\equiv \frac{k(2k-1)\alpha}{9}\ {\rm mod}\ m_R.$$
This proves the second claim. 
\end{proof}

Next we study the vector valued case. 
Let us recall the notation in the proof of Proposition \ref{theta-vector}. 
Consider a local expansion of a non-zero element 
$f=\ds\sum_{i=0}^{r}
f_i\delta_i$ in $GM_{(k_1,k_2)}(\G(N),\bF_p)$ where we put $r=k_1-k_2\ge 0$.  
Assume that $p>r$ (see Remark \ref{remark-decomp} for $r=p-1,p-2$). 
We will consider two cases. 

We first treat the case $r=k_1-k_2=1$. In this case for $f=f_0\delta_0+f_1\delta_1$, we have 
\begin{equation}\label{theta-n=1}
\theta^{(k_1,k_2)}_2(f)=\det(A)\cdot 
\sideset{^t}{}{\mathop{%
 \begin{pmatrix} 
\frac{1}{3}\nabla_{12}(f_0)-\ds\frac{2}{3}\nabla_{11}(f_1)+\frac{2k_2-1}{3}(c_{11}f_1-c_{12}f_0) \\ 
\frac{2}{3}\nabla_{22}(f_0)-\ds\frac{1}{3}\nabla_{12}(f_1)+\frac{2k_2-1}{3}(c_{12}f_1-c_{22}f_0)  
 \end{pmatrix}}}
  \begin{pmatrix} 
\delta_0(\w_1\wedge\w_2) \\ 
\delta_1(\w_1\wedge\w_2)  
 \end{pmatrix}
\end{equation}
provided if $c_{12}=c_{21}$. Then we have 
\begin{thm}\label{non-vanishing-n=1}Assume $k_1-k_2=1$. 
Let  $f=f_0\delta_0+f_1\delta_1$ be as above. Assume that $f$  is not identically zero on  $S_{(0,0)}$.  
Then if $p\not|(2k_2-1)$, $\theta^{(k_1,k_2)}_2(f)$ is not identically zero. 
\end{thm}
\begin{proof}We use the notation in Theorem \ref{non-vanishing1}. 
Then the local expansion at $X\in S_{(0,0)}$ is given by 
$$\sideset{^t}{}{\mathop{%
 \begin{pmatrix} 
\frac{2k_2-1}{3}(t_{22}f_1|_{X}+t_{12}f_0|_{X}) \\ 
\frac{2k_2-1}{3}(-t_{12}f_1|_{X}-t_{11}f_0|_{X})  
 \end{pmatrix}}} \mod m^2_R. $$   
 The claim follows from this since either $f_0|_{X}$ or $f_1|_{X}$ is non-zero. 
\end{proof}

Next we study the case $k_1-k_2>1$, but we will be concerned with $\theta^{(k_1,k_2)}_1$ rather than $\theta^{(k_1,k_2)}_2$. 
By (\ref{delta}) and looking the construction of $\theta^{(k_1,k_2)}_1$ with the explicit presentation in (\ref{decom4}), 
the coefficient $b^{(2)}_{k_1-k_2-2}$ is given by 
\begin{eqnarray}\label{first-coordinate}
b^{(2)}_{k_1-k_2-2}&=&(k_1-k_2)(k_1-k_2-1)\nabla_{11}(f_0)-(k_1-k_2-1)\nabla_{12}(f_1)+2\nabla_{22}(f_2) \nonumber \\
 &&+ (c_{11}k_1-c_{22})(k_1-k_2)(k_1-k_2-1)f_0-c_{12}(k_1-k_2)(k_1-k_2-1)f_1\\
&&-((k_1-k_2-1)c_{11}+(k_2+2)c_{22})f_2 \nonumber 
\end{eqnarray}
provided if $c_{12}=c_{21}$. 
Recall that $\theta^{(k_1,k_2)}_1(f)=\det(A)\cdot \widetilde{\theta}^{(k_1,k_2)}_1(f)$. 
We are ready to discuss the nonvanishing of $F$ under the theta operator $\theta^{(k_1,k_2)}_1$ which decreases $k_1-k_2$ by $-2$. 
\begin{thm}\label{non-vanishing2}Assume $p>r:=k_1-k_2$.
Let  $f$ be as above and assume that $f$  is not identically zero on  $S_{(0,0)}$. 
Then $\theta^{(k_1,k_2)}_1(f)$ is not identically zero.   
\end{thm}
\begin{proof}By assumption we may assume $f|_X\not=0$ for some $X\in S_{(0,0)}$. By using the action of 
$\rho_r(\begin{pmatrix}
1 & a \\
0 & 1
\end{pmatrix}),a\in\bF_p$, we may further assume that $f_1|_X\not=0$ without loss of generality. 
Let $\mathcal{X}$ be the generic square zero deformation of $X$ over ${\rm Spec}\hspace{0.5mm}R$. 
We proceed as in the proof of Theorem \ref{non-vanishing1}. 
Then the local behavior of the coefficient in front of $f_1$ in (\ref{first-coordinate}) is given by 
$$-t_{12}(k_1-k_2)(k_1-k_2-1)$$
which is immediately deduced to be non-zero in $R$ by our assumption. 
\end{proof}
Next we construct a function on superspecial locus and reduce the weight of forms on the locus to be less than 
or equal to $p+1$.   

For $X\in S_{(0,0)}$, let $\mathcal{X}$ be the generic square zero deformation over $R$. 
Let  $m_X$ be the maximal ideal of $\mathcal{O}_{S_{N,p},X}$. We may assume $m_{X}/m^2_X=m_R/m^2_R$. Fix a local basis $\w_1,\w_2$ of $\E$ and let 
 $\w_{i,X}$ be the restriction of $\w_i$ to $X$.
As in \cite{Edix}, for each $X\in S_{(0,0)}$, we consider 
the following local section of $\w^{p-1}|_{\{X\}}\otimes_{\bF_p}{\rm Sym}^2(m_X/m^2_X)=
\w^{p-1}_{\X}|_{\{X\}}\otimes_{\bF_p}{\rm Sym}^2(m_R/m^2_R)$:
$$(\w_{1,X}\wedge\w_{2,X})^{p-1}\otimes \sum_{1\le i\le j\le 2\atop 1\le k\le l\le 2}
\Big(\frac{\partial}{\partial t_{ij}}\frac{\partial}{\partial t_{kl}}H_{p-1}(\mathcal{X}) \Big)\Big|_{X} t_{ij}\otimes t_{kl}.$$
This section is independent of the choice of the local basis $\w_1,\w_2$. 

It is easy to see that there is a canonical factor of ${\rm Sym}^2(m_X/m^2_X)$ which is 
isomorphic to $\w^{2}|_{\{X\}}$ and it is generated by 
$t_{11}\otimes t_{22}-\ds\frac{1}{4}t_{12}\otimes t_{12}. $
We define 
\begin{equation}\label{BX} 
B_X:=\Big\{\Big(\frac{\partial}{\partial t_{11}}\frac{\partial}{\partial t_{22}}-\frac{1}{4}\frac{\partial}{\partial t_{12}}\frac{\partial}{\partial t_{12}}\Big)H_{p-1}(\mathcal{X})\Big\}\Big|_X(\w_{1,X}\wedge \w_{2,X})^{p+1}.
\end{equation}
As mentioned above, this is independent of the choice of any local basis.  
Then  $B_X$ is a non-zero global section of  $\w^{p+1}|_{\{X\}}$ since 
$$\Big\{\Big(\frac{\partial}{\partial t_{11}}\frac{\partial}{\partial t_{22}}-\frac{1}{4}\frac{\partial}{\partial t_{12}}\frac{\partial}{\partial t_{12}}\Big)H_{p-1}(\mathcal{X})\Big\}\Big|_X=
\Big(\frac{\partial}{\partial t_{11}}\frac{\partial}{\partial t_{22}}-\frac{1}{4}\frac{\partial}{\partial t_{12}}\frac{\partial}{\partial t_{12}}\Big)(t_{11}t_{22}-t^2_{12})=1-\frac{1}{2}=\frac{1}{2}.$$ 
Hence the collection $B:=(B_X)_{X\in S_{(0,0)}}$ is an invertible function on $S_{0,0}$ which is of weight $p+1$. 
In the sense of Ghitza (Section 4 of \cite{ghitza1}), $B$ is a superspecial form of weight $p+1$ with level one.   
\begin{prop}\label{B}
The function $B$ induces an isomorphism 
$$SS_{(k_1,k_2)}\stackrel{\times B\atop\sim}{\lra}SS_{(k_1+p+1,k_2+p+1)}.$$
Further, $T(\ell^i)(B\cdot f)=\ell^{2i}B\cdot T(\ell)(f)$ for all $f\in SS_{(k_1,k_2)}$ and prime $\ell\not|pN$. 
\end{prop}  
\begin{proof}Let us freely use the notation of superspecial modular forms and 
Hecke operators on those forms due to 
Ghitza in Section 4 of \cite{ghitza1}. 

Let $m$ be  an element of $\G(N)\backslash \Delta(N)/\G(N)$ such that 
$\nu(m)=\ell^i$ with $i\ge 0$ and recall 
the Hecke operator $T_m$ in (\ref{Hecke-ope-def}). 

Put $\uk=(k_1,k_2)$ for simplicity. Recall the algebraic representation $\lambda_{\uk}=(\lambda_{\uk},V_{\uk})$ of $GL_2$ in Section \ref{class}. 
As explained in Section 4 of \cite{ghitza1},  using the moduli interpretation of 
$S_{N,p}$ (see Section \ref{geometry})
each element $f$ of $SS_{\uk}$ can be regarded as a function 
$$f:\{(A,\lambda,\phi,\eta)\ |\ (A,\lambda,\phi)\in S_{(0,0)},\ 
\eta \text{ is a basis of $H^0(A,\Omega^1_A)$}\}\lra V_{\uk}(\bF_p)$$ 
satisfying 
$f([A,\lambda,\phi,M\eta])=\lambda_{\uk}(M)^{-1}f([A,\lambda,\phi,\eta])$ 
for $M\in GL_2(\bF_p)$ 
where $(A,\lambda,\phi)$ is an object in $S_{(0,0)}$. 
Here we slightly modified Ghitza's formulation since $GU_2(\bF_p)=GL_2(\bF_p)$ in his 
notation. 
A geometric counterpart of $T_m$ is given in Section 3.2.1 of \cite{ghitza1} and it yields 
\begin{equation}\label{hecke-on-ss}
(T_m f)([A,\lambda,\phi,\eta])=\sum_{C\subset A[\ell^i]}
f([A_C,\lambda_C,\phi_C,\eta_C])
\end{equation}
where 
\begin{itemize}
\item $C$ runs over all subgroups in $A[\ell^i]$ of type 
$GSp_4(\Z_\ell)mGSp_4(\Z_\ell)$ 
in the sense of \cite{ghitza1} (see p.352 there). 
The number of such $C$'s is equal to the cardinality of $$GSp_4(\Z_\ell)\backslash GSp_4(\Z_\ell)mGSp_4(\Z_\ell);$$
\item $A_C:=A/C$ which is also a superspecial abelian surface and 
$\lambda_C$ is a principal polarization on $A_C$ such that $\pi^\ast_C \lambda_C=\ell^i \lambda$ where $\pi_C:A\lra A_C$ is the natural projection (see Corollary in p.214 
to Theorem 2 of \cite{Mum} for the existence of $A_C$ and $\lambda_C$);
\item $\phi_C$ is naturally induced from $\phi$ via $\pi$ since $\ell\nmid pN$;
\item the pullback of $\pi_C$ induces an isomorphism 
$M(\pi_C):H^0(A_C,\Omega_{A_C})\lra H^0(A,\Omega_{A})$. Then put 
$\eta_C:=M(\pi_C)^{-1}\eta$. By using Dieudonne theory 
(cf. Section 2.3.3 of \cite{ghitza1}), we see that $\det(M(\pi_C))=\ell^i$.  
\end{itemize}
For each $X=(A,\lambda,\phi)\in S_{(0,0)}$ we fix $\eta_X=\{\w_{1,X},\w_{2,X}\}$ as in (\ref{BX}). 
Since $B_X=\ds\frac{1}{2}(\w_{1,X}\wedge\w_{2,X})^{p+1}$, 
$B([X,\eta_X])=\ds\frac{1}{2}$ as seen before. 
By (\ref{hecke-on-ss}), 
$$T_m(Bf)([X,\eta_X])=\sum_{C\subset A[\ell^i]}
(Bf)([A_C,\lambda_C,\phi_C,\eta_C])=
\sum_{C\subset A[\ell^i]}
B([A_C,\lambda_C,\phi_C,\eta_C])f([A_C,\lambda_C,\phi_C,\eta_C]).$$
Recall $\det(M(\pi_C))=\ell^i$, $\eta_C=\{\w_{1,A_C},\w_{w,A_C}\}:=M(\pi_C)^{-1}\eta$, and 
$B_{(A_C,\lambda_C,\phi_C)}=\frac{1}{2}(\w_{1,A_C}\wedge\w_{2,A_C})^{p+1}$. 
Then it follows that 
$$B([A_C,\lambda_C,\phi_C,\eta_C])=(\det(M(\pi_C)))^{p+1}B([X,\eta_X])=
(\ell^{2i})B([X,\eta_X]).$$ 
The claim follows from this. 
\end{proof}

By Theorem \ref{reduction}, Corollary \ref{cor-of-eigensystem}, and Theorem \ref{non-vanishing2} 
we have obtained the following 
\begin{thm}\label{main1}
For any Hecke eigenform $f\in GM_{(k_1,k_2)}(\G(N),\bF_p)$, there exists a Hecke 
eigenform $G\in GS_{(k'_1,k'_2)}(\G(N),\bF_p),\ 
k'_1-k'_2\in\{0,1\}$ such that 
\begin{enumerate}
\item $G$ is not identically zero on $S_{(0,0)}$, 
\item there exists an integer $m\ge 0$ such that 
$\lambda_f(\ell^i)=\ell^{mi}\lambda_G(\ell^i)$ for any prime $\ell\nmid pN$ and any integer $i\ge 0$. 
Moreover the parity of $k_1-k_2$ is the same as $k'_1-k'_2$. 
\end{enumerate}
\end{thm} 
\begin{proof}
By Corollary \ref{cor-of-eigensystem}, one can find a Hecke eigenform 
$G\in GS_{(k'_1,k'_2)}(\G(N),\bF_p)$ with $k'_1\ge k'_2\ge 3$ which 
is not identically zero on $S_{(0,0)}$.  
Put $\uk'=(k'_1,k'_2)$ and $r'=k'_1-k'_2$. 
Recall the weight of $G$ corresponds to the algebraic representation 
$V_{\uk'}$ of $GL_2(\bF_p)$. Henceforth, we view it as an $SL_2(\bF_p)$-module 
(namely, we ignore the determinant twist). 
Then $V_{\uk'}$ is of highest weight $r'$. If $r'<p$, then by 
Theorem \ref{non-vanishing2} and Remark \ref{remark-decomp} we are done. 
If $r'\ge p$, then $V_{\uk'}$ contains a constituent of form 
$\ds\bigotimes_{i=0}^m\rho^{(p^i)}_{a_i},\ a_i\in\{0,\ldots,p-1\},\ a_m\neq 0$ (up to the determinant part). Notice that $r'=\ds\sum_{i=0}^m a_i p^i$. 
Take a non-zero image of $G$ to the space of global sections of such a constituent and denote it by $G$ again.  
Then we hit $V$ to de-twist each of $\rho^{(p^i)}_{a_i}$'s and then  
the resulting form $V'(G)$ is of weight $\ds\bigotimes_{i=0}^m\rho_{a_i}$ 
where $V'=\ds\bigotimes_{i=0}^m V^{\otimes i}$.  
By using (\ref{expressV}), we see that 
$$r_n(V'(G))=0,\ r_{n+1}(V'(G))\neq 0,\ n:=\sum_{0\le i\le m \atop a_i\neq 0} i$$
since $r_0(G)\neq 0$ (see (\ref{rn}) for the map $r_n$). Obviously, 
$n\le \ds\frac{1}{2}m(m+1)$. 
As in the proof of Theorem \ref{eigensystem}, we have a form $G_1$ from $G$ with 
$G_1|_{S_{(0,0)}}\not\equiv 0$ of 
the weight $\ds\bigotimes_{i=0}^m\rho_{a_i}\otimes {\rm Sym}^n({\rm Sym}^2{\rm St}_2)$ 
whose highest weight is 
$$a_0+\cdots+a_m+2n=r'-(p-1){\rm ord}_p(r'!)+2n\le r'p^{-\lfloor \log_p r' \rfloor}+
(p-1)\lfloor \log_p r' \rfloor+m(m+1)$$
$$\le r'p^{-\lfloor \log_p r' \rfloor}+\lfloor \log_p r' \rfloor(\lfloor \log_p r' \rfloor+p),$$
since the condition $a_m\neq 0$ yields $m\le \lfloor \log_p r' \rfloor$. 
Here we used the formula 
${\rm ord}_p(r'!)=\ds\frac{r'-(a_0+\cdots+a_m)}{p-1}$ and 
$${\rm ord}_p(r'!)=\ds\sum_{i=1}^{\lfloor \log_p r' \rfloor}\Big\lfloor \ds\frac{r'}{p^i}\Big\rfloor \ge \ds\sum_{i=1}^{\lfloor \log_p r' \rfloor}\Big(\frac{r'}{p^i}-1\Big)
=r'\Big(\frac{1-p^{-\lfloor \log_p r' \rfloor}}{p-1}\Big)-\lfloor \log_p r' \rfloor.$$
By direct computation, we can check that 
$$r'p^{-\lfloor \log_p r' \rfloor}+\lfloor \log_p r' \rfloor(\lfloor \log_p r' \rfloor+p)<r'$$
provided if $r'\ge p+3$. When $p\le r'\le p+2$, one can similarly reduce the 
highest weight less than $p$ (this is in fact possible only when $p\ge 5$). 
Hence repeating this procedure,  
one can eventually reduce $r'$ to be less than $p$ as desired. 
Notice that the construction of $G_1$ from $G$ preserves both 
Hecke eigensystems up to mod $p$ cyclotomic twists.

By construction  the highest weights appeared above modulo $2$ are preserved and hence $k_1-k_2$ mod 2 is unchanged under this 
replacement. Assume $k_1-k_2>1$. Then we can apply Theorem \ref{non-vanishing2} to obtain a form of weight $(k'_1,k'_2):=(k_1+p-1,k_2+p+1)$ 
which satisfies $k'_1-k'_2=k_1-k_2-2$. Applying this procedure again if necessarily we obtain the desired form. The result on the 
Hecke eigenvalues follows from Proposition \ref{theta}-(2).   
\end{proof}
We are now ready to prove the first main theorem: 
\begin{thm}\label{main00}$($Weight reduction theorem$)$ Suppose $p\ge 5$. 
For any Hecke eigenform $f\in M_{(k_1,k_2)}(\G(N),\bF_p)$, there exists a Hecke eigenform $G\in S_{(l_1,l_2)}(\G(N),\bF_p)$ 
such that 
\begin{enumerate} 
\item $p>l_1-l_2+3$ and $l_2\le p^4+p^2+2p+1$, 
\item $G$ is not identically zero on $S_{(0,0)}$, 
\item there exists an integer $0\le \alpha\le p-2$ such that 
$\lambda_f(\ell^i)=\ell^{i\cdot \alpha}\lambda_G(\ell^i)$ for any prime $\ell\not|pN$ and any integer $i\ge 0$. 
\end{enumerate}
\end{thm}
\begin{proof}As already observed, one can reduce $l_1-l_2$ as small as possible. However, we may stop the reduction for $l_1-l_2$ once  the condition $p>l_1-l_2+3$ is fulfilled. If $l_2<p^4+p^2+2p$, we have nothing to prove. Otherwise by using 
superspecial modular form $B$ of weight $p+1$, we may assume that there exists a Hecke eigen superspecial form 
$G$ of weight $(l_1,l_2)$ satisfying $p>l_1-l_2+3$ and $1\le l_1\le p+1$. 
Take the smallest positive integer $i$ so that 
$i(p+1)+l_1>p^4+p^2+p$ to apply Corollary \ref{surj}. In fact, we can take $i=p^3-p^2+2p$ if $l_1\not=1$, $i=p^3-p^2+2p+1$ if $l_1=1$. Then $(l_1+i(p+1),l_2+i(p+1))$ is a desired weight 
satisfying $i(p+1)+l_2\le i(p+1)+l_1 \le p^4+p^2+2p+1$.       
\end{proof}
\begin{proof}(A proof of Theorem \ref{main-thm2}) As in the proof of Theorem \ref{main1} and Theorem \ref{main00} we may assume that 
$p-3>k_1-k_2$. By using the operator $\theta^{(k_1,k_2)}_1$ one can reduce 
the difference $k_1-k_2$ with $k_1-k_2-2$. Hence the parity is preserved under this operator. 
One can further reduce $k_1-k_2$ until it becomes $0$ or 1 depending on the parity of $k_1-k_2$. The weight $k_2$ is reduced as 
we have seen in the proof of Theorem \ref{main00}.   
\end{proof}
A reason why we stop the reduction of $l_1-l_2$ is not only to work on the range of the weight 
so that vanishing theorem is applicable, but also 
to compare the local information of the mod $p$ Galois representation associated to $G$ in Theorem \ref{main00} though we do not discuss Galois representations in this paper.

\section{$p$-singular forms and weak $p$-singular forms}
In this subsection we study the kernel of $\theta$ or $\Theta$. 
Throughout this section we assume $p\ge 5$. 
We first introduce some notions.
\begin{dfn}
Let $f$ be an element in $GM_{(k,k)}(\G(N),\bF_p)$ for $k\ge 2$ with the Fourier expansion 
$f=\ds\sum_{T\in \mathcal{S}(\Z)_{\ge 0}}A(T)q^T_N$. 
\begin{enumerate}
\item We say $f$ is a $p$-singular form if $A(T)=0$ unless $p$ divides $T$. 
\item We say $f$ is a weak $p$-singular form if $A(T)=0$  unless $p$ divides $\det(T)$. 
\end{enumerate}
\end{dfn}
Clearly, if $f$ is $p$-singular, then $f$ is weak $p$-singular. However, the converse seems to be subtle. 
The following is a characterization of the $p$-singularity in terms of kernel of $\theta$. 
\begin{thm}Let $f$ be as above. Assume that $p|k$. 
Then $f$ is a $p$-singular form if and only if $\theta(F)=0$.
\end{thm}
\begin{proof}
Put $f=f_0(\w_1\wedge\w_2)^k$ where $f_0$ is a local section of $\mathcal{O}_{S_{N,p}}$. 
Recall the dual basis $D_{ij}$ of $d_{ij}=\langle \w_i,\nabla\w_j \rangle_{{\rm dR}}$. 
By Kodaira-Spencer, $D_{ij}$'s make up a basis of $Der(\O_{S_{N,p}})$. 
Then there exist local coordinates $x_{11},x_{12},x_{22}$ of $S_{N,p}$ so that 
$(\ds\frac{\partial}{\partial x_{11}},\ds\frac{\partial}{\partial x_{12}},
\ds\frac{\partial}{\partial x_{22}})=(D_{11},D_{12},D_{22})M,\ M\in GL_3(\O_{S_{N,p}})$. 

By Proposition \ref{first-step}, the condition $\theta(f)=0$ is equivalent to 
$D_{ij}(F_0)=0$ for any $1\le i\le j\le 2$. This is also equivalent to 
$\nabla(\ds\frac{\partial}{\partial x_{ij}})(f_0)=0,\ 1\le i\le j\le 2$ since $M$ is invertible. Because $S_{N,p}$ is a smooth variety over $\bF_p$, 
this implies $f_0=G^p_0$ for some local section $G_0$ of $\mathcal{O}_{S_{N,p}}$. Put $k=pk'$. 
Then we have $f=(G_0(\w_1\wedge\w_2)^{k'})^p$. By gluing, there exists a $G$ in 
$GM_{(k',k')}(\G(N),\bF_p)$ such that 
$f=G^p$. Then $f$ must have the following $q$-expansion  $f=\ds\sum_{T\in \mathcal{S}(Z)_{\ge 0}}A(T)q^{pT}$. 
This implies $f$ is a $p$-singular form. 
\end{proof}
Obviously if $\theta(f)=0$, then $\Theta(f)=0$. However, so far we do not know if the converse is always true. 

\section{$\theta$-cycles}\label{first-kind-theta}
For a fixed integer $r\ge 0$, put 
$$GM_r(N):=\ds\bigoplus_{k\ge 0}GM_{(r+k,k)}(\G(N),\bF_p),\ GS_r(N):=\ds\bigoplus_{k\ge 0}GS_{(r+k,k)}(\G(N),\bF_p).$$ 

\begin{dfn}\label{single-forms}
For $f\in GM_r(N)$ such that 
$f\in GM_{(r+k,k)}(\G(N),\bF_p)$ for some $k$, we denote by 
${\rm wt}(f):=(r+k,k)$ the 
weight of $f$. Whenever we consider the theta cycles for $f$ we always consider such a form and call it a single form. 
\end{dfn}
If two single forms $f_1,f_2\in GM_{r}(N)$ have the same $q$-expansion with respect to 
the Mumford semi-abelian scheme, then ${\rm wt}(f_i)=(r+k_i,k_i),\ i=1,2$ satisfy $k_1-k_2\equiv 0\ {\rm mod}\ (p-1)$. 
This fact follows from the bigness of the monodromy of the Igusa tower (see the proof of Theorem 1 of \cite{ichikawa} and Theorem 2 of \cite{ichikawa1}). 
In other words, if we write $k_1-k_2=(p-1)m>0$, then we have $f_1=f_2H^{m}_{p-1}$ in $GM_{(r+k_1,k_1)}(\G(N),\bF_p)$. 
The filtration $w(F)$ of a single form $f\in GM_r(N)$ is defined to be the smallest integer $k$ for which 
there exists a single form in $GM_{(r+k,k)}(N,\bF_p)$ whose $q$-expansion is 
the same as one of $f$. 
For any single form $f\in GM_r(N)$ there exists a single form $\widetilde{f}\in M_{r}(N)$ such that ${\rm wt}(\widetilde{f})=w(f)$, 
$\lambda_{\widetilde{f}}(\ell^i)=\lambda_f(\ell^i)$ for any prime 
$\ell\not|pN$ and $i\ge 0$. Further, $\widetilde{f}$ can be not identically 
zero on the non-ordinary locus.  

In this subsection, we set the following convention for two single forms $f_1$ and $f_2$:  
\begin{equation}\label{conv1}
w(f_1)=w(f_2) \mbox{ if the Hecke eigensystems  of $f_1$ and $f_2$ for $\mathbb{T}_N$ are the same as each other.}
\end{equation}
We now study the filtration under the theta operators.
\begin{thm}\label{filtration1}$($scalar valued case$)$ 
Let $f$ be a single form in $GM_0(N)$ such that $\widetilde{f}$ is not identically zero on $S_{(0,0)}$. 
Then 
\begin{enumerate}
\item $w(\Theta(f))\le w(f)+p+1$. If $p{\not|}k(2k-1)$, then the equality holds and 
$\Theta(f)$ is non-zero on $S_{(0,0)}$.  

\item $w(\Theta^{\frac{p+1}{2}}(f))=w(\Theta(f))$. 

\item If $p{\not|}k$, then $w(\theta(f))=w(f)+p-1$. 
\end{enumerate}
\end{thm}
\begin{proof}The inequality follows by definition. 
By the proof of Theorem \ref{non-vanishing1}, $\Theta(f)$ and $\theta(f)$ are both non-zero at some point in $S_{(0,0)}$ under the assumption on $k$. 
Further, the equality for each of $\Theta$ and $\theta$ holds since the Hasse invariant is identically zero on 
$S_{(0,0)}$. 

The second claim is a consequence of Proposition \ref{theta}-(1)-(a) with the convention (\ref{conv1}).  
\end{proof} 
\begin{rmk}The second claim of Theorem \ref{filtration1} looks strange but 
it is natural as $\br_{f,p}\sim \br_{\Theta^{\frac{p+1}{2}}(f),p}$ 
which follows from Proposition \ref{theta}-$($1$)$-$($a$)$ and hence, it satisfies 
$($\ref{conv1}$)$. 
\end{rmk}

As in \cite{joch}, \cite{Edix} (the idea was due to J. Tate for elliptic modular forms), 
but we use a slight modification, the theta cycle of the scalar valued single form $F$ is defined by 
$${\rm Cyc}(f):=(w(\Theta(f)),w(\Theta^2(f)),\ldots,w(\Theta^{\frac{p-1}{2}}(f))).$$ 
Since $w(\Theta(f))=w(\Theta^{\frac{p+1}{2}}(f))$ under our convention (\ref{conv1}), actually it makes up a ``cycle".  
\begin{dfn}\label{low-pt1}
Let $f$ be a single form in $GM_0(N)$ with $k\ge 2$. 
\begin{enumerate}
\item We say $\Theta^i(f)$ is a low point of the first type (resp. the second type) if 
 $w(\Theta^{i-1}(f))\equiv 0\ {\rm mod}\ p$ (resp. $(2w(\Theta^{i-1}(f))-1)\equiv 0\ {\rm mod}\ p$). 
If  $f_i:=\Theta^{c_i}(f)$ is a low point  for some integer $c_i>0$, then the number $c_i-1$ means one of times we add $(p+1)$ to 
$w(f)$. We say $c_i$ the low number of the low point $\Theta^{c_i}(f)$. We say $c_i$ the low number for $f_i$. 
We write $c_i=c^{(1)}_i$ (resp. $c_i=c^{(2)}_i$) if 
the low point is of the first type (resp. the second type). 

\item We define the number $b_i$ so that $b_i(p-1)=w(\Theta^{c_i-1}f)+(p+1)-w(\Theta^{c_i}F)$ 
which means the amount falling the filtration at the low point $f_i$ with the next application of $\Theta$. 
We say $b_i$ the jumping number of the low point $\Theta^{c_i}(f)$.  
As to $c_i$, we also write $b_i=b^{(1)}_i$ or $b_i=b^{(2)}_i$ according to the first type or the second type respectively. 
\end{enumerate}
\end{dfn}

Let $f$ be a Hecke eigenform in $M_0(N)$ which is also a single form. 
Let $\lambda_f(p^i)$ be the Hecke eigen-value of $f$ for $T(p^i)$. 
\subsubsection{non semi-ordinary case}
We first assume that $\lambda_f(p^i)=0$ for all $i\ge 0$. 
This is equivalent to $\lambda_f(p^i)=0$ for $i=1,2$ by 
Proposition 3.3.35 at p.165 of \cite{and}. We say $f$ is non semi-ordinary at $p$ when 
$\lambda_f(p)=\lambda_f(p^2)=0$. 
Otherwise we say $f$ is semi-ordinary. 
By convention and definition, we have $w(\Theta^{\frac{p-1}{2}}(f))=w(f)$. 
Let $\{c_i\}_{i=1}^r$ (resp. $\{b_i\}_{i=1}^r$) be the collection of all low numbers (jumping numbers) for $f$. 
We define $f^{(m_i)}_i,1\le i\le r$ and  $m_i\in\{1,2\}$ inductively so that  
$f^{(m_{i+1})}_{i+1}=\Theta^{c^{(m_{i+1})}_i}(f^{(m_i)}_i)$, $m_1=1$, and $f^{m_1}_1=f$. 
Since the length of the theta cycle of $f$ is $\ds\frac{p-1}{2}$, one has 
$\ds\sum_ic_i=\ds\frac{p-1}{2}$. 
The total amount of the varying weights in the theta cycle is $(p+1)\ds\frac{(p-1)}{2}$. 
It follows from this that $\ds\sum_ib_i(p-1)=(p+1)\frac{(p-1)}{2}.$ Hence we have 
$\ds\sum_ib_i=\frac{p+1}{2}$. 

On the other hand, $w(\Theta^{c^{(1)}_i-1}f^{(m_i)}_i)(2w(\Theta^{c^{(2)}_i-1}f^{(m_i)}_i)-1)\equiv 0\ {\rm mod}\ p$ for all $i$. From this we have 
$$
\begin{array}{l}
-b^{(1)}_i\equiv 1-w(f^{(1)}_{i+1})\ {\rm mod}\ p,\ {\rm or} \\
-2b^{(2)}_i\equiv 3-2w(f^{(2)}_{i+1})\ {\rm mod}\ p. 
\end{array}
$$
By definition we have $w(\Theta^{c^{(j)}_{i+1}-1}f^{(j')}_{i+1})=(c^{(j)}_{i+1}-1)(p+1)+w(f^{(j')}_{i+1}),\ 
j,j'\in\{1,2\}$. 
Then we also have 
$$
\begin{array}{l}
0\equiv (c^{(1)}_{i+1}-1)+w(f^{(j')}_{i+1}) \ {\rm mod}\ p,\ {\rm or} \\
1\equiv 2(c^{(2)}_{i+1}-1)+2w(f^{(j')}_{i+1}) \ {\rm mod}\ p. 
\end{array}
$$
for $j'\in\{1,2\}$. 

Putting these together, we have four cases for the consecutive low points and jumping numbers
\begin{equation}\label{low-jump1}
\begin{array}{lll}
{\rm Case\ 1} & b^{(1)}_i+c^{(1)}_{i+1}\equiv 0               & {\rm mod}\ p, \\
{\rm Case\ 2} & b^{(1)}_i+c^{(2)}_{i+1}\equiv \ds\frac{p+3}{2}& {\rm mod}\ p, \\
{\rm Case\ 3} & b^{(2)}_i+c^{(1)}_{i+1}\equiv \ds\frac{p-1}{2}& {\rm mod}\ p, \\
{\rm Case\ 4} & b^{(2)}_i+c^{(2)}_{i+1}\equiv 0               & {\rm mod}\ p. 
\end{array}
\end{equation}
Put $b_0=c_{r+1}=0$.  
Since $b_r+c_1+\ds\sum_{i=1}^{r-1}(b_i+c_{i+1})=\sum_{i}(b_i+c_i)=p$, we have the following for each case in (\ref{low-jump1}):
\begin{equation}\label{low-jump1}
\begin{array}{ll}
{\rm Case\ 1} & r=1\ {\rm and}\ c^{(1)}_1=\ds\frac{p-1}{2},\ b^{(1)}_1=\ds\frac{p+1}{2}, \\
{\rm Case\ 2} & r=2\ {\rm and}\ c^{(1)}_1+b^{(2)}_2=\ds\frac{p-3}{2},\ c^{(2)}_2+b^{(1)}_1=\ds\frac{p+3}{2},\ p\not=5, \\
{\rm Case\ 3} & r=2\ {\rm and}\ c^{(1)}_2+b^{(2)}_1=\ds\frac{p-1}{2},\ c^{(2)}_1+b^{(1)}_2=\ds\frac{p+1}{2}, \\
{\rm Case\ 4} & r=1\ {\rm and}\ c^{(2)}_1=\ds\frac{p-1}{2},\ b^{(2)}_1=\ds\frac{p+1}{2}. 
\end{array}
\end{equation}
In Case 2, since  $c^{(1)}_1+b^{(2)}_2\ge 2$, this forces $p$ to be greater than or equal to $7$.

From now on we further assume that $\widetilde{f}$ is not identically zero on $S_{(0,0)}$. 
Put $w(f)=ap+a',\ 1\le a'\le p,\ a\in \Z_{\ge 0}$.  

\medskip
\noindent
Case 1. 

Then $w(\Theta^{c^{(1)}_1-1}f)=w(\Theta^{\frac{p-3}{2}}f)=ap+a'+\ds\frac{(p-3)}{2}(p+1)$ by Theorem \ref{filtration1}. 
The condition $w(\Theta^{c^{(1)}_1-1}f)\equiv 0\ {\rm mod}\ p$ forces that $a'\equiv \ds\frac{(p+3)}{2}\ {\rm mod}\ p$. 
Note that if $a'= \ds\frac{(p+3)}{2}$, then 
$w(\theta^{c^{(1)}_1}f)=w(f)+\ds\frac{(p-3)}{2}(p+1)+(p+1)-b^{(1)}_1(p-1)=w(f)$ which never contradicts with our setting.  
Summing up we have proved the following: 
$${\rm Cyc}(f)=
(k+(p+1),k+2(p+1),\ldots,k+\frac{(p-3)}{2}(p+1),k)
$$
provided if $ w(f)\equiv \ds\frac{(p+3)}{2}\ {\rm mod}\ p$

\medskip
\noindent 
Case 4. 

This case is completely similar to Case 1. So we omit the details. We have 
$${\rm Cyc}(f)=
(k+(p+1),k+2(p+1),\ldots,k+\frac{(p-3)}{2}(p+1),k)
$$
provided if $w(F)\equiv 2\ {\rm mod}\ p.$

\medskip
\noindent
Case 2. 

Since $w(\theta^{c^{(1)}_1}(f))=ap+a'+(c^{(1)}_1-1)(p+1)$ has to be divisible by $p$, 
we have $a'+(c^{(1)}_1-1)\equiv 0$ mod $p$. However, $1\le a'+(c^{(1)}_1-1)\le p+\frac{p-5}{2}<2p$ ( note that $p>5$ in this case). 
Hence we have  $a'+(c^{(1)}_1-1)=p$. Then we have 
$$c^{(1)}_1=p+1-a',\ c^{(2)}_2=a'-\frac{p+3}{2},\ b^{(1)}_1=p+3-a',\ c^{(2)}_2=a'-\frac{p+5}{2}.$$
The theta cycle is computed as follows:
$$w(\Theta^{c^{(1)}_1-1}(f))=ap-a'p+p(p+1),\ w(\Theta^{c^{(1)}_1}(f))=w(\Theta^{c^{(1)}_1-1}f)-b^{(1)}_1(p-1)=ap+4-a',$$
and for $f^{(1)}_1=\Theta^{c^{(1)}_1}(f)$, 
$$w(\Theta^{c^{(2)}_2-1}(f^{(1)}_1)=ap+4-a'+(c^{(2)}_2-1)(p+1)=ap+a'p-\frac{1}{2}(p^2+6p-3).$$ 
However, this contradicts with $2w(\Theta^{c^{(2)}_2-1}(f^{(1)}_1)-1\equiv 0\ {\rm mod}\ p$. Therefore ,this case does not occur. 

\medskip
\noindent
Case 3. 

We proceed as in Case 2. Since $2w(\theta^{c^{(2)}_1}(f))-1=2ap+2a'+2(c^{(2)}_1-1)(p+1)-1$ is divisible by $p$, 
we must have $a'+c^{(2)}_1=\ds\frac{p+3}{2}$ or $\ds\frac{3p+3}{2}$. The latter case does not occur since 
$c^{(2)}_1+c^{(1)}_2=\ds\frac{p-1}{2}$. Therefore, we have 
$$c^{(2)}_1=\frac{p+3}{2}-a',\ c^{(1)}_2=a'-2,\ b^{(2)}_1=\frac{p+3}{2}-a',\ c^{(2)}_2=a'-1.$$
This forces $a'$ to satisfy 
$$3\le a'\le \frac{p+1}{2}.$$
The filtrations are computed as follows:
$$w(\Theta^{c^{(2)}_1-1}(f))=ap-a'p+\frac{1}{2}(p+1)^2,\ 
w(\Theta^{c^{(2)}_1}(f))=w(\Theta^{c^{(2)}_1-1}f)-b^{(2)}_1(p-1)=ap+p+3-a',$$
and for $f^{(2)}_1=\Theta^{c^{(2)}_1}(f)$, 
$$w(\Theta^{c^{(1)}_2-1}(f^{(2)}_1)=ap+4-a'+(c^{(1)}_2-1)(p+1)=ap+a'p-2p.$$
Note that  $2w(\Theta^{c^{(2)}_1-1}(f))-1\equiv 0$ mod $p$ and 
$w(\Theta^{c^{(1)}_2-1}(f^{(2)}_1)\equiv 0\ {\rm mod}\ p$.
So there is no contradiction here. 
Put $k_0=a'$. 
Then we have 
$$\begin{array}{rl}
{\rm Cyc}(f)=&(k+(p+1),k+2(p+1),\ldots,k+(\frac{p+1}{2}-k_0)(p+1), \\
&k_1, k_1+p+1,\ldots, k_1+(k_0-3)(p+1)),\ k_1=k+p+3-2k_0
\end{array}
$$
provided if $3\le k_0\le \ds\frac{p+1}{2}.$

Summing up we have proved 
\begin{thm}\label{non-semiord}Let $f$ be a single form in $GM_0(N)$ with $wt(f)\ge 2$ and $\widetilde{f}$ is not identically 
zero on $S_{(0,0)}$. Assume $f$ is non semi-ordinary. Put $w(f)=ap+k_0,\ 1\le k_0\le p$ and 
$k_1=k+p+3-2k_0$. 
Then 
$${\rm Cyc}(f)=\left\{\begin{array}{ll}
(k+(p+1),k+2(p+1),\ldots,k+\frac{(p-3)}{2}(p+1),k) &\ {\rm if}\ w(f)\equiv 2\ {\rm mod}\ p\\ 
(k+(p+1),k+2(p+1),\ldots,k+\frac{(p-3)}{2}(p+1),k) &\ {\rm if}\ w(f)\equiv \ds\frac{p+3}{2}\ {\rm mod}\ p\\ 
\begin{array}{l}(k+(p+1),k+2(p+1),\ldots,k+(\frac{p+1}{2}-k_0)(p+1), \\
k_1, k_1+(p+1),\ldots, k_1+(k_0-3)(p+1),k) \end{array} & \ {\rm if}\ 3\le k_0\le \ds\frac{p+1}{2}\\
 (k+(p+1),k+2(p+1),\ldots,k+\frac{(p-3)}{2}(p+1),k)   &\ {\rm otherwise}. \\
\end{array}\right.
$$
\end{thm}
Notice that the last case is just due to the convention (\ref{conv1}). 

\subsubsection{Semi-ordinary case}Let $f$ be a single form in $GM_0(N)$ with $wt(f)\ge 2$ and $\widetilde{f}$ is not identically 
zero on $S_{(0,0)}$. From now we assume $f$ is semi-ordinary, hence $\lambda_f(p)\not=0$ or $\lambda_f(p^2)\not=0$. 
Assume that $k=w(f)$ satisfies both of $k\not\equiv 0 \ {\rm mod}\ p$ and $k\not\equiv \ds\frac{p+1}{2} \ {\rm mod}\ p$. Then $w(\Theta(f))=k+p+1$. Put $k=ap+k_0,\ 1\le k_0\le p$ and $k'_1=k+p+1-2k_0$. 
Applying $G=\Theta(f)$ to Theorem \ref{non-semiord}, we have 
$${\rm Cyc}(f)=\left\{\begin{array}{ll}
(k+(p+1),k+2(p+1),\ldots,k+\frac{(p-1)}{2}(p+1)) &\ {\rm if}\ w(f)\equiv 1\ {\rm mod}\ p\\ 
(k+(p+1),k+2(p+1),\ldots,k+\frac{(p-1)}{2}(p+1)) &\ {\rm if}\ w(f)\equiv \ds\frac{p+1}{2}\ {\rm mod}\ p\\ 
\begin{array}{l}(k+(p+1),k+2(p+1),\ldots,k+(\frac{p+1}{2}-k_0)(p+1), \\
k'_1+(p+1), k'_1+2(p+1),\ldots, k'_1+(k_0-1)(p+1),k) \end{array} & \ {\rm if}\ 2\le k_0\le \ds\frac{p-1}{2}\\
 (k+(p+1),k+2(p+1),\ldots,k+\frac{(p-1)}{2}(p+1))   &\ {\rm otherwise}. \\
\end{array}\right.
$$
The remaining cases are $p|k$ or $p|(2k-1)$. We treat only $k=p$ or $k=\ds\frac{p+1}{2}$. 
When $k=p$, the possible values of $w(\Theta(f))=p+(p+1)-b_1(p-1), b_1=0,1,2,\ldots$ are $2p+1,p+2,3$. 
Then by Theorem \ref{non-semiord} we have 
 $${\rm Cyc}(f)=\left\{\begin{array}{ll}
(2p+1,2p+1+(p+1),\ldots,2p+1+\frac{(p-3)}{2}(p+1)) &\ {\rm if}\ w(f)=2p+1\\ 
(p+2,p+2+(p+1),\ldots,p+2+\frac{(p-3)}{2}(p+1))  &\ {\rm if}\ w(f)=p+2\\ 
(3,3+(p+1),\ldots,3+(\frac{p-5}{2})(p+1), p)   &\ {\rm if}\ w(f)=3. 
\end{array}\right.
$$
Similarly when $k=\ds\frac{p+1}{2}$, the possible values of $w(\Theta(f))=\ds\frac{p+1}{2}+(p+1)-b_1(p-1), 
b_1=0,1,2,\ldots$ are $\ds\frac{3p+3}{2},\ds\frac{p+5}{2}$ if $p>7$ and $9,5,1$ if $p=5$. 
Then by Theorem \ref{non-semiord} we have 
 $${\rm Cyc}(f)=\left\{\begin{array}{ll}
(\ds\frac{3p+3}{2},\ds\frac{3p+3}{2}+(p+1),\ldots,\ds\frac{3p+3}{2}+\frac{(p-3)}{2}(p+1)) &\ {\rm if}\ w(f)=\ds\frac{3p+3}{2}\\ 
(\ds\frac{p+5}{2},\ds\frac{p+5}{2}+(p+1),\ldots,\ds\frac{p+5}{2}+\frac{(p-3)}{2}(p+1))  &\ {\rm if}\ w(f)=\ds\frac{p+5}{2}\\
(1,1+(p+1),\ldots,1+\frac{(p-3)}{2}(p+1))  &\ {\rm if\ }p=5\ {\rm and}\ w(f)=1.
\end{array}\right.
$$
In all lists of theta cycles, we have not used semi-ordinary-ness. As in Proposition 3.3 of 
\cite{Edix}, some of cycles might not occur according to whether $f$ is $p$-singular or not. 

\subsubsection{vector valued case}
In the following  we study the theta cycles for vector valued forms $f$ in $GM_{1}(N)$. 
Assume $f$ is of weight ${\rm wt}(F)=(k+1,k)$ and it is a Hecke eigenform. 
In this case we make use of the theta operator $\theta^{k+1,k}_2$. We drop the superscript and then simply denote it by 
$\theta_2$.  We denote by $w_2(f)$ the second component of the filtration $w(f)$. 
The following theta cycle looks like one in \cite{Edix} and different from our scalar valued case: 
$${\rm Cyc}(f):=(w_2(\theta_2(f)),w_2(\theta^2_2(f)),\ldots,w_2(\theta^{p-1}_2(f))).$$   
We still keep the convention (\ref{conv1}). 
Let $f(q)=\sum_{T}A_f(T)q^T_N,\ A_f(T)\in St_2(\bF_p)$ be  the Fourier expansion of $f$. 
By Theorem \ref{non-vanishing-n=1}, we see that for any integer $m\ge 0$, 
$$\theta^{4m}_2(f)(q)=\sum_T\Big(\frac{\det(T)}{18N^2}\Big)^{2m}A_f(T)q^T_N.$$
Notice that  $w_2(\theta_2(f))=w_2(\theta^{p}_2(f))$ under our convention (\ref{conv1}) and Proposition \ref{theta-vector}-(2). 

As in the scalar values case, we use a similar notation for each of 
low point, low number, and jumping number in the vector valued case. 
Note that low point will happen only when $p|(2k-1)$. 

We first assume that $F$ is non semi-ordinary. Then by convention and definition, we have $w_2(\theta^{p-1}_2(f))=w_2(f)$. 
By (the proof of) Theorem \ref{non-vanishing-n=1}, if $p{\not|}(2k-1)$,  
$w_2(\theta_2(F))=k+p$. Since $p\nmid 2w_2(f)-1$,  by using the same argument 
inductively, we have 
$${\rm Cyc}(f)=(k+p,k+2p,\ldots,k+(p-2)p,k).$$
Next we assume $2k-1\equiv 0$ mod $p$. 
Let $\{c_i\}_{i=1}^r$ (resp. $\{b_i\}_{i=1}^r$) be the collection of all low numbers (jumping numbers) for $f$. 
Note that in this case $b_i$ is given by 
\begin{equation}\label{bi's}
b_i(p-1)=w_2(\theta^{c_i-1}_2f)+p-w_2(\theta^{c_i}_2f).
\end{equation}
We define $f_{i+1}=\theta^{c_i}_2(f_i),\ f_1=f$ inductively. 
Since the length of each theta cycle is $p-1$, one has 
$\ds\sum_i c_i=p-1$. 
The total amount of the varying weights (with respect to $k$) in the theta cycle of $f$ is $p(p-1)$. 
It follows from this that $\ds\sum_ib_i(p-1)=p(p-1).$ Hence we have 
$\ds\sum_ib_i=p$. 

On the other hand, $2w_2(\theta^{c_i-1}_2(f_i))-1\equiv 0\ {\rm mod}\ p$ for all $i$. From this we have 
$$-2b_i\equiv 1-2w_2(f_{i+1})\ {\rm mod}\ p.$$ 
By definition we have $w(\theta^{c_{i+1}-1}_2(f_{i+1}))=(c_{i+1}-1)p+w(f_{i+1})$.  
Then we also have 
$$1\equiv 2w_2(f_{i+1}) \ {\rm mod}\ p.$$
Putting everything together we have $c_{i+1}\equiv \ds\frac{p+1}{2}\ {\rm mod}\ p$ and 
$b_{i}\equiv 0\ {\rm mod}\ p$. This yields that $r=1,\ b_1=p$, and $c_1=p-1$. 
Therefore, $w_2(\theta^i_2(f))=k+ip$ for $1\le i<c_1=p-1$. When $c_1=p-1$ 
the jump occurs and by (\ref{bi's}) we have 
$$w_2(\theta^{c_1}_2(f))=w_2(\theta^{c_1-1}_2(f))+p-b_1(p-1)
=k+(p-2)p+p-p(p-1)=k.$$
Summing up, we have 
\begin{thm}\label{vec-ord}Let $f$ be a single form in $GM_1(N)$ and $\widetilde{f}$ is not identically 
zero on $S_{(0,0)}$. Assume $f$ is non semi-ordinary. Then 
$${\rm Cyc}(f)=(k+p,k+2p,\ldots,k+(p-2)p,k).  
$$
\end{thm}
The semi-ordinary case is similar. Hence we have 
\begin{thm}\label{vec-non-ord}
Let $f$ be a single form in $GM_1(N)$ with $k:=w_2(f)$ and $\widetilde{f}$ is not identically 
zero on $S_{(0,0)}$. Assume $f$ is semi-ordinary. 
Put $k'=w_2(\theta_2(f))$. 
Then 
$${\rm Cyc}(f)=(k'+p,k'+2p,\ldots,k'+(p-1)p). 
$$
Further if $k'=k+p$ if $p\not|(2k-1)$. 
\end{thm}
In forthcoming paper \cite{yam} the author will study some theta cycles for 
Siegel modular forms of general weights $(k_1,k_2)$ and it will be proved that the theta cycles include non-trivial jumping points only when $k_1=k_2$. 
 

\section{Appendix A}
In this section we give an explicit form of Pieri's decomposition. 
Let $R$ be an $\bF_p$-algebra. Put ${\rm St}_2(R)=Re_1\oplus Re_2$ and let $GL_2(R)$ acts on $St_2$ by 
$$ge_1=ae_1+c e_2,\ ge_2=be_1+de_2,\ g=\begin{pmatrix} a & b \\ c& d \end{pmatrix}.$$
For a positive integer $n$, 
let $V(n)={\rm Sym}^n {\rm St}_2(R)$ be the $n$-th symmetric representation of $GL_2(R)$. 
Put $V(n,m)=V(n)\otimes_R \det^n(St_2(R))$ for an integer $m\ge 0$. 

We are concerned with an explicit decomposition of $V(n)\otimes_R V(2)$. 
Consider the basis $u_i=e^i_1e^{n-i}_2,\ i=0,\ldots,n$ (resp. $v_2=e^2_1,\ v_1=e_1e_2,\ v_0=e^2_2$) of $V(n)$ 
(resp. $V(2)$). 
We define the operators $E,F$ on $V(n)$ by 
$Eu_i=iu_{i-1},\ Fu_i=(n-i)u_{i+1}$. 
We also define the same operators on $V(n)\otimes_R V(n')$ by Leibniz rule 
$E(u_i\otimes u_{j})=Eu_i\otimes u_j+u_i\otimes Eu_j,\ F(u_i\otimes u_{j})=Fu_i\otimes u_j+u_i\otimes Fu_j$.  
 
We first assume that $n+3\le p$.  
Put 
$$w_0=u_n\otimes v_2,\ w_1=u_n\otimes v_1-u_{n-1}\otimes v_2,\ w_2=u_n\otimes v_0-2u_{n-1}\otimes v_1+u_{n-2}\otimes v_2$$ Since $Fw_i=0$ for $i=0,1,2$, we expect that these vectors would be highest weight vectors. 
Put 
$$W_0(n)=\langle f^{(0)}_i:=\frac{1}{(n+2)_i}E^iw_0,\ 0\le i\le n+2   \rangle_R,\ 
W_1(n)=\langle f^{(1)}_i:=\frac{1}{(n)_i}E^iw_1,\ 0\le i\le n   \rangle_R,$$
and $$W_2(n)=\langle f^{(2)}_i:=\frac{1}{(n-2)_i}E^iw_2,\ 0\le i\le n-2   \rangle_R$$ 
where $(x)_i=x!/(x-i)!$  and we set $(\ast)_0:=1$. 

Recall $V(n,m)=V(n)\otimes_R \det^m$ and denote by $\{u_i\}_i$ its basis  by abuse of notation. 
By direct computation, we have the following: 
\begin{prop} For $j=0,1,2$, as $GL_2(R)$-modules, 
$$V(n+2-2j,j)\stackrel{\sim}{\lra} W_j(n),\ u_i\mapsto f^{(j)}_i.$$
\end{prop} 
When $n=p-1$ or $p-2$, the denominators of the coefficients appearing in $f^{(2)}_i:=\frac{1}{(n-2)_i}E^iw_2$ are not 
divisible by $p$. Hence only $W_2(n)$ still makes sense and so does 
the isomorphism $V(n-2,2)\stackrel{\sim}{\lra} W_2(n)$. 
Clearly  $W_2(n)$ gives a splitting of the surjection $V(n)\otimes_R V(2)\lra V(n-2,2)$. 

To end this section we give an explicit realization of an isomorphism $V(n)\otimes_R V(2)\simeq W_2(n)\oplus W_1(n)\oplus W_0(n)$ in 
terms of our basis. Note that $\theta^{\underline{k}}_1$ is related to $W_2(n)$. 
We identify $v=\ds\sum_{0\le i\le n \atop 0\le j\le 2}a_{ij}u_i\otimes v_j\in V(n)\otimes_R V(2)$ with the low vector
$$(a_{n2},a_{n1},a_{n0},a_{(n-1)2},a_{(n-1)1},a_{(n-1)0},\ldots,a_{12},a_{11},a_{10},a_{02},a_{01},a_{00}).$$ 
Let us first assume that $n<p-2$. 
If we write 
$$v=\sum_{i=0}^{n+2}b^{(0)}_if^{(0)}_i+\sum_{i=0}^{n}b^{(1)}_if^{(1)}_i+\sum_{i=0}^{n-2}b^{(2)}_if^{(2)}_i,$$
then we have 
\begin{equation}\label{decom1}
{}^t(b^{(0)}_i)_{0\le i\le n+2}=
\left(\begin{array}{c}
a_{n2} \hfill \\
a_{(n-1)2}+a_{n1} \hfill\\
 a_{(n-2)2}+a_{(n-1)1}+a_{n0} \hfill \\
 a_{(n-3)2}+a_{(n-2)1}+a_{(n-1)0} \hfill \\
  \vdots \\
  a_{12}+\hspace{3.5mm}a_{21}\hspace{3.5mm}+a_{30} \\
   a_{02}+\hspace{3.5mm}a_{11}\hspace{3.5mm}+a_{20} \\
\hspace{13.5mm} a_{01}\hspace{3.5mm}+a_{10} \\ 
\hspace{27mm}a_{00}  
\end{array}
\right)
\end{equation}
where the superscript $``t"$ stands for the transpose. 

As for $b^{(1)}_i$ we have 
\begin{equation}\label{decom2}
b^{(1)}_i=\left\{\begin{array}{lc}
-\ds\frac{2}{n+2}a_{(n-1)2}\hspace{5mm}+\frac{n}{n+2}a_{n1}& (i=0) \\
 & \\
-\ds\frac{2+2i}{n+2}a_{(n-1-i)2}+\frac{n-2i}{n+2}a_{(n-i)1}+\frac{2n+2-2i}{n+2}a_{(n+1-i)0},\ 
&(i=1,\ldots,n-1) \\
& \\
\hspace{31mm}-\ds\frac{n}{n+2}a_{01}\hspace{7mm}+\hspace{10mm}\frac{2}{n+2}a_{10} & (i=n)
\end{array}
\right.
\end{equation}
For the remaining coefficients $b^{(2)}_i$, we have 
\begin{equation}\label{decom3}
b^{(2)}_i=\frac{(i+1)(i+2)}{n(n+1)}a_{(n-2-i)2}-
\frac{(i+1)(n-i-1)}{n(n+1)}a_{(n-1-i)1}+
\frac{(n-i-1)(n-i)}{n(n+1)}a_{(n-i)0}
\end{equation}
for $i=0,\ldots, n-2$. 
When $n=p-1$ or $p-2$ 
we have a splitting projection $V(n)\otimes_R V(2)\lra W_2(n)$ which is given by 
replacing $b^{(2)}_i$ in (\ref{decom3}) with 
\begin{equation}\label{decom4}
b^{(2)}_i=(i+1)(i+2)a_{(n-2-i)2}-
(i+1)(n-i-1)a_{(n-1-i)1}+
(n-i-1)(n-i)a_{(n-i)0}
\end{equation} 
obtained by multiplying $n(n+1)$. Notice that $n(n+1)$ is zero if $n=p-1$ but we can justify 
by working over $\Z[1/(n+2)!]$ at first and then by multiplying $n(n+1)$. 
This yields (\ref{decom4}) from (\ref{decom3}).

\end{document}